\documentclass[twoside,12pt]{amsart}
\usepackage{amssymb}
\usepackage{amsthm}
\usepackage{amsfonts}
\usepackage{amsmath}
\usepackage[T1]{fontenc}
\usepackage{lscape}
\usepackage{url}
\usepackage{longtable}
\usepackage{graphicx}
\usepackage{latexsym}
\usepackage{colortbl}
\usepackage[all]{xy}

\newcommand{\Aut}{\mbox{\rm Aut}}
\newcommand{\Inn}{\mbox{\rm Inn}}

\newcommand{\Res}{\mbox{\rm Res}}
\newcommand{\Irr}{\text{\rm Irr}}

\newcommand{\Stab}{\text{\rm Stab}}
\newcommand{\diag}{\text{\rm diag}}
\newcommand{\Inndiag}{\text{\rm Inndiag}}
\newcommand{\GL}{\text{\rm GL}}
\newcommand{\GU}{\text{\rm GU}}
\newcommand{\SL}{\text{\rm SL}}
\newcommand{\PGL}{\text{\rm PGL}}
\newcommand{\PGU}{\text{\rm PGU}}
\newcommand{\PSL}{\text{\rm PSL}}
\newcommand{\SU}{\text{\rm SU}}
\newcommand{\PSU}{\text{\rm PSU}}
\newcommand{\SO}{\text{\rm SO}}
\newcommand{\Sp}{\text{\rm Sp}}
\newcommand{\Spin}{\text{\rm Spin}}
\newcommand{\PCSO}{\text{\rm PCSO}}

\renewcommand{\P}{\text{\rm P}}

\newcommand{\id}{\mbox{\rm id}}
\newcommand{\Id}{\mbox{\rm Id}}
\newcommand{\ad}{\mbox{\rm ad}}

\newcommand{\ex}{\diamond}

\newtheorem{num}{Notation}[section]
\newtheorem{defn}[num]{Definition}

\newtheorem{thm*}{Theorem}
\newtheorem{lem}[num]{Lemma}
\newtheorem{prp}[num]{Proposition}

\newtheorem{rem}[num]{Remark}
\newtheorem{hypo}[num]{Hypothesis}

\begin{document}

\title{The eigenvalue one property of finite groups, II}
\author{Gerhard Hiss}
\author{Rafa{\l} Lutowski}

\address{Gerhard Hiss:
Lehrstuhl f{\"u}r Algebra und Zahlentheorie,
RWTH Aachen University,
52056 Aachen, Germany}
\email{gerhard.hiss@math.rwth-aachen.de}

\address{Rafa{\l} Lutowski:
Institute of Mathematics, Physics and Informatics,
University of Gda\'nsk,
ul. Wita Stwosza 57,
80-308 Gda\'nsk, Poland}
\email{rafal.lutowski@ug.edu.pl}



\begin{abstract}
We prove a conjecture of Dekimpe, De Rock and Penninckx concerning the 
existence of eigenvalues one in certain elements of finite groups acting 
irreducibly on a real vector space of odd dimension. This yields a sufficient 
condition for a closed flat manifold to be an $R_{\infty}$-manifold.
\end{abstract}

\date{\today}

\maketitle

\section{Introduction}

The purpose of this article is to complete the proof of 
\cite[Theorem~$1.1.5$]{HL1}, thereby confirming a conjecture of Dekimpe, De Rock 
and Penninckx \cite[Conjecture~$4.7$]{DDRP}. More precisely, we establish the 
following result.

\bigskip
\noindent
\textbf{Theorem.}
\emph{
Let~$G$ be a finite group of Lie type of even characteristic. Then~$G$ is not a
minimal counterexample to {\rm \cite[Theorem~$1.1.5$]{HL1}}.}

\bigskip
The proof of this theorem, which follows from Propositions~\ref{E1LieEvenEasy},
\ref{PO8Plus}, \ref{PSU3qEven}, \ref{PSUSmallDegrees}, and \ref{E6SmallDegrees} 
below, is by far the most difficult part of our work. Let~$G$ be a finite simple 
group of Lie type of even characteristic. One reason for the complexity of our 
task is the fact that almost every irreducible character of~$G$ has odd degree. 
For example, if $G = E_6(q)$ with~$q$ even, then~$G$ has $q^6 + q^5 + x$ with 
$x < q^5$ irreducible characters, of which $q^{6} + 8q^2$ have odd degree. Of 
course, not all of these are real, but the task is to identify the real 
characters among these. It is here, where we have to use the full power of 
Deligne-Lusztig theory, in particular Lusztig's generalized Jordan decomposition 
of characters. The book~\cite{GeMa} by Geck and Malle is of enormous help in 
this respect. A rough sketch of the proof is given in \cite{HLIschia}.
 
The general strategy is the same as in the first part of our work. The large 
degree method (see \cite[Subsection~$3.1$]{HL1}) rules out most of the instances
to be considered. For the characters of smaller degrees we apply the restriction 
method (see \cite[Subsection~$4.3$]{HL1}). This general approach fails for 
numerous particular cases, be it for characters of very small degrees or for
very small values of~$q$ or the Lie rank of~$G$. These extra cases require 
lengthy and tedious calculations. The Tables~\cite{LL,LL2} by Frank L{\"u}beck 
and extensive computations with GAP~\cite{GAP04} and 
Chevie~\cite{chevie,Michel} provide substantial contributions to our proofs.

Let us now briefly comment on the content of the individual sections. 
Section~$2$ collects a few preliminaries of general nature. In Section~$3$ we 
introduce the groups under examination, and prove our result for those with a 
trivial Schur multiplier and no graph automorphism of order~$3$. We are then 
left with~$G$ equal to one of $\SL_d( q )$ with 
$\gcd( d, q - 1 ) > 1$, $\SU_d( q )$ with $\gcd( d, q + 1 ) > 1$, $E_6(q)$ with 
$3 \mid q - 1$, ${^2\!E}_6(q)$ with $3 \mid q + 1$, or $\P\Omega^+_8(q)$. In
Section~$4$ we provide the relevant theoretical background for these groups, 
including the specific $\sigma$-setups, duality, automorphisms and semisimple
characters. In Section~$5$ we establish bounds on the orders of centralizers of 
particular automorphisms and elements in order to apply the large degree
method. The results are sufficient to rule out $\P\Omega^+_8(q)$ as a minimal
counterexample. Section~$6$ is devoted to the special linear and unitary groups.
After lengthy preparations, and a special treatment for the groups of 
degree~$3$, we can finally show that these groups do not yield minimal 
counterexamples. The analogous result is then proved in Section~$7$ for the
simple groups of Lie type~$E_6$, twisted and untwisted.

In this second part of our series of two papers we will use the notation and
preliminary results of the first part. Recall in particular
\cite[Notation~$4.1.1$]{HL1}: The symbol~$G$ denotes a non-abelian finite
simple group, which here is a group of Lie type of even characteristic. 
Next,~$V$ denotes a non-trivial irreducible $\mathbb{R}G$-module of odd 
dimension, and~$\rho$ the representation of~$G$ afforded by~$V$. Moreover, 
$n \in \GL(V)$ is an element of finite order normalizing~$\rho(G)$. 
Finally,~$\nu$ denotes the automorphism of~$G$ induced by~$n$.

\section{General preliminaries}

We begin with estimates for bounding character degrees and centralizer
orders in linear and unitary groups.
\begin{lem}
\label{OrderEstimates}
Let $a \in \mathbb{R}$ and $m \in \mathbb{Z}$ with $a, m \geq 2$.
Then the following inequalities hold.

{\rm (a)} We have $$a^{m(m+1)/2}(1 - a^{-1} - a^{-2}) \leq 
\prod_{i=1}^m (a^i - 1) \leq a^{m(m+1)/2}.$$

{\rm (b)} We have $$a^{m(m+1)/2}(1 - a^{-1} - a^{-2}) \leq 
\prod_{i=1}^m (a^i - (-1)^i) \leq a^{m(m + 1)/2}(1 - a^{-1} - a^{-2})^{-1}.$$
\end{lem}
\begin{proof}
(a) The second estimate is trivial. According to
\cite[Proof of Lemma~$4.1$]{LMT}, we
have
$$
\prod_{i=1}^{\infty}(1 - a^{-i}) \geq 1 - a^{-1} - a^{-2}.
$$
Thus $\prod_{i=1}^{m}(1 - a^{-i}) \geq 1 - a^{-1} - a^{-2}$. Multiplying this
inequality with $a^{m(m+1)/2}$, we obtain
$\prod_{i=1}^m (a^i - 1) \geq a^{m(m+1)/2}(1 - a^{-1} - a^{-2})$.

(b) To prove the first estimate, observe that $\prod_{i=1}^m (a^i - (-1)^i)
\geq \prod_{i=1}^m (a^i - 1)$,
and use the first estimate of~(a). To prove the second estimate, observe that
\begin{eqnarray*}
\prod_{i=1}^m (a^i - (-1)^i) & = &
\frac{\prod_{i=1}^{\lceil m/2 \rceil} (a^{2i - 1} - 1)(a^{2i - 1} + 1)
\prod_{i=1}^{\lfloor m/2 \rfloor}(a^{2i} - 1)^2}
{\prod_{i=1}^{m}( a^i - 1)} \\
& \leq & \frac{\prod_{i = 1}^m a^{2i}}{\prod_{i=1}^{m}( a^i - 1)},
\end{eqnarray*}
and use the first estimate of~(a).
\end{proof}

We will also need the following lemma on conjugation in algebraic groups.

\addtocounter{subsection}{1}
\begin{lem}
\label{Conjugation0}
Let~$H$ be an algebraic group (which could be finite), and let $M \leq H$ be a
closed subgroup. Let $t, t' \in Z( M )$ with $C_{H}( t ) = M$. Suppose that~$t$
and~$t'$ are conjugate in~$H$.

Then~$t$ and~$t'$ are conjugate in~$N_{H}( M )$. If~$N_{H}( M )$ is a product
$N_{H}( M ) = M S$ for some $S \leq N_{H}( M )$, then~$t$ and~$t'$ are conjugate
in~$S$. Moreover, if $M \cap S = \{ 1 \}$, then~$S$ acts regularly on the set of
elements of~$Z( M )$ that are $H$-conjugate to~$t$.
\end{lem}
\begin{proof}
Let $h \in H$ with ${^ht} = t'$. Then
$$M \leq C_{H}( t' ) = {^h\!C}_{H}( t ) = {^h\!M}.$$
This implies that $M = {^h\!M}$, as~$M$  and~${^h\!M}$ are closed subgroups
of~$H$ of the same dimension. Hence $h \in N_{H}( M )$. The second statement is
clear, since $t \in Z( M )$. If $M \cap S = \{ 1 \}$, the centralizer of~$t$
in~$S$ is trivial. This implies the final statement.
\end{proof}

Let us introduce one further piece of notation. Let~$H$ be a group,
$\beta \in \Aut(H)$ and~$l$ a positive integer. We then write
$$N_{\beta,l}( h ) := \prod_{i=0}^{l-1} \beta^i(h)$$
for $h \in H$. With this notation, we have 
\begin{equation}
\label{PowerFormula}
(\ad_h \circ \beta)^l = \ad_{N_{\beta,l}(h)} \circ \beta^l
\end{equation}
for every $\beta \in \Aut(H)$, every $h \in H$ and every positive integer~$l$.

\section{The elementary cases}
\label{EvenCharacterteristic}

In this section we introduce the groups to be investigated and prove our result
for a large subclass of them by an elementary argument.

\subsection{The groups}
If $G = \P\Omega_{2d + 1}( 2 )$ for $d \geq 3$ or if $G = {^2\!F}_4(2)'$, 
then~$G$ has the $E1$-property by \cite[Lemma~$4.4.1$]{HL1}. If $G = F_4(2)$, 
then~$G$ has the $E1$-property by \cite[Lemma~$4.4.2$]{HL1}. In order to prove 
the $E1$-property for simple groups of Lie type of even characteristic, we may 
therefore assume that~$G$ is one of the groups specified in part~(a) of the 
following hypothesis. Part~(b) lists a subset of these groups which require a
deeper analysis.

\addtocounter{num}{1}
\begin{hypo}
\label{GroupsOfEvenCharacteristic}
{\rm
(a) Let~$G$ be one of the groups in Lines~$1$--$3$, $5$--$14$ or~$16$ of 
\cite[Table~$1$]{HL1}, where~$q$ is even. If~$G$ is as in Line~$3$ or~$8$ of 
this table, assume that $q > 2$. For the convenience of the reader, we reproduce 
the restricted list in Table~\ref{TableOfGroupsEven}.

\setlength{\extrarowheight}{0.5ex}
\begin{table}
\caption{\label{TableOfGroupsEven} The relevant simple groups of Lie type of 
even characteristic}
$$
\begin{array}{rrrc}\hline\hline
\text{\rm Row} & \multicolumn{1}{c}{\text{\rm Names}} & \text{\rm Rank} & 
\text{\rm Conditions} \\ \hline\hline
1 & A_{d-1}(q), \PSL_{d}(q) & d \geq 2 & q \text{\rm\ even}, (d,q) \neq (2,2), (2,4), (3,2) \\ 
2 & {^2\!A}_{d-1}(q), \PSU_{d}(q) & d \geq 3 & q \text{\rm\ even}, (d,q) \neq (3,2), (4,2) \\ 
3 & B_d(q), \P\Omega_{2d+1}(q) & d \geq 2 & q > 2 \text{\rm\ even\ } \\ 
5 & D_d(q), \P\Omega^+_{2d}(q) & d \geq 4 & q \text{\rm\ even\ } \\ 
6 & {^2\!D}_d(q), \P\Omega^-_{2d}(q) & d \geq 4 & q \text{\rm\ even\ } \\ 
7 & G_2(q) & & q > 2 \text{\rm\ even\ } \\
8 & F_4(q) & & q > 2 \text{\rm\ even\ } \\
9 & E_6(q) & & q \text{\rm\ even\ }\\
10 & E_7(q) & & q \text{\rm\ even\ }\\
11 & E_8(q) & & q \text{\rm\ even\ }\\
12 & {^3\!D}_4(q) & & q \text{\rm\ even\ }\\
13 & {^2\!E}_6(q) & & q \text{\rm\ even\ }\\
14 & {^2\!B}_2(q) & & q = 2^{2m+1} > 2 \\
16 & {^2\!F}_4(q) & & q = 2^{2m+1} > 2 \\ \hline
\end{array}
$$
\end{table}

(b) Let~$G$ be in one of the following subsets of the groups of 
Table~\ref{TableOfGroupsEven}.
\begin{itemize}
\item[(i)] $G = \PSL_d(q)$ with $d \geq 3$ and $\gcd( d, q - 1 ) > 1$, 
\item[(ii)] $G = \PSU_d( q )$ with $d \geq 3$ and $\gcd( d, q + 1 ) > 1$, 
\item[(iii)] $G = E_6(q)$ and $3 \mid q - 1$, 
\item[(iv)] $G = {^2\!E}_6(q)$ and $3 \mid q + 1$, 
\item[(v)] $G = \P\Omega^+_8(q)$.
\end{itemize}
}\hfill{$\Box$}
\end{hypo}

\addtocounter{subsection}{1}
\subsection{Restricting to a maximal unipotent subgroup}
\label{UsingGelfandGraevCharacters} 
For~$G$ as in Hypothesis~\ref{GroupsOfEvenCharacteristic}, let $B = UT$ denote 
the standard Borel subgroup of~$G$, with standard unipotent subgroup~$U$ and 
standard torus~$T$; see \cite[Subsection~$5.3$]{HL1}. Here we will restrict the 
character of~$V$ to~$U$ to deal with many of the cases listed in
Hypothesis~\ref{GroupsOfEvenCharacteristic}. In the following proposition we 
use \cite[Notation~$4.1.1$]{HL1}; see the last paragraph of the introduction.

\addtocounter{num}{1}
\begin{prp}
\label{EvenCharacteristicTheorem}
Let~$G$ be as in {\rm Hypothesis~\ref{GroupsOfEvenCharacteristic}(a)}. Assume 
that~$U$ and~$T$ are stable under~$\nu$. Let~$\chi$ denote the character of~$V$.
Then $(G,V,n)$ has the $E1$-property, if any one of the two conditions below is 
satisfied.

{\rm (a)} There is $T$-orbit~$\mathcal{O}$ of linear characters of~$U$ fixed
by~$\nu$ such that $\langle \Res_U^G( \chi ), \lambda \rangle$ is odd for every 
$\lambda \in \mathcal{O}$.

{\rm (b)} Every $T$-orbit of linear characters of~$U$ is fixed by~$\nu^2$. 
\end{prp}
\begin{proof}
(a) Let $\lambda \in \mathcal{O}$. If~$\lambda$ is the trivial character, 
$\langle \Res_U^G( \chi ), \lambda \rangle$ equals the degree of 
${^*\!R}_T^G( \chi )$, which then must be odd. Let~$V_1 \leq V$ denote the 
subspace of $U$-fixed points of~$V$. This is an $\mathbb{R}T$-module with 
character ${^*\!R}_T^G( \chi )$. Thus there is an irreducible 
$\mathbb{R}T$-submodule~$S$ of $V_1$ of odd dimension. As~$|T|$ is odd,~$S$ is 
the trivial $\mathbb{R}T$-module. It follows that $\Res_B^G( V )$ contains a 
trivial constituent. Hence~$\chi$ is a principal series character, and in 
particular unipotent. The only unipotent character of~$G$ of odd degree is the 
trivial character; see \cite[Theorem~$6.8$]{MalleHeight0}, taking into account 
our restrictions on~$G$. This is excluded by hypothesis. Thus~$\lambda$ is not 
the trivial character. 

Since~$U$ is generated by involutions,~$U/U'$ is an elementary abelian 
$2$-group, so that~$\lambda$ is realizable over~$\mathbb{R}$. Let~$V_1$ denote 
the $\mathbb{R}U$-submodule of $\Res^G_U( V )$ with character 
$\langle \Res_U^G( \chi ), \lambda \rangle \lambda$, i.e.\ $V_1$ is the 
$\lambda$-homogeneous component of $\Res^G_U( V )$. As~$\nu$ normalizes~$U$ and 
thus also~$U'$, the submodule~$nV_1$ of~$\Res^G_U(V)$ is the 
${^\nu\!\lambda}$-homogeneous component of $\Res^G_U(V)$. Since ${^\nu\!\lambda}
\in \mathcal{O}$, there is $t \in T$ such that $\rho(t)nV_1 = V_1$. 
Replacing~$n$ by $\rho(t)n$, we may assume that~$n$ stabilizes~$V_1$. As~$U$ is 
a $2$-group, it has the $E1$-property by \cite[Corollary~$3.2.4$]{HL1}.
In 
particular, $(U,V_1)$ has the $E_1$-property by \cite[Lemma~$3.1.2$]{HL1}.
In 
turn~$(G,V,n)$ has the $E1$-property by \cite[Lemma~$3.1.3$]{HL1}.

(b) Let~$\mathcal{O}$ be a $T$-orbit of linear characters of~$U$ and let
$\lambda \in \mathcal{O}$. Clearly, $\langle \Res_U^G( \chi ), \lambda \rangle = 
\langle \Res_U^G( \chi ), \lambda' \rangle$ for every other
$\lambda' \in \mathcal{O}$, so $\mathcal{O}$ contributes
$|\mathcal{O}|\langle \Res_U^G( \chi ), \lambda \rangle$ to the degree of~$\chi$.
If~${^{\nu}\!\mathcal{O}} \neq \mathcal{O}$, we obtain an even contribution to
the degree of~$\chi$, as $\langle \Res_U^G( \chi ), {^{\nu}\!\lambda} \rangle = 
\langle \Res_U^G( {^{\nu}\!\chi} ), {^{\nu}\!\lambda} \rangle = 
\langle {^{\nu}\Res}_U^G( \chi ), {^{\nu}\!\lambda} \rangle =
\langle \Res_U^G( \chi ), \lambda \rangle$ and ${^{\nu^2}\!\mathcal{O}} = 
\mathcal{O}$. Thus there is a $\nu$-invariant orbit~$\mathcal{O}$ with
$|\mathcal{O}|\langle \Res_U^G( \chi ), \lambda \rangle$ odd for all
$\lambda \in \mathcal{O}$. Hence the hypothesis of~(a) is satisfied.
\end{proof}

We can now prove the $E1$-property for a large subclass of the simple groups of
Lie type of even characteristic. Our argument involves the notion of regular 
characters of~$U$, for which we follow~\cite[14.B]{CeBo2}. Notice that regular 
characters are called non-degenerate in \cite[p.~$253$]{C2}. Recall the 
$\sigma$-setup for~$G$ introduced in \cite[Subsection~$5.2$]{HL1}. Recall in 
particular that $\overline{G}$ is a simple adjoint algebraic group so that 
$Z( \overline{G} )$ is trivial. Under the hypothesis of the following 
proposition, we have $G = \overline{G}^{\sigma}$. We will also use the notation 
introduced in \cite[Subsection~$5.3$]{HL1} related to the $BN$-pair of~$G$.

\begin{prp}
\label{E1LieEvenEasy}
Let~$G$ be one of the groups of
{\rm Hypothesis~\ref{GroupsOfEvenCharacteristic}(a)} not listed
in {\rm Hypothesis~\ref{GroupsOfEvenCharacteristic}(b)(i)--(iv)}. If 
$G = \P\Omega_8^+(q)$, assume that~$\nu^2$ is a product of an inner 
automorphism of~$G$ with a field automorphism. Then~$(G,V,n)$ has the 
$E1$-property.
\end{prp}
\begin{proof}
We have $\nu = \ad_g \circ \mu$ for some $g \in \overline{G}^{\sigma}$ and some 
$\mu \in \Gamma_G \times \Phi_G$; see \cite[Subsection~$5.5$]{HL1}. Since 
$G = \overline{G}^{\sigma}$, we may assume that 
$\nu \in \Gamma_G \times \Phi_G$. In particular,~$\nu$ fixes~$U$ and~$T$. Every 
element of~$\Gamma_G$, except in the case where $G = \P\Omega_8^+(q)$, has order 
at most~$2$. Thanks to our assumption in the latter case, we have 
$\nu^2 \in \Phi_G$, and thus~$\nu^2$ stabilizes every standard Levi subgroup 
of~$G$ and the unipotent radical of every standard parabolic subgroup of~$G$.

Let~$\lambda$ denote a linear character of~$U$. By~\cite[14.B]{CeBo2}, there is 
a standard parabolic subgroup $P_I = L_IU_I$ such that~$U_I$ is in the kernel 
of~$\lambda$. Choose~$U_I$ maximal with this property. Then~$\lambda$, viewed as 
a character of~$U/U_I$, is regular. Indeed, by \cite[Theorem~$2.8.7$]{C2}, the 
unipotent radicals of parabolic subgroups of~$L_I$ are of the form $U_J/U_I$ for 
subsets $J \subseteq I$.
Now $L_I = \overline{L}_I^{\sigma}$ for the standard Levi 
subgroup~$\overline{L}_I$ of~$\overline{G}$. As~$\overline{L}_I$ has connected 
center (see \cite[Proposition~$8.1.4$]{C2}), the argument given in the proof 
of~\cite[Theorem~$8.1.2$(ii)]{C2} shows that then~$L_I$ has a unique $T$-orbit 
of regular characters.  As $\nu^2$ stabilizes~$U$ and~$U_I$ and sends~$\lambda$
to a regular character of~$U/U_I$, it follows that~$\nu^2$ stabilizes the
$T$-orbit containing~$\lambda$. Our assertion follows from 
Proposition~\ref{EvenCharacteristicTheorem}.
\end{proof}

If $G = \P\Omega_8^+(q)$ and~$\nu$ does not satisfy the hypothesis of 
Proposition~\ref{E1LieEvenEasy}, then $\nu = \ad_h \circ \iota \circ \mu$ for 
some $h \in G$, some $\iota \in \Gamma_G$ of order~$3$, and some 
$\mu \in \Phi_G$. This case will be treated in Proposition~\ref{PO8Plus} below.

\section{Preliminaries for the remaining groups}

Here, we collect the properties relevant to our proofs for the remaining 
group to be considered.

\subsection{The $\sigma$-setup for the remaining groups}
\label{TheRemainingGroups}
Let~$G$ be one of the groups listed in 
Hypothesis~\ref{GroupsOfEvenCharacteristic}(b). We now specify the choice 
of~$\overline{G}$ in the $\sigma$-setup for~$G$ and recall some notation
concerning automorphisms. In the present context,~$\mathbb{F}$ is an algebraic
closure of the field with two elements. In cases (i) and (ii) for~$G$ we take 
$\overline{G} = \PGL_d( \mathbb{F} )$, in cases (iii) and (iv) we take
$\overline{G} = E_6( \mathbb{F} )_{\rm ad}$, and in case (v) we let 
$\overline{G} = \PCSO_8( \mathbb{F} )$, the simple algebraic group 
over~$\mathbb{F}$ of adjoint type~$D_4$. The standard torus~$\overline{T}$ 
contained in the standard Borel subgroup~$\overline{B}$ gives rise to the root 
system of~$\overline{G}$, which is of type~$A_{d-1}$,~$E_6$ or~$D_4$, 
respectively. We write~$\Pi$ for the set of simple roots determined 
by~$\overline{T}$ and~$\overline{B}$. The corresponding Dynkin diagrams are 
displayed in \cite[Figure~$1$]{HL1}.

Recall the notation introduced in \cite[Subsection~$5.5$]{HL1}. In particular, 
$\varphi = \varphi_2$ is the standard Frobenius morphism of~$\overline{G}$ 
defined in \cite[Theorem~$1.15.4$(a)]{GLS}, and 
$\Phi_{\overline{G}} = \langle \varphi \rangle$. If~$\iota$ denotes a symmetry 
of the Dynkin diagram of~$\overline{G}$, then~$\iota$ also denotes the 
corresponding standard graph automorphism of~$\overline{G}$. Let~$f$ be a 
positive integer and $q = 2^f$. Taking $\sigma := \varphi^f$, we obtain the 
untwisted groups $\overline{G}^{\sigma} = \PGL_d( q )$, respectively 
$E_6(q)_{\text{ad}}$, respectively $\P\Omega^+_8(q)$. The corresponding simple 
groups are $G = \PSL_d( q )$, respectively $E_6(q)$, respectively 
$\P\Omega^+_8(q)$. If $\overline{G} = \PGL_d( \mathbb{F} )$ or 
$\overline{G} = E_6( \mathbb{F} )_{\rm ad}$, the group of symmetries of the 
Dynkin diagram of~$\overline{G}$ has order~$2$, so that 
$\Gamma_{\overline{G}} = \langle \iota \rangle$ for the non-trivial such 
symmetry~$\iota$. Taking $\sigma := \iota \circ \varphi^f$, we obtain the 
twisted groups $\overline{G}^{\sigma} = \PGU_d( q )$, respectively 
${^2\!E}_6(q)_{\text{ad}}$.  The corresponding simple groups are 
$G = \PSU_d( q )$, respectively ${^2\!E}_6(q)$. Recall that 
$\overline{G}^{\sigma} = G\overline{T}^{\sigma}$; 
see \cite[Theorem~$2.2.6$(g)]{GLS}.

For the sake of a uniform treatment, we will use the common 
$\varepsilon$-convention to distinguish between the twisted and the untwisted 
groups in cases (i)--(iv). Fix $\varepsilon \in \{ -1, 1 \}$. If 
$\varepsilon = 1$, let $\sigma = \varphi^f$, and if $\varepsilon = -1$, let 
$\sigma = \iota \circ \varphi^f$, where, as above, $\iota$ is the non-trivial
symmetry of the Dynkin diagram of~$\overline{G}$. We will thus understand
$\PSL^{\varepsilon}(q)$ to be $\PSL_d(q)$ if $\varepsilon = 1$, and $\PSU_d(q)$, 
if $\varepsilon = -1$. Likewise, $E_6^{\varepsilon}(q)$ denotes the simple group 
$E_6(q)$ if $\varepsilon = 1$, and ${^2\!E}_6(q)$, if $\varepsilon = -1$. This 
$\varepsilon$-convention is also used for the groups $\overline{G}^{\sigma}$.

Clearly,~$\overline{G}^{\sigma}$ and thus also~$G$, are invariant 
under~$\varphi$ and~$\iota$ for every symmetry~$\iota$ of the Dynkin diagram
of~$\overline{G}$. The restrictions to~$\overline{G}^{\sigma}$ or~$G$ 
of~$\varphi$ and any such~$\iota$ are denoted by the same letters. In 
particular, $\Phi_G = \langle \varphi \rangle$. If $\overline{G} = 
\PGL_d( \mathbb{F} )$ or $\overline{G} = E_6( \mathbb{F} )_{\rm ad}$, 
we have $\Gamma_{\overline{G}} = \langle \iota \rangle$ for a non-trivial
symmetry~$\iota$ of the Dynkin diagram of~$\overline{G}$. If the 
corresponding~$G$ is untwisted, i.e.\ $G = \PSL_d( q )$ or $G = E_6(q)$, 
then $\Gamma_G = \langle \iota \rangle$ has order~$2$. If~$G$ is twisted, 
then $\iota = \varphi^f$, viewed as elements of~$\Aut(G)$, so that 
$\iota \in \Phi_G$. Thus according to the conventions 
in~\cite[Definitions~$2.2.4$, $2.5.10$]{GLS}, we have $\Gamma_G = \{ 1 \}$, 
although $\Gamma_{\overline{G}}$ is non-trivial. If~$G$ is as in case~(v),
$\Gamma_{\overline{G}} \cong \Gamma_G$ is isomorphic to a symmetric group on
three letters. In this case, $\Gamma_G$ has elements of order~$2$ and~$3$; the
latter are called triality automorphisms of~$G$.

\subsection{The dual groups}
\label{TheDualGroups}
At some stage we will also need to consider the groups dual to~$\overline{G}$.
Namely, to describe the irreducible characters of~$G$ it is more convenient to 
realize~$G$ as a central quotient of the group 
${\overline{G}^*}^{\sigma^*}$, where~$\overline{G}^*$ is a group dual 
to~$\overline{G}$ and~$\sigma^*$ is a Steinberg morphism of~$\overline{G}^*$ 
dual to~$\sigma$; see \cite[Definition~$1.5.17$]{GeMa}. For a reason that will
become apparent in Subsection~\ref{AutomorphismsIII} below, we will 
henceforth usually suppress the asterisk in the notation of~$\sigma^*$, and 
just write~$\sigma$ instead. Only in some proofs we stick, for reasons of 
clarity, to the more precise notation. 

We take~$\overline{G}^*$ as the simply connected version of the simple algebraic 
group of type $A_{d-1}$,~$E_6$, respectively~$D_4$. To be specific, if 
$\overline{G} = \PGL_d( \mathbb{F} )$, $E_6( \mathbb{F} )_{\rm ad}$ or 
$\PCSO_8( \mathbb{F} )$, we let $\overline{G}^* = \SL_d( \mathbb{F} )$, 
$E_6( \mathbb{F} )_{\rm sc}$ and $\Spin_8( \mathbb{F} )$, respectively. 
There is an isogeny, i.e.\ a surjective homomorphism of algebraic groups with 
finite kernel,
\begin{equation}
\label{Isogeny}
\overline{G}^* \rightarrow \overline{G}.
\end{equation}
If $\overline{G}^* = \SL_d( \mathbb{F} )$, we have 
${\overline{G}^*}^{\sigma} = \SL_d( q )$, respectively $\SU_d(q)$. If 
$\overline{G}^* = E_6( \mathbb{F} )_{\rm sc}$, the group 
${\overline{G}^*}^{\sigma}$ is the universal covering group of $G = E_6(q)$, 
respectively ${^2\!E}_6(q)$. In all cases, the center of 
${\overline{G}^*}^{\sigma}$ is cyclic and of odd order. The groups 
$\overline{G}^\sigma$ are almost simple, whereas the groups 
${\overline{G}^*}^{\sigma}$ are quasisimple. 

\subsection{Automorphisms, I}
\label{AutomorphismsIII}
Assume that~$G$ is one of the groups of 
Hypothesis~\ref{GroupsOfEvenCharacteristic}(b)(i)--(iv), so that 
$\overline{G} = \PGL_d( \mathbb{F} )$ or 
$\overline{G} = E_6( \mathbb{F} )_{\rm ad}$. Here, we shortly comment on the 
connections between the automorphism groups of~$G$, $\overline{G}^{\sigma}$, 
$\overline{G}$, and their duals. Let us recall the notation introduced in 
\cite[Definition~$1.15.1$]{GLS}. Thus $\Aut_1( \overline{G} )$ is the set of all 
automorphisms~$\alpha$ of~$\overline{G}$ (as an abstract group), such that 
$\alpha$ or $\alpha^{-1}$ is an endomorphism of~$\overline{G}$ as an algebraic 
group. Notice that $\sigma \in \Aut_1( \overline{G} )$. Also, 
$\Aut_1( \overline{G} )$ is a group; see \cite[Theorem~$1.15.7$(a)]{GLS}. 
Analogous notations will be used for~$\overline{G}^*$.

The isogeny~(\ref{Isogeny}) gives rise to a natural map
\begin{equation}
\label{Aut1Ismorphism}
\Aut_1( \overline{G}^* ) \rightarrow \Aut_1( \overline{G} ),
\end{equation}
which is in fact an isomorphism; see \cite[Theorems~$1.15.6$(c)]{GLS}. 
By \cite[Theorems~$1.15.6$(b)]{GLS}, the isomorphism~(\ref{Aut1Ismorphism}) 
preserves isogenies. 

\addtocounter{num}{3}
\begin{lem}
\label{DualIsogenies}
Let $t \in \overline{T}$ and let 
$\mu \in \Gamma_{\overline{G}} \times \Phi_{\overline{G}}$ be an isogeny.
Then $\alpha := \ad_t \circ \mu$ is an isogeny of~$\overline{G}$. Let 
$\alpha^* \in \Aut_1( \overline{G}^* )$ be the inverse image of~$\alpha$
under~{\rm (\ref{Aut1Ismorphism})}. Then~$\alpha^*$ is an isogeny 
of~$\overline{G}^*$, which is dual to~$\alpha$ in the sense of 
\cite[Proposition~$11.1.11$]{DiMi2}.
\end{lem}
\begin{proof}
Clearly,~$\alpha$ is an isogeny. As~(\ref{Aut1Ismorphism}) preserves isogenies, 
$\alpha^*$ is an isogeny as well. To prove the second claim, we may assume that 
$t = 1$. We may also assume that $\alpha = \iota$, or $\alpha = \varphi$, the
generators, introduced in Subsection~\ref{TheRemainingGroups},
of $\Gamma_{\overline{G}}$ and $\Phi_{\overline{G}}$, respectively.
The claim is then obvious.
\end{proof}

Now identify~$\Aut_1( \overline{G}^* )$ and~$\Aut_1( \overline{G} )$
via~(\ref{Aut1Ismorphism}). As $\sigma \in \Aut_1( \overline{G} )$ is a 
Steinberg morphism, its inverse image in $\Aut_1( \overline{G}^* )$ is a 
Steinberg morphism of~$\overline{G}^*$ dual to~$\sigma$. This justifies our 
simplified notation introduced in Subsection~\ref{TheRemainingGroups}.

Let $\alpha \in \Aut( G )$. Then~$\alpha$ is the restriction to~$G$ of an 
element of~$\Aut_1( \overline{G} )$ commuting with~$\sigma$; see 
\cite[Theorem~$2.5.4$]{GLS}. In fact, this element giving rise to~$\alpha$ may 
be chosen to be an isogeny. Indeed, as~$\alpha$ has finite order, 
$\alpha^{-1} = \alpha^l$ for some non-negative integer~$l$. The claim follows
from this, as either~$\beta$ or~$\beta^{-1}$ is an isogeny for every
$\beta \in \Aut_1( \overline{G} )$. Also,~$\alpha$ extends to a unique 
automorphism of $\overline{G}^{\sigma}$ arising from the elements 
of~$\Aut_1( \overline{G} )$ commuting with~$\sigma$ and restricting 
to~$\alpha$, and we will tacitly view~$\alpha$ as an element of 
$\Aut( \overline{G}^{\sigma} )$ this way. 

\addtocounter{subsection}{1}
\subsection{Semisimple characters}
\label{SemisimpleCharacters}
As already mentioned in Subsection \ref{TheDualGroups}, we are going to
describe the irreducible characters of~$G$ through the epimorphism
${\overline{G}^*}^{\sigma} \rightarrow G$ via inflation. This does not introduce 
new characters to be investigated, as $Z( {\overline{G}^*}^{\sigma} )$, the 
kernel of this epimorphism, is cyclic of odd order, so that every real 
irreducible character of~${\overline{G}^*}^{\sigma}$ has 
$Z( {\overline{G}^*}^{\sigma} )$ in its kernel. (Contrary to a more common 
approach, we interchange the roles of the group and its dual here, as our focus 
is on the irreducible characters of the group~${\overline{G}^*}^{\sigma}$.) 

Viewed as element of $\Irr( {\overline{G}^*}^{\sigma} )$, the character of~$V$ 
is what is called a semisimple character. For this notion, we follow 
\cite[Definition~$2.6.9$]{GeMa}. By \cite[Corollary~$2.6.18$(a)]{GeMa} and the 
results in \cite[Section~$15$]{CeBo2}, in particular 
\cite[Th{\'e}or{\`e}me~$15.10$, Corollaire~$15.11$, Proposition~$15.13$ and 
Corollaire~$15.14$]{CeBo2}, this matches with \cite[D{\'e}finition 15.A]{CeBo2}.

The set $\Irr( {\overline{G}^*}^{\sigma} )$ is partitioned into Lusztig series 
$\mathcal{E}( {\overline{G}^*}^{\sigma}, s )$, where~$s$ runs through the 
$\overline{G}^{\sigma}$-conjugacy classes of semisimple elements 
of~$\overline{G}^{\sigma}$; see \cite[Definition~$2.6.1$]{GeMa}. The following 
lemma collects the properties of semisimple characters relevant to our further 
investigation.

\addtocounter{num}{1}
\begin{lem}
\label{SemisimpleCharactersLemma}
Let $\chi \in \Irr( {\overline{G}^*}^{\sigma} )$ and let 
$s \in \overline{G}^{\sigma}$ be semisimple such that 
$\chi \in \mathcal{E}( {\overline{G}^*}^{\sigma}, s )$. Let 
$\alpha^* \in \Aut_1( {\overline{G}^*} )$ be an isogeny commuting with~$\sigma$.
Let $\alpha \in \Aut_1( \overline{G} )$ be an isogeny commuting with~$\sigma$, 
dual to~$\alpha^*$. (Such an isogeny exists; see, e.g.\ \cite[$1.7.11$]{GeMa}.)
Then the following statements hold.

{\rm (a)} The degree of~$\chi$ is odd, if and only if~$\chi$ is semisimple in 
the sense of \cite[Definition~$2.6.9$]{GeMa}. If~$\chi(1)$ is odd, then
$$\chi(1) = [\overline{G}^\sigma\colon\!C_{\overline{G}^{\sigma}}(s)]_{2'}.$$

{\rm (b)} Assume that~$\chi$ has odd degree.
Then the following assertions are true.

\begin{itemize}
\item[(i)] The semisimple characters in 
$\mathcal{E}( {\overline{G}^*}^{\sigma}, s )$ all have the same 
degree, and they correspond, via Lusztig's Jordan decomposition of characters, 
to the irreducible characters of 
$(C_{\overline{G}}( s )/C^{\circ}_{\overline{G}}( s ))^\sigma$. In particular, 
if $C_{\overline{G}}( s )$ is connected, then~$\chi$ is the unique
semisimple character in $\mathcal{E}( {\overline{G}^*}^{\sigma}, s )$.

\item[(ii)] If~$\chi$ is real, then~$s$ is real in $\overline{G}^{\sigma}$. 
If~$s$ is real in $\overline{G}^{\sigma}$ and 
$C_{{\overline{G}}}( s )$ is connected, then~$\chi$ is real.

\item[(iii)] If~$\chi$ is $\alpha^*$-invariant, then the 
$\overline{G}^{\sigma}$-conjugacy class of~$s$ is invariant under~$\alpha$.
If the $\overline{G}^{\sigma}$-conjugacy class of~$s$ is $\alpha$-invariant
and $C_{{\overline{G}^*}}( s )$ is connected, then~$\chi$ is
$\alpha^*$-invariant.

\item[(iv)] Let $\overline{L} \leq \overline{G}$ be a standard Levi subgroup 
of~$\overline{G}$ with $s \in \overline{L}$. Let~$\overline{L}^*$ denote the 
standard Levi subgroup of~$\overline{G}^*$ dual to~$\overline{L}$.
Suppose that~$s$ is real in $\overline{L}^{\sigma}$
and that $C_{{\overline{L}^*}}( s )$ is connected.

Then $\mathcal{E}( {\overline{L}^*}^{\sigma}, s )$ contains a real, semisimple
character~$\psi$ of odd degree, such that 
$\langle R_{{\overline{L}^*}^{\sigma}}^{{\overline{G}^*}^{\sigma}}( \psi ), 
\chi \rangle = 1$.
\end{itemize}
\end{lem}
\begin{proof}
(a) and (b)(i). Use the notion of unipotent characters of 
$C_{\overline{G}}( s )^\sigma$ as introduced in \cite[$12$]{Lu}. Then, via 
Lusztig's generalized Jordan decomposition of characters,~$\chi$ corresponds to 
a unipotent character~$\lambda$ of $C_{\overline{G}}( s )^\sigma$ such that
$\chi(1) = 
\lambda(1)[\overline{G}^\sigma\colon\!C_{\overline{G}^{\sigma}}( s )]_{2'}$;
see \cite[Proposition~$5.1$]{Lu}. 

Suppose that~$\chi(1)$ is odd. Thus $\lambda(1)$ is odd, and hence~$\lambda$
lies over an odd degree unipotent character~$\lambda'$ of 
${C^{\circ}_{\overline{G}^{\sigma}}( s )}$. Then 
\cite[Theorem~$6.8$]{MalleHeight0} implies that~$\lambda'$ is the trivial 
character. It follows that~$\chi$ is semisimple by~\cite[Theorem~$2.6.11$(c) 
and Corollary~$2.6.18$(a)]{GeMa} and the definition of the generalized Jordan 
decomposition from the usual Jordan decomposition through a regular embedding 
of~$\overline{G}^*$. This also implies that the semisimple characters in 
$\mathcal{E}( {\overline{G}^*}^{\sigma}, s )$ are exactly those that correspond, 
via the generalized Jordan decomposition of characters, to the extensions 
of~$\lambda'$ to $C_{\overline{G}^{\sigma}}( s )$. Since 
${C_{\overline{G}^{\sigma}}( s )}/{C^{\circ}_{\overline{G}^{\sigma}}( s )}
= (C_{\overline{G}}( s )/C^{\circ}_{\overline{G}}( s ))^\sigma$ is abelian by 
\cite[Lemma $8.3$]{CeBo2}, this implies (b)(i), one direction of~(a), as well
as the degree formula.

Suppose now that~$\chi$ is semisimple. Then~$\chi(1)$ is odd by
\cite[Theorem~$2.6.11$(b) and Corollary~$2.6.18$(a)]{GeMa}. This gives the
other direction of~(a).

(b)(ii) If $\overline{S}$ is a $\sigma$-stable maximal torus of~$\overline{G}$
containing~$s$, then $R_{\overline{S}}^{\overline{G}}( s )$ is complex conjugate
to $R_{\overline{S}}^{\overline{G}}( s^{-1} )$ by the character formula; see
\cite[Theorem~$2.2.16$]{GeMa}. The definition of the set
$\mathcal{E}( {\overline{G}^*}^{\sigma}, s )$ 
(see \cite[Definition~$2.6.1$]{GeMa}) implies that the character complex 
conjugate to~$\chi$ lies in $\mathcal{E}( {\overline{G}^*}^{\sigma}, s^{-1} )$.
This proves the first assertion. If $C_{{\overline{G}}}( s )$ is 
connected,~$\chi$ is the unique semisimple character in 
$\mathcal{E}( {\overline{G}^*}^{\sigma}, s )$ by~(i). Since~$s$ is real, the 
character complex conjugate to~$\chi$ lies in 
$\mathcal{E}( {\overline{G}^*}^{\sigma}, s )$, and is semisimple by~(a).
Hence~$\chi$ is real.

(b)(iii) As~$\alpha$ and~$\alpha^*$ are dual isogenies, we have
$${^{\alpha^*}\!\mathcal{E}}( {\overline{G}^*}^{\sigma}, s ) =
\mathcal{E}( {\overline{G}^*}^{\sigma}, \alpha^{-1}( s ) );$$ 
see \cite[Proposition~$7.2$]{JTayl}. The proof now proceeds as in~(ii).

(b)(iv) The first claim follows from (i), (ii), and~(a), applied 
to~$\overline{L}$. Further, under any Jordan decomposition of 
characters,~$\psi$ corresponds to the trivial character; see 
\cite[Theorem~$2.6.11$(c)]{GeMa}. As~$\chi$ is semisimple by~(a), it 
corresponds, by~(i), to an irreducible character~$\lambda$ 
of~$C_{\overline{G}^{\sigma}}( s )$ which has 
$C^{\circ}_{\overline{G}^{\sigma}}( s )$ in its kernel. As 
$C_{\overline{L}}( s )$ is connected, we have $C_{\overline{L}^{\sigma}}( s )
\leq C^{\circ}_{\overline{G}^{\sigma}}( s )$. As~$\overline{L}$ is $1$-split,
$C_{\overline{L}^{\sigma}}( s )$ is a $1$-split Levi subgroup of 
$C^{\circ}_{\overline{G}}( s )$, and thus Lusztig induction from
$C_{\overline{L}^{\sigma}}( s )$ to $C^{\circ}_{\overline{G}^{\sigma}}( s )$
is just usual Harish-Chandra induction. Thus the trivial character of
$C_{\overline{L}^{\sigma}}( s )$, Harish-Chan\-dra induced to 
$C_{\overline{G}^{\sigma}}( s )$ in the sense of 
\cite[Definition~$4.8.8$]{GeMa}, contains~$\lambda$ with multiplicity~$1$; see
\cite[Proposition~$4.8.10$]{GeMa}. The second claim now follows from 
\cite[Theorem~$4.8.24$]{GeMa}. 
\end{proof}

The next lemma paves the way for our applications.
\begin{lem}
\label{ProofByHCInductionQEven}
Let~$G$ be as in {\rm Hypothesis~\ref{GroupsOfEvenCharacteristic}(b)(i)--(iv)} 
and suppose that every proper subgroup of~$G$ has the $E1$-property. Let~$\iota$
denote the standard graph automorphism of~$\overline{G}$ of order~$2$. Let 
$(V,n,\nu)$ be as in \cite[Notation~$4.1.1$]{HL1}, and let~$\chi$ denote the 
character of~$V$. View~$V$ as an 
$\mathbb{R}{\overline{G}^*}^{\sigma}$-module and~$\chi$ as a character 
of~${\overline{G}^*}^{\sigma}$ via inflation. Let $s \in \overline{G}^{\sigma}$ 
be such that $\chi \in \mathcal{E}( {\overline{G}^*}^{\sigma}, s )$. Then 
$(G,V,n)$ has the $E1$-property under the following hypotheses.

There is a $\iota$-stable, proper standard Levi 
subgroup~$\overline{L}$ of~$\overline{G}$ and a 
$\overline{G}^{\sigma}$-conjugate $s' \in \overline{L}^{\sigma}$, such that the 
conditions {\rm (i),(ii),(iii)} hold.

\begin{itemize}
\item[{\rm (i)}] The element~$s'$ is real in $\overline{L}^\sigma$.
\item[{\rm (ii)}] The centralizer~$C_{\overline{L}}( s' )$ is connected.
\item[{\rm (iii)}] For every $\alpha \in \Aut(G)$ 
stabilizing~$\overline{L}^{\sigma}$, the following holds: If~$\alpha(s)$ and~$s$ 
are conjugate in~$\overline{G}^{\sigma}$, then~$\alpha(s')$ and~$s'$ are 
conjugate in~$\overline{L}^{\sigma}$.
\end{itemize}
\end{lem}
\begin{proof}
Let the notation be as in Lemma~\ref{SemisimpleCharactersLemma}. 
As~$\overline{L}$ is $\iota$-stable and a standard Levi subgroup 
of~$\overline{G}$, it is $\sigma$-stable. As~$\iota$ commutes with~$\sigma$,
the finite group~$\overline{L}^{\sigma}$ is $\iota$-stable as well.
We may assume that $s = s' \in \overline{L}^{\sigma}$. Let~$\overline{L}^*$ 
denote the standard Levi subgroup of~$\overline{G}^*$ dual to~$\overline{L}$.
Then~$\overline{L}^*$ is~$\sigma$-stable. Let $\psi \in 
\Irr( {\overline{L}^*}^{\sigma} )$ denote the unique semisimple character in
$\mathcal{E}( {\overline{L}^*}^{\sigma}, s )$; see 
Lemma~\ref{SemisimpleCharactersLemma}(b)(i). Then~$\psi$ is real, of odd degree,
and~$\chi$ occurs with multiplicity~$1$ 
in~$R_{\overline{L}^*}^{\overline{G}^*}( \psi )$; see
Lemma~\ref{SemisimpleCharactersLemma}(b)(iv). Since~$\chi$ is not the trivial 
character, $s \neq 1$, an thus~$\psi$ is not the trivial 
character of~${\overline{L}^*}^{\sigma}$.

Write $\nu = \ad_h \circ \mu$ with
$h \in \overline{G}^{\sigma}$ and $\mu \in \Gamma_G \times \Phi_G$; see
\cite[Subsection~$5.5$]{HL1}.
As $\overline{L}$ is $\iota$-stable,~$\mu$ 
stabilizes~$\overline{L}^{\sigma}$ and the corresponding standard parabolic 
subgroup~$\overline{P}^{\sigma}$. Since
$\overline{G}^{\sigma} = G\overline{T}^{\sigma}$ (see 
Subsection~\ref{TheRemainingGroups}), there is $g \in G$ such that 
$\alpha := \ad_g \circ \nu = \ad_{t} \circ \mu$ with 
$t \in \overline{T}^{\sigma}$. In particular,~$\alpha$
stabilizes~$\overline{L}^{\sigma}$ and $\overline{P}^{\sigma}$. Also,~$\alpha$ 
is the restriction to~$G$ of an isogeny $\alpha \in \Aut_1( \overline{G} )$
commuting with~$\sigma$, which fixes~$\overline{T}$; see the penultimate 
paragraph in Subsection~\ref{AutomorphismsIII}. 

Let $\alpha^* \in \Aut_1( \overline{G}^* )^{\sigma}$ denote the inverse image
of~$\alpha$ under the isomorphism~(\ref{Isogeny}) in
Subsection~\ref{AutomorphismsIII}. Then~$\alpha$ and $\alpha^*$ are dual 
isogenies by Lemma~\ref{DualIsogenies}. Now~$\alpha^*$ restricts to an 
automorphism of~${\overline{G}^*}^{\sigma}$, which 
stabilizes~${\overline{L}^*}^{\sigma}$ and the corresponding standard parabolic 
subgroup of~${\overline{G}^*}^{\sigma}$. The automorphism induced by 
this~$\alpha^*$ on 
$G \cong {\overline{G}^*}^{\sigma}/Z( {\overline{G}^*}^{\sigma} )$ is the 
original~$\alpha$ we started with. As such,~$\alpha^*$ fixes~$\chi$, and 
so~$\alpha^*$ also stabilizes the $\overline{G}^{\sigma}$-conjugacy class 
of~$s$ by Lemma~\ref{SemisimpleCharactersLemma}(b)(iii). By Hypothesis~(iii), 
the $\overline{L}^{\sigma}$-conjugacy class of~$s$ is $\alpha$-stable as well.
It follows that~$\psi$ is fixed by~$\alpha^*$, once more by
Lemma~\ref{SemisimpleCharactersLemma}(b)(iii). The assertion now follows from 
\cite[Lemma~$5.4.1$]{HL1}, applied to the group ${\overline{G}^*}^{\sigma}$ and 
the triple $(V,gn,\alpha)$.
\end{proof}

\section{Bounds on orders of centralizers and elements}

In this section we establish bounds on the orders of certain centralizers and 
elements. The results are enough to rule out $G = \P\Omega_8^+(q)$ as a minimal 
counterexample. 

\subsection{Bounds on orders of centralizers and elements}
\label{CentralizersOfAutomorphisms}
We need to estimate the sizes of the centralizers of certain automorphisms of 
the groups~$G$ listed in Hypothesis~\ref{GroupsOfEvenCharacteristic}(b). We 
will use the notations introduced and summarized in 
Subsections~\ref{TheRemainingGroups}. Recall in particular the $\sigma$-setup 
for~$G$, the notion $q = 2^f$, and the significance of the symbols~$\varphi$ 
and~$\iota$, the latter denoting a non-trivial element 
of~$\Gamma_{\overline{G}}$, as well as its restriction to~$G$. Also recall the 
$\varepsilon$-convention for the linear and unitary groups as well as for the 
groups of type~$E_6$. We also define the parameter $\delta \in \{ 1, 2 \}$ by 
$\delta := 1$, if $\varepsilon = 1$, and $\delta := 2$, if $\varepsilon = -1$.

We begin with a definition.

\addtocounter{num}{1}
\begin{defn}
\label{DefineMG}
{\rm 
With~$G$ as in {\rm Hypothesis~\ref{GroupsOfEvenCharacteristic}(b)}, define the 
positive integer $M_G$ as follows:

{(a)} If $G = E_6^{\varepsilon}(q)$, then $M_G := q^{48}$.

{(b)} If $G = \PSL_d^{\varepsilon}(q)$ with $d \geq 5$, then
$M_G := q^{d(d+1)/2}$.

{(c)} If $G = \PSL_3^{\varepsilon}(q)$, then 
$M_G := 
\begin{cases}
q^4, & \text{\ if } \varepsilon = 1 \text{\ and\ } q \neq 16, \\
62\,401, & \text{\ if } \varepsilon = 1 \text{\ and\ } q = 16, \\
q^4 + q^3, & \text{\ if } \varepsilon = -1.
\end{cases}
$

{\rm (d)} If $G = \P\Omega_8^+(q)$, then $M_G := q^{14} + q^{12}$.
}
\end{defn}

\begin{lem}
\label{CentralizersEstimates}
Let~$p$ be a prime and let $\alpha \in \Aut(G) \setminus \Inndiag(G)$ with 
$|\alpha| = p$. If $G = \P\Omega_8^+(q)$ assume that $q \neq 2$. 
Then $|C_G( \alpha )| < M_G$, unless $G = E_6^{\varepsilon}(q)$ or 
$G = \P\Omega_8^+(q)$ and~$\alpha$ is a graph automorphism of~$G$ of 
order~$2$. In the former case, $|C_G( \alpha )| < q^{52}$, in the latter 
case, $|C_G( \alpha )| < q^{21}$.
\end{lem}

\begin{proof}
Since~$|\alpha| = p$ and $\alpha \not\in \Inndiag(G)$, the order of~$\alpha$ 
modulo $\Inndiag(G)$ also equals~$p$. It follows that $\alpha = \ad_g \circ \mu$ 
for some $g \in \overline{G}^{\sigma}$ and some $\mu \in \Gamma_G \times \Phi_G$
of order~$p$. Recall that $\Gamma_G = \{ 1 \}$ if~$G$ is twisted. Write 
$\mu = \iota' \circ \varphi'$ with $\iota' \in \Gamma_G$ and
$\varphi' \in \Phi_G$. As $|\mu| = p$ is a prime, we must have
$|\varphi'| = p$ or $|\iota'| = p$. In the latter case,~$G$ is untwisted and
$p = 2$ or $p = 3$, where $p = 3$ only occurs for $G = \P\Omega_8^+(q)$.

Suppose first that~$\alpha$ is a graph automorphism of~$G$. If~$G$ is untwisted,
this means that $\varphi' = 1$, and either $p = 2$, or $p = 3$ and 
$G = \P\Omega_8^+(q)$. If~$G$ is twisted, this means that $\iota' = 1$ and 
$p = 2$. If $p = 2 = r$, then 
$C_G( \alpha ) \cong F_4(q), \Sp_{2\lfloor d/2 \rfloor}(q), \Sp_2(q), \SO_7(q)$ 
in the respective cases of Definition~\ref{DefineMG}; see 
\cite[Proposition~$4.2.9$]{GLS}. If $p = 3$, we have 
$C_G( \alpha ) \cong G_2(q)$ by \cite[Table~$4.7.3$A]{GLS}. Notice that the 
latter reference only gives $O^{2'}( C_G( \alpha ) )$, but according to 
\cite[Table~$8.50$]{BHRD}, the only maximal subgroups of~$G$ containing a 
subgroup isomorphic to~$G_2(q)$ are~$G_2(q)$ itself or~$\Sp_6(q)$. 
However,~$G_2(q)$ is the only maximal subgroup of $\Sp_6(q)$ 
containing~$G_2(q)$; see \cite[Tables~$8.28$,~$8.29$]{BHRD}. Thus 
$O^{2'}( C_G( \alpha ) ) = C_G( \alpha )$. In any case, the given upper bound 
for $|C_G( \alpha )|$ is satisfied.

Now suppose that~$\alpha$ is not a graph automorphism. Then $\varphi' \neq 1$ so
that $|\varphi'| = p$. Recall that $\Phi_G$ is cyclic of order~$f$ if~$G$ is
untwisted, and of order~$2f$, if~$G$ is twisted. As~$p$ is odd in the twisted 
case, we have $p \mid f$ in any case. If~$G$ is twisted,~$\iota$ is an 
involution in~$\Phi_G$. Then~$\varphi^{f/p}$ in the untwisted case, and
$\iota \circ \varphi^{f/p}$ in the twisted case, are elements of order~$p$
in~$\Phi_G$. Clearly, $C_G( \alpha ) = C_G( \alpha^l )$ for all integers~$l$ 
prime to~$p$. By Equation~(\ref{PowerFormula}), we may thus assume that 
${\varphi'} = \varphi^{f/p}$, respectively 
${\varphi'} = \iota \circ \varphi^{f/p}$. Recall that $\varphi$ and~$\iota$ 
arise from restrictions of corresponding elements of $\Aut_1( \overline{G} )$, 
denoted by the same letters. Viewed as such, $\mu = \iota' \circ \varphi'$ and
$\alpha = \ad_g \circ \mu$ are Steinberg morphisms of~$\overline{G}$. By our 
choice of~$\varphi'$, we have $\sigma = {\varphi'}^p$. Hence $\mu^p = \sigma$, 
as $|\iota'| \in \{ 1, p \}$. 
It follows that $C_{\overline{G}^{\sigma}}( \mu ) = C_{\overline{G}} ( \mu )$.
Now $C_{G}( \alpha ) \leq C_{\overline{G}} ( \alpha ) \cong 
C_{\overline{G}} ( \mu ) = C_{\overline{G}^{\sigma}}( \mu )$;
for the isomorphism in the latter chain see \cite[Lemma~$1.4.14$]{GeMa}. 
In particular, $|C_{G}( \alpha )| \leq |C_{\overline{G}} ( \mu )| = 
|\overline{G}^{\mu}|$. The groups of maximal order among 
the~$\overline{G}^{\mu}$ occur for~$G$ untwisted and $p = 2$. These are, in the
respective cases, ${^2\!E}_6(q_0)$, $\PGU_d( q_0 )$, $\PGU_3( q_0 )$ and 
$\P\Omega_8^-( q_0 )$, where $q = q_0^2$. It is easy to check that the orders of
these groups satisfy the asserted bounds. Notice that $|\PGU_3(4)| = 62\,400$.
\end{proof}

We will also need the following technical result on the orders of some
elements of $\Aut(G)$. 

\begin{lem}
\label{OrderBound}
Suppose that~$G$ is as in
{\rm Hypothesis~\ref{GroupsOfEvenCharacteristic}(b)(i)--(iv)}.
Let $\beta = \ad_t \circ \mu \in \Aut(G)$ with $t \in \overline{T}^{\sigma}$
and $\mu \in \Gamma_G \times \Phi_G$. Suppose that $3 \mid q - \varepsilon$.
Then the following hold.

{\rm (a)} If $t = 1$, then $|\beta|$ divides $\delta f$.

{\rm (b)} If~$t$ is a $3$-element, then $|\beta|$ divides $3 \delta f$.

{\rm (c)} If $\varepsilon = -1$, $|\mu|$ is even and $t^{q+1} = 1$,
then $|\beta|$ divides $2 f$.
\end{lem}
\begin{proof}
(a) This is clear, as $\Gamma_G \times \Phi_G$ has exponent~$\delta f$.

(b) Put $c := \delta f/f_3$. If $\delta = 1$, then~$f$ is even as then
$3 \mid 2^f - 1$. Thus~$c$ is even. Also, $|\mu^c|$ divides $f_3$. Hence
$\mu^c = \varphi^{c'}$ for some even integer~$c'$. By
Equation~(\ref{PowerFormula}), we have $\beta^c = \ad_x \circ \mu^c$ with a
$3$-element $x \in \overline{T}^{\sigma}$. To prove our claim, it suffices to
show that $\beta^{3cf_3} = 1$.

Suppose that $f_3 := 3^b$ for some non-negative integer~$b$. Then
$(q - \varepsilon)_3 = 3^{b+1}$ by \cite[Lemma IX.$8.1$(e)]{HuBII}.
Observe that~$\varphi$ acts on~$\overline{T}$ by squaring the elements; see
\cite[Theorem~$1.12.1$]{GLS}. Put $q' := 2^{c'}$. Once more by
Equation~(\ref{PowerFormula}), we find that
$\beta^{3c} = \ad_{x'} \circ \mu^{3c}$ with $x' = x^{1+q'+{q'}^2}$. As~$q'$ is
an even power of~$2$, we have $3 \mid 1+q'+{q'}^2$. Hence $|x'| = |x|/3$, unless
$|x| = 1$, and $|\mu^{3c}| = |\mu^{c}|/3$, unless $|\mu^{c}|_3 = 1$. Iterating
this argument, we find that $\beta^{cf_3} = \ad_y \circ \id$ with
$y \in \overline{T}^{\sigma}$ of order~$1$ or~$3$. This completes the proof
of~(a).

(c) Since $\varepsilon = -1$, we have $\Gamma_G = \{ \id \}$, and $\Phi_G$ is
cyclic of order~$2f$. Moreover,~$f$ is odd, as $3 \mid 2^f + 1$. As $|\mu|$ is
even, Equation~(\ref{PowerFormula}) implies that
$\beta^f = \ad_x \circ \varphi^f$ for some $x \in \overline{T}^{\sigma}$ with
$x^{q+1} = 1$. From Equation~(\ref{PowerFormula}) we get $\beta^{2f} = 
\ad_{x^{1+q}} \circ \id = \id$. This yields our assertion.
\end{proof}

\addtocounter{subsection}{3}
\subsection{The proof for the eight-dimensional orthogonal group}
As an illustration of Lemma~\ref{CentralizersEstimates}, we complete the proof 
that the groups $\P\Omega_8^+(q)$ have the $E1$-property.

\addtocounter{num}{1}
\begin{prp}
\label{PO8Plus}
Let $G = \P\Omega_8^+(q)$. Assume that $\nu = \ad_h \circ \iota \circ \mu$,
where $h \in G$, $\iota \in \Gamma_G$ of order~$3$ and $\mu \in \Phi_G$.
Then $(G,V,n)$ has the $E1$-property.
\end{prp}
\begin{proof}
By pre-multiplying~$n$ with~$h^{-1}$, we may assume that 
$\nu = \iota \circ \mu$. If~$|\nu|$ is even, we put $g := 1$. If~$|\nu|$ is odd, 
we let $g \in \P\Omega_8^+(2) \leq G$ be a $\iota$-stable involution, whose
centralizer in $\P\Omega_8^+(2)$ has order $2^{10}3$. Using the fact that the
centralizer of~$\iota$ in $\P\Omega_8^+(2)$ is isomorphic to $G_2(2)$, the 
corresponding class fusion (computed with GAP) shows the existence of such 
a~$g$. By construction,~$\nu$ fixes~$g$. The table of unipotent characters 
of~$G$ in Chevie~\cite{chevie} contains, in particular, the orders of the 
centralizers of the unipotent conjugacy classes of~$G$. An inspection of these 
entries shows that $|C_G( g )| = q^{10}(q^2 - 1)$. Put 
$\alpha := \ad_g \circ \nu$. Assume that $q \neq 2$. Then 
$|C_G( \alpha_{(p)} )| \leq q^{14} + q^{12}$ for every prime~$p$ 
dividing~$|\alpha|$. For odd~$p$ or even~$|\nu|$ this follows from
Lemma~\ref{CentralizersEstimates}, as $\alpha$ is not a graph automorphism of 
order~$2$. For odd~$|\nu|$ we have $C_G( \alpha_{(2)} ) = C_G( g )$.

By \cite[Lemma~$4.3.3$]{HL1} the triple $(G,V,n)$ has the $E1$-property, if
\begin{equation}
\label{POmegaDegrees}
\chi(1) > (|\alpha| - 1)(q^{14}+q^{12})^{1/2}. 
\end{equation}
Suppose first that $\chi(1) > q^9/8$. If $f \geq 5$, then 
$|\alpha| - 1 \leq 6f - 1 \leq 2^f = q$. 
For $f = 4$ we have $|\alpha| \leq 12 \leq 2^4 = q$. These bounds clearly imply 
that~(\ref{POmegaDegrees}) is satisfied for $f \geq 4$. For $f = 2$ or~$3$, we
have $|\alpha| = 6$. Using this and the exact values for $\chi(1)$ given 
in~\cite{LL2}, we see that~(\ref{POmegaDegrees}) also holds for $f = 2, 3$.
We are left with the case $q = 2$, i.e.\ $G = \P\Omega_8^+(2)$. Here,~$\Phi_G$
is trivial and thus $\alpha = \ad_g \circ \iota$, where~$g$ is a
$\iota$-stable involution with centralizer of order $2^{10}3$. Let
$G^{\ex} := \langle \Inn(G) , \alpha \rangle = \langle \Inn(G) , \nu \rangle$ and 
let~$\chi^{\ex} $ denote the character of~$G^{\ex} $ afforded by~$V$; see 
\cite[Remark~$4.2.5$]{HL1}. Now~$G^{\ex} $ is a group of shape~$G.3$. The character 
table of~$G.3$, and thus of~$G^{\ex} $ is contained in the Atlas~\cite{Atlas} and in 
Gap~\cite{GAP04}. It is easy to locate the conjugacy classes of~$G^{\ex} $ containing 
the elements of $\langle \alpha \rangle$. Assume that~$\chi$ does not 
satisfy~(\ref{POmegaDegrees}). Using Gap, we find that then either 
$\chi(1) = 175$ or $\chi(1) = 525$. In either case, we check that
$\Res^{G^{\ex} }_{\langle \alpha \rangle}( \chi^{\ex} )$ contains each of the real 
irreducible characters of~$\langle \alpha \rangle$ as constituents. Hence 
$(G,V,n)$ has the $E1$-property by \cite[Lemma~$4.3.1$]{HL1}. This completes the 
proof in case $\chi(1) > q^9/8$.

Now assume that $\chi(1) \leq q^9/8$. Using~\cite{LL2}, we find that~$\chi$ lies 
in one of two series of irreducible characters of~$G$ of degrees 
$(q-1)^2(q^2+q+1)(q^2+1)$ and $(q+1)^2(q^2+1)(q^2-q+1)$. We claim that~$\chi$ is 
not invariant under~$\nu$. This would contradict the fact that~$\chi$ extends to 
$G^{\ex}  = \langle \Inn(G), \nu \rangle$; see \cite[Remark~$4.2.5$]{HL1}.
Using~\cite{LL}, it is easy to see that $\chi \in \mathcal{E}( G, s )$, 
where~$s$ is a semisimple element of~$G^*$ with 
$C_{G^*}( s ) \cong \GU_4(q)$ or $C_{G^*}( s ) \cong \GL_4(q)$. In fact, we may 
assume that $C_{\overline{G}^*}( s )$ is a standard Levi subgroup of 
$\overline{G}^*$ of type~$A_3$. There are three such Levi subgroups, which are 
mutually non-conjugate in~$\overline{G}^*$. Put $\nu^* = \iota^* \circ \mu^*$, 
where~$\iota^*$ and~$\mu^*$ denote the standard graph automorphism 
of~$\overline{G}^*$ of order~$3$, dual to~$\iota$, and the standard field 
automorphism of~$\overline{G}^*$ dual to~$\mu$, respectively. Then~$\nu^*$ is an 
isogeny of $\overline{G}^*$ dual to~$\nu$. Notice that~$\iota^*$ permutes the 
three standard Levi subgroups of $\overline{G}^*$ of type~$A_3$ cyclically. On 
the other hand, each standard Levi subgroup is fixed by~$\varphi^*$. Hence 
$C_{\overline{G}^*}( \nu^*( s ) )$ is not conjugate in~$\overline{G}^*$ to 
$C_{\overline{G}^*}( s )$. It follows that~$\nu^*( s )$ and~$s$ are not 
$G^*$-conjugate. As ${^\nu\!\chi} \in \mathcal{E}( G, {\nu^*}^{-1}(s) )$ by 
\cite[Proposition~$7.2$]{JTayl}, it follows that ${^\nu\!\chi} \neq \chi$, as 
claimed.
\end{proof}

\section{The linear and unitary groups}
\label{RemainingLinearUnitary}

In this section, let~$G$ be one of the groups of 
{\rm Hypothesis \ref{GroupsOfEvenCharacteristic}(b)(i),(ii)}. Thus  
$G = \PSL_d^{\varepsilon}(q)$ with $d \geq 3$ and $q = 2^f$, where $q \neq 2$ if
$d = 3$ and $\varepsilon = -1$. Put $e := \gcd( d, q - \varepsilon)$ and recall
that $e > 1$ by hypothesis. Also recall the definition of the parameter~$\delta$
from the introduction to Subsection~\ref{CentralizersOfAutomorphisms}. We will 
further use the notations and results introduced in 
Subsections~\ref{TheRemainingGroups}--\ref{AutomorphismsIII}.  

\subsection{On the natural representation of the linear and unitary groups}
\label{LinearAndUnitaryNaturalRepresentation}
We extend our $\varepsilon$-convention to the general linear and unitary groups
in the obvious way. We will need some results on semisimple elements 
of~$\GL_d^\varepsilon(q)$ acting on its natural vector space. In view of 
Lusztig's Jordan decomposition of characters, the description of real 
semisimple elements is relevant for the classification of real irreducible 
characters.

Let $E := {\mathbb{F}_{q^\delta}}^d$ denote the natural column vector space 
for~$\GL^\varepsilon_d(q)$. If $\varepsilon =  -1$, the group
$\GU_d( q ) \leq \GL_d(q^2)$ is the stabilizer of the Hermitian form
\begin{equation}
\label{HermitianForm}
( ( x_1, \ldots , x_d )^t, ( x_1, \ldots , x_d )^t )
\mapsto \sum_{i = 1}^d x_i y_{d-i+1}^q
\end{equation}
on~$E$. A subspace $E_1$ is totally isotropic, if the hermitian 
form~(\ref{HermitianForm}) vanishes on $E_1 \times E_1$. A pair 
$(E_1,E_1^{\dagger})$ of totally isotropic subspaces of~$E$ is called a 
complementary pair, if $E_1 \oplus E_1^{\dagger}$ is non-degenerate (which 
implies that $\dim(E_1) = \dim(E_1^{\dagger})$).

Let~$\Delta$ be a monic irreducible polynomial over~$\mathbb{F}_{q^\delta}$. We 
write $\Delta^*$ for the monic polynomial whose roots in~$\mathbb{F}$ are the 
inverses of the roots of~$\Delta$. Then~$\Delta^*$ is an irreducible polynomial 
over~$\mathbb{F}_{q^\delta}$. If $\delta = 2$, we write $\Delta^\dagger$ for the 
monic polynomial whose roots are the $-q$ths powers of the roots of~$\Delta$. 
Then~$\Delta^\dagger$ is an irreducible polynomial over~$\mathbb{F}_{q^2}$. The 
notions~$\Delta^*$ and~$\Delta^{\dagger}$ are extended to all monic polynomials 
over~$\mathbb{F}_{q^{\delta}}$ by multiplicativity. We collect a few formal 
properties of these notions.

\addtocounter{num}{1}
\begin{lem}
\label{FormalitiesOnPolynomialsI}
Let~$\Delta$ be a monic irreducible polynomial over~$\mathbb{F}_{q^\delta}$ of
degree~$k$. Then the following hold.

{\rm (a)} If $\Delta = \Delta^*$, then~$k$ is even, unless $k = 1$ and~$1$ is
the root of~$\Delta$.

{\rm (b)} If $\delta = 2$ and $\Delta = \Delta^\dagger$, then~$k$ is odd.
\end{lem}
\begin{proof}
(a) Let~$\zeta \in \mathbb{F}$ be a root of~$\Delta$. As~$q$ is even, 
$\zeta \neq \zeta^{-1}$, unless $\zeta = 1$.

(b) Suppose that $\Delta = \Delta^\dagger$ and let $\zeta \in \mathbb{F}$ be a
root of~$\Delta$. As~$\Delta$ is irreducible over~$\mathbb{F}_{q^2}$, the roots
of~$\Delta$ are $\zeta, \zeta^{q^2}, \ldots , \zeta^{q^{2(k-1)}}$, and~$k$ is the
smallest positive integer such that $\zeta^{q^{2k}} = \zeta$.

As $\Delta = \Delta^\dagger$, there is $0 \leq i \leq k - 1$ such that
$\zeta^{-q} = \zeta^{q^{2i}}$. Hence $\zeta^{q(q^{2i-1} + 1)} = 1$. This implies
$\zeta^{q^{2i-1} + 1} = 1$ and thus $\zeta^{q^{2(2i-1)} - 1} = 1$. It follows that
$\zeta^{q^{2(2i-1)}} = \zeta$ which implies that~$k$ is odd.
\end{proof}

\begin{lem}
\label{FormalitiesOnPolynomialsII}
Let $\hat{s}, \hat{s}' \in \GL_d( q^{\delta} )$ be semisimple with characteristic
polynomial~$\Xi$ and~$\Xi'$, respectively.

{\rm (a)} The elements~$\hat{s}$ and~$\hat{s}'$ are conjugate 
in~$\GL_d( q^{\delta} )$, if and only if $\Xi = \Xi'$.

{\rm (b)} The element $\hat{s}$  is real, if and only if $\Xi = \Xi^*$.

{\rm (c)} Suppose that $\delta = 2$ and that $\hat{s} \in \GU_d(q)$.
Then $\Xi = \Xi^{\dagger}$.

{\rm (d)} Suppose that $\delta = 2$ and that $\hat{s}, \hat{s}' \in \GU_d(q)$. 
Then~$\hat{s}$ and~$\hat{s}'$ are conjugate 
in~$\GU_d(q)$, if and only if they are conjugate in~$\GL_d(q^2)$.

\end{lem}
\begin{proof}
These assertions are well known.
\end{proof}

\begin{lem}
\label{dCongruentd1}
Let $\hat{s} \in \GL_d( q^{\delta} )$ be semisimple and real and let~$d_1$ 
denote the dimension of the eigenspace of~$\hat{s}$ for the eigenvalue~$1$. 
Then $d \equiv d_1\,\,(\mbox{\rm mod}\,\,2)$.
\end{lem}
\begin{proof}
The eigenvalues of~$\hat{s}$ different from~$1$ come in pairs of mutually
inverse elements.
\end{proof}

\begin{defn}
{\rm
Let~$\hat{s} \in \GL^{\varepsilon}_d(q)$ be semisimple, and let~$\Theta$ be a 
monic factor of the minimal polynomial of~$\hat{s}$. We then put 
$E_{\Theta}(\hat{s}) := \ker( \Theta(\hat{s}) )$.
}
\end{defn}

\begin{lem}
\label{FormalitiesOnPolynomialsIII}
Put $\hat{G} := \GL_d^{\varepsilon}(q)$. Let~$\hat{s} \in \hat{G}$ be semisimple 
with characteristic polynomial~$\Xi$, and let~$\Delta$ be a monic irreducible 
factor of~$\Xi$ of multiplicity~$m$ and degree~$k$.

{\rm (a)} The centralizer $C_{\hat{G}}( \hat{s} )$ 
stabilizes~$E_{\Theta}(\hat{s})$ for every monic factor~$\Theta$ of the minimal
polynomial of~$\hat{s}$. In particular, $C_{\hat{G}}( \hat{s} )$
stabilizes~$E_{\Delta}( \hat{s} )$. If $\delta = 1$, then $C_{\hat{G}}( \hat{s} )$
induces $\GL_m( q^k )$ on~$E_{\Delta}( \hat{s} )$.

{\rm (b)} Suppose that $\delta = 2$ and let~$\Delta'$ be a monic irreducible 
factor of~$\Xi$ with $\Delta' \neq \Delta^{\dagger}$. Then 
$E_{\Delta'}(\hat{s})$ and $E_{\Delta}(\hat{s})$ are orthogonal.

{\rm (c)} Suppose that $\delta = 2$ and that $\Delta = \Delta^\dagger$. 
Then~$E_{\Delta}(\hat{s})$ is non-degenerate and~$C_{\GU_d(q)}( \hat{s} )$ 
induces~$\GU_m(q^k)$ on~$E_{\Delta}(\hat{s})$.

{\rm (d)} Suppose that $\delta = 2$ and that $\Delta \neq \Delta^\dagger$.
Then~$E_{\Delta}(\hat{s})$ is totally isotropic. In this case, 
$E_{\Delta}( \hat{s} ) \oplus E_{\Delta^{\dagger}}(\hat{s})$ is non-degenerate 
and~$C_{\GU_d(q)}( \hat{s} )$ induces $\GL_m( q^{2k} )$ on 
$E_{\Delta}( \hat{s} ) \oplus E_{\Delta^{\dagger}}(\hat{s})$.
\end{lem}
\begin{proof}
These assertions are well known and easily proved; see e.g.\
\cite[Proposition~(1A)]{fs2}. 
\end{proof}

\begin{lem}
\label{FormalitiesOnPolynomialsIV}
Let $\hat{G} = \GU_d(q)$ for some $d \geq 2$. Let $1 \neq \hat{s} \in \hat{G}$ 
be semisimple and real with characteristic polynomial~$\Xi$. Let~$\Delta_1$ 
denote the monic polynomial of degree~$1$ with root~$1$, and let~$d_1$ be the 
multiplicity 
of~$\Delta_1$ in~$\Xi$. Let $\Delta \neq \Delta_1$ denote a monic irreducible
factor of~$\Xi$ of degree~$k$ and multiplicity~$m$. Then the following hold.

{\rm (a)} If $\Delta = \Delta^{\dagger}$, then $\Delta \neq \Delta^*$, 
and~$\Delta^*$ occurs with multiplicity~$m$ in~$\Xi$. 

{\rm (b)} If $\Delta \neq \Delta^{\dagger}$, then~$\Delta^{\dagger}$ occurs 
with multiplicity~$m$ in~$\Xi$. 

{\rm (c)} We have $d \geq d_1 + 2mk$.

{\rm (d)} We have 
$|C_{\GU_d(q)}( \hat{s} )|_{2'} \leq |\GU_{m}( q^k )|_{2'}^2|\GU_{d - 2mk}(q)|_{2'}$.

{\rm (e)} Let $d'$ denote the greatest multiplicity of an irreducible factor 
of~$\Xi$. If $d' = 1$, then $|C_{\GU_d(q)}( \hat{s} )|_{2'} \leq (q+1)^d$.
If $d' = 2$, then $|C_{\GU_d(q)}( \hat{s} )|_{2'} \leq 
|\GU_2(q)|_{2'}^{\lfloor d/2 \rfloor}(q+1)^{\overline{d}}$ with 
$\overline{d} = d - 2\lfloor d/2 \rfloor$.
\end{lem}
\begin{proof}
The assertion in~(c) follows from those in~(a) and~(b). To prove~(a), 
observe that~$k$ is odd if $\Delta = \Delta^{\dagger}$; see 
Lemma~\ref{FormalitiesOnPolynomialsI}(b). 
Lemma \ref{FormalitiesOnPolynomialsI}(a) then gives
$\Delta \neq \Delta^*$. The remaining assertion follows from 
Lemma~\ref{FormalitiesOnPolynomialsII}(b). 
The proof for~(b) is analogous.

Let us now prove~(d) and~(e). If $\Delta = \Delta^{\dagger}$, then 
$\Delta \neq \Delta^*$ by~(b) and we put $\Delta^\circ := \Delta^*$. If
$\Delta \neq \Delta^{\dagger}$, we put $\Delta^\circ := \Delta^{\dagger}$. Then
$E_1 := E_{\Delta}( \hat{s} ) \oplus E_{\Delta^\circ}( \hat{s} )$ is 
non-degenerate by Lemma~\ref{FormalitiesOnPolynomialsIII}(c)(d), and 
$C_{\GU_d(q)}( \hat{s} )$ induces $\GU_m( q^k ) \times \GU_m( q^k )$, 
respectively $\GL_m( q^{2k} )$ on~$E_1$. Let~$E_2$ denote the orthogonal 
complement of~$E_1$. Then $C_{\GU_d(q)}( \hat{s} )$ fixes~$E_2$ by
Lemma \ref{FormalitiesOnPolynomialsIII}(a)(b). Write 
$\hat{s} = \hat{s}_1 + \hat{s}_2$, where $\hat{s}_j$ is the projection
of~$\hat{s}$ to~$E_j$, $j = 1, 2$. We then have, with a slight abuse of 
notation, $C_{\GU_d(q)}( \hat{s} ) = C_{\GU_{2mk}(q)}( \hat{s}_1 ) \times
C_{\GU_{d - 2mk}(q)}( \hat{s}_2 )$. The claim in~(d) follows from
$|\GL_{m}( q^{2k} )|_{2'} < |\GU_{m}( q^k )|_{2'}^2$, which is easily proved by 
induction on~$m$. To see~(e), first assume that~$d' = 1$. Clearly, 
$|\GU_{1}( q^k )|_{2'} = q^k + 1 \leq (q + 1)^k$, and so the claim follows
by induction on~$d$. If $d' = 2$, suppose first that $d' = d_1$ and that all
irreducible factors of~$\Xi$ different from~$\Delta_1$ occur with 
multiplicity~$1$. Applying the first part of the statement to the orthogonal 
complement of~$E_{\Delta_1}( \hat{s} )$, we obtain 
$|C_{\GU_d(q)}( \hat{s} )|_{2'} \leq |\GU_2(q)|_{2'}(q+1)^{d-2}$, and thus, as
$(q+1)^2 \leq |\GU_{2}( q )|_{2'}$, our assertion.
If some irreducible factor of~$\Xi$ different from~$\Delta_1$ occurs with
multiplicity~$2$, we can choose~$\Delta$ such that $m = 2$. Now
$|\GU_{2}( q^k )|_{2'} = (q^k + 1)(q^{2k}-1) \leq (q+1)^k(q^2-1)^k = 
|\GU_{2}( q )|_{2'}^k$, where the middle inequality follows by induction 
on~$k$. We are done by induction on~$d$.
\end{proof}

Let $g \in \PGL_d^\varepsilon(q)$. An inverse image 
$\hat{g} \in \GL_d^\varepsilon(q)$ under the canonical epimorphism 
$\GL_d^\varepsilon(q) \rightarrow \PGL_d^\varepsilon(q)$ will be called a lift 
of~$g$.

\begin{lem}
\label{SemisimpleElementsInLeviSubgroups}
Let $s \in \PGL_d^\varepsilon(q)$ be semisimple and real. Then there exists a 
real lift $\hat{s} \in \GL_d^\varepsilon(q)$ of~$s$.
\end{lem}
\begin{proof}
Let~$y \in \PGL_d^\varepsilon(q)$ with $ysy^{-1} = s^{-1}$, and let 
$\hat{s}, \hat{y} \in \GL_d^\varepsilon(q)$ denote lifts of~$s$ and~$y$, 
respectively. Then there is~$\zeta \in \mathbb{F}_{q^\delta}^*$ such that 
$\hat{y}\hat{s}\hat{y}^{-1} = \zeta\hat{s}^{-1}$. Moreover, $\zeta^{q+1} = 1$ if 
$\delta = 2$. Let~$\xi \in \mathbb{F}_{q^\delta}^*$ with $\xi^{-2} = \zeta$. 
(Recall that~$q$ is even.) Thus $\xi\hat{s} \in \GL_d^\varepsilon(q)$ is a lift 
of~$s$ and $\hat{y} (\xi \hat{s}) \hat{y}^{-1} = (\xi \hat{s})^{-1}$.
\end{proof}

\addtocounter{subsection}{7}
\subsection{Automorphisms, II}
\label{AutomorphismsIV}
Here, we extend the notations introduced in Subsections~\ref{AutomorphismsIII} 
in case of $\overline{G} = \PGL_d( \mathbb{F} )$. Writing 
$\overline{\hat{G}} := \GL_d( \mathbb{F} )$, we get 
$\overline{G} = \overline{\hat{G}}/Z( \overline{\hat{G}} )$. This allows a 
convenient descriptions of $\Aut_1(\overline{G})$ using the natural matrix 
representations of~$\overline{\hat{G}}$. Let $\hat{\varphi}$ denote the 
Steinberg morphism of~$\overline{\hat{G}}$ which squares every matrix entry,
and let $\hat{\iota}$ denote the automorphism of~$\overline{\hat{G}}$, defined 
as the inverse-transpose automorphism, followed by conjugation with the matrix 
with~$1$s along the anti-diagonal, and~$0$s, elsewhere. Then~$\hat{\iota}$ is an 
involution that commutes with~$\hat{\varphi}$. Let $\Gamma_{\overline{\hat{G}}} 
:= \langle \hat{\iota} \rangle \leq \Aut( \overline{\hat{G}} )$ and
$\Phi_{\overline{\hat{G}}} := 
\langle \hat{\varphi} \rangle \leq \Aut( \overline{\hat{G}} )$, where
$\Aut( \overline{\hat{G}} )$ denotes the automorphism group
of~$\overline{\hat{G}}$ as an abstract group. Then
$\Phi_{\overline{\hat{G}}} \times \Gamma_{\overline{\hat{G}}} 
\leq \Aut( \overline{\hat{G}} )$
and $(\Gamma_{\overline{\hat{G}}} \times \Phi_{\overline{\hat{G}}}) \cap
\Inn( \overline{\hat{G}} )$ is trivial. 
We put
$$\Aut'( \overline{\hat{G}} ) := \Inn( \overline{\hat{G}} ) \rtimes
(\Gamma_{\overline{\hat{G}}} \Phi_{\overline{\hat{G}}}).$$
Write $\iota$ and $\varphi$ for the automorphisms of~$\overline{G}$ induced 
by~$\hat{\iota}$, respectively~$\hat{\varphi}$. Then~$\iota$ and~$\varphi$ are 
the standard graph automorphism, respectively the standard Frobenius morphism 
of~$\overline{G}$ as in Subsection~\ref{TheRemainingGroups}.

The natural homomorphism
\begin{equation}
\label{Aut1HatGToAut1G}
\Aut'( \overline{\hat{G}} ) \rightarrow \Aut_1( \overline{G} )
\end{equation}
is in fact and isomorphism.
For $\mu \in \Aut_1( \overline{G} )$ we write $\hat{\mu}$ for its inverse image 
in $\Aut'( \overline{\hat{G}} )$ under~(\ref{Aut1HatGToAut1G}).

\subsection{Further preliminary results}
We collect more preliminary results.

\addtocounter{num}{2}
\begin{lem}
\label{GUdPrepPrep}
{\rm (a)} Some odd prime divides~$d$, and if $\varepsilon = -1$, then some
odd prime different from~$7$ divides~$d$.

{\rm (b)} If $q \geq 8$, then $q^3 \geq 2 f (q+1)^2$ and
$1 - q^{-1} - q^{-2} \geq q^{-1/4}$.

{\rm (c)} If $q = 4$, we have $q^4 > 2 f (q+1)^2$
and $1 - q^{-1} - q^{-2} \geq q^{-1/2}$.

{\rm (d)} If $q = 2$, we have $q^5 > 2 f (q+1)^2$
and $1 - q^{-1} - q^{-2} \geq q^{-2}$.

{\rm (e)} We have $|\GL_d(q)| \leq q^{d^2}$ and
$|\GU_d(q)| \leq (1 - q^{-1} -q^{-2})^{-1}q^{d^2}$. If $q \geq 4$, then
$|\GU_d(q)| \leq q^{d^2 + 1/2}$, and if $q = 2$, then
$|\GU_d(q)| \leq q^{d^2 + 2}$.
\end{lem}
\begin{proof}
(a) The first assertion is clear since $\gcd( d, q - \varepsilon ) > 1$ and
$q - \varepsilon$ is odd. Also, $7 \nmid q + 1$, as
$2^f\!\!\mod 7 \in \{ 1, 2, 4 \}$. This yields the second assertion.

(c)--(d) These assertions are trivially verified.

(e) This follows from the known order formula for $\GL_d^{\varepsilon}(q)$,
together with Lemma~\ref{OrderEstimates} and the estimates in (b)--(d).
\end{proof}

\begin{lem}
\label{CentralizerOrderComparison}
Let $t \in \overline{G}^{\sigma} = \PGL^{\varepsilon}_d(q)$ be semisimple
and let $\hat{t} \in \GL_d^{\varepsilon}(q)$ denote a lift of~$t$. Then
$|C_{\overline{G}^{\sigma}}(t)| \leq |C_{\GL_d^{\varepsilon}(q)}(\hat{t})|$.
\end{lem}
\begin{proof}
Put $C := C_{\GL_d^{\varepsilon}(q)}(\hat{t})$ and let $\tilde{C}$ denote
the inverse image of $C_{\overline{G}^{\sigma}}(t)$ in $\GL_d^{\varepsilon}(q)$.
Then $C \leq \tilde{C}$ and $[\tilde{C}\colon\!C] \leq q - \varepsilon$.
Indeed, the map $\tilde{C} \rightarrow Z( \GL_d^{\varepsilon}(q) )$,
$g \mapsto [g,t]$ is a homomorphism with kernel~$C$. As
$|C_{\overline{G}^{\sigma}}(t)| = |\tilde{C}|/(q - \varepsilon)$, our assertion
follows.
\end{proof}

Let $u \in \GL_d( q^{\delta} )$ be an involution. We say that~$u$ consists
of~$l$ Jordan blocks of size~$2$ if the Jordan normal form of~$u$ has~$l$
Jordan blocks of size $(2 \times 2)$ and $d - 2l$ Jordan blocks of size
$(1 \times 1)$.

\begin{lem}
\label{InvolutionCentralizerInLinearGroup}
Let $u \in \GL^\varepsilon_d( q )$ be an involution consisting of~$l$ Jordan
blocks of size~$2$ and let $C := C_{\GL^\varepsilon_d( q )}( u )$. Then~$C$ is a
semidirect product $C = UL$ with a unipotent radical~$U$ of order $q^{2ld-3l^2}$
and a complement
$L \cong \GL_l^\varepsilon(q) \times \GL^\varepsilon_{d - 2l}( q )$.

In particular, $|C| \leq q^{2l^2 - 2ld + d^2}$ if $\varepsilon = 1$, and
$|C| \leq q^{2l^2 - 2ld + d^2 +4}/(q^2 -q - 1)^2$ if $\varepsilon = -1$.
\end{lem}
\begin{proof}
It is clear that~$u$ is conjugate in~$\GL_d(q^{\delta})$ to the matrix
$$u' := \left[ \begin{array}{ccc} \Id_l & 0 & \Id_l \\
                                   0 & \Id_{d-2l} & 0 \\
                                   0 & 0 & \Id_l 
\end{array} \right]
$$
Observe that $u' \in \GU_d( q )$. Thus~$u$ is conjugate to~$u'$ in $\GU_d(q)$
if $\varepsilon = -1$. We may thus replace~$u$ by $u'$. A routine matrix
calculation then yields our assertion on the structure of~$C$.

The estimates for~$|C|$ follow from Lemma~\ref{OrderEstimates}.
\end{proof}

\addtocounter{subsection}{3}
\subsection{Modification of~$\nu$}
We now prove a crucial result on centralizers of certain automorphisms of~$G$.
The aim is the modification of the automorphism~$\nu$ by an inner automorphism
in order to minimize the centralizer orders of its powers. A priory,~$\nu$ can
be any element of~$\Aut(G)$. Recall Definition~\ref{DefineMG} for the 
quantity~$M_G$.

\addtocounter{num}{1}
\begin{prp}
\label{QEvenPrep}
Let $\beta \in \Aut(G)$. Then there 
is $g \in G$ such that $\alpha := \ad_g \circ \beta$ has even order and the 
following statements hold.

{\rm (a)} We have $|C_G( \alpha_{(p)} )| < M_G$
for every prime~$p$ dividing $|\alpha|$ and every element $\alpha_{(p)}$
of $\langle \alpha \rangle$ of order~$p$.

{\rm (b)} If $G = \PSL_d^{\varepsilon}(q)$ with $d \geq 5$, then~$|\alpha|$
divides $\delta f(q-\varepsilon)$. 

{\rm (c)} If $G = \PSL_3^{\varepsilon}(q)$, then $|\alpha|$ divides 
$3 \delta f$. 

{\rm (d)} For the specified values of $\varepsilon$,~$d$ and~$q$, the order 
of~$\alpha$ and the structure of $G^{\ex}  := \langle \Inn( G ), \beta \rangle$ are as 
given in the following table.
\setlength{\extrarowheight}{0.5ex}
$$\begin{array}{rcrcc} \hline\hline
\varepsilon & d & q & |\alpha| & G^{\ex}  \\ \hline\hline
1 & 3 & 4 & 6 & G.3 \cong \PGL_3(4) \\
1 & 3 & 4 & 6 & G.6 \\
-1 & 3 & 8 &     6 & G.3 \\
-1 & 3 & 8 &     6 & G.6 \\
-1 & 3 & 8 &        18 & G.3 \\
-1 & 3 & 32 & 30 & \PGU_3( 32 ) \cong \langle \Inn( G ), \alpha^{10} \rangle \\
-1 & \geq 5 & 2 & 2, 6 & \\
-1 & \geq 5 & 4 & 2, 4, 10 & \\
-1 & 6 & 2 & 6 & G.3 \cong \PGU_6(2) \\
-1 & 6 & 8 & 2, 6, 18 & \\ \hline\hline
\end{array}
$$
\end{prp}
\begin{proof}
By Lemma~\ref{GUdPrepPrep}(a), we have $d \not\in \{ 4, 8, 16 \}$, and if 
$\varepsilon = -1$, we also have $d \not\in \{ 7, 14 \}$. We represent the 
elements of $\overline{G}^{\sigma} = \PGL^\varepsilon_d(q)$ by elements in 
$\GL^\varepsilon_d(q)$ and make use of the notation introduced in
Subsection~\ref{AutomorphismsIV}. Let $\hat{t} \in \GL^\varepsilon_d(q)$ be a 
diagonal matrix. Then $\hat{\varphi}(\hat{t}) = {\hat{t}}^2$, and 
$\hat{\iota}(\hat{t})$ is obtained from~$\hat{t}$ by reversing the order of its 
diagonal entries and inverting them. In particular, if~$\hat{t}$ is palindromic,
i.e.\ invariant under reversing its diagonal entries, then $\hat{\iota}(\hat{t}) 
= \hat{t}^{-1}$.

Recall that~$\beta$ has a factorization as $\beta = \ad_h \circ \mu$ with
$h \in \PGL^\varepsilon_d(q)$ and $\mu \in \Gamma_G \times \Phi_G$, and that~$h$ 
and~$\mu$ are uniquely determined by~$\beta$. If $d = 3$, we may assume 
that~$|h|$ is a $3$-power by replacing~$h$ with $g'h$ for a suitable $g' \in G$. 
Let $\hat{h} \in \GL^\varepsilon_d(q)$ represent~$h$. If $d = 3$, assume 
that~$|\hat{h}|$ is a $3$-power. Notice that, in any case, $\det(\hat{h})$ 
lies in the subgroup of $\mathbb{F}_{q^{\delta}}^*$ of order $q - \varepsilon$.
We now choose a particular diagonal palindromic element 
$\hat{t} \in \GL^\varepsilon_d(q)$ with $\det(\hat{t}) = \det(\hat{h})$ as 
follows. For $3 \leq d \leq 6$, choose
$$
\hat{t} :=
\begin{cases}
\diag(\zeta,1,\zeta), & \text{\ if\ } d = 3,\\
\diag(\zeta,1,1,1,\zeta), & \text{\ if\ } d = 5,\\
\diag(\zeta,\zeta,1,1,\zeta,\zeta), & \text{\ if\ } d = 6,\\
\diag(\zeta,\zeta,1,1,1,\zeta,\zeta), & \text{\ if\ } d = 7 
\text{\ (and\ $\varepsilon = 1$)},
\end{cases}
$$
where $\zeta \in \mathbb{F}_{q^{\delta}}^*$ is such that $\det(\hat{t}) = 
\det(\hat{h})$. This is possible since~$q$ is even. Notice that 
$\zeta^{q - \varepsilon} = 1$, so that~$\hat{t}$ indeed lies 
in~$\GL^\varepsilon_d(q)$. Notice also that $\hat{t} = 1$ if $h \in G$. Now 
suppose that $d \geq 9$. Define $\bar{d} \in \mathbb{Z}$ by $\bar{d} := 1$, 
if~$d$ is odd, and $\bar{d} := 2$, if~$d$ is even. Put 
$d' := \lfloor (d - \bar{d})/4 \rfloor$, so that $d - \bar{d} = 4d'$ if
$d - \bar{d}$ is divisible by~$4$, and $d - \bar{d} = 4d' + 2$, otherwise.
Choose $\zeta \in \mathbb{F}_{q^{\delta}}^*$ of order $q - \varepsilon$.
Put
$$
\hat{t} :=
\begin{cases}
\diag( \zeta \Id_{d'}, \Id_{d'}, \xi \Id_{\bar{d}}, \Id_{d'}, \zeta \Id_{d'}),
& \text{\ if\ } d - \bar{d} = 4d',\\
\diag( \zeta \Id_{d'}, \Id_{d'}, \zeta^{-1}, \xi \Id_{\bar{d}}, \zeta^{-1}, 
\Id_{d'}, \zeta \Id_{d'}), & \text{\ otherwise}. 
\end{cases}
$$
Here, $\xi$ is chosen as a power of $\zeta$ in such a way that
$\det(\hat{t}) = \det(\hat{h})$. Once more this is possible since~$q$ is even.
As $\hat{t}$ is palindromic, $\hat{\mu}(\hat{t})$ is a power of~$\hat{t}$ by 
what we have said above. In turn, $N_{\hat{\mu},l}( \hat{t} )$ is a power 
of~$\hat{t}$ for all integers~$l$. Finally, 
$\hat{t} \in \GL_d^{\varepsilon}(q)$, as $\zeta^{q - \varepsilon} = 1$.

As $\det(\hat{t}) = \det(\hat{h})$, there is
$\hat{y} \in \SL^\varepsilon_d(q)$ such that $\hat{y}\hat{h} = \hat{t}$. 
Let~$t$ and~$y$ denote the image of $\hat{t}$, respectively~$\hat{y}$ 
in~$\PGL^\varepsilon_d(q)$ and notice that $y \in G$.  Put 
$\beta' := \ad_y \circ \beta = \ad_t \circ \mu$.
As~$\langle t \rangle$ is $\mu$-invariant, Equation~(\ref{PowerFormula}) shows
that~$|\beta'|$ divides $|t||\mu|$, and that $|\mu|$ divides~$|\beta'|$.
Since~$|t|$ is odd, this implies that~$|\beta'|$ is even if and only
if~$|\mu|$ is even. We now distinguish two cases. Suppose first that $|\mu|$ is
even. Then $|\beta'|$ is even and we put $\alpha := \beta'$.
Now suppose that~$|\mu|$ is odd so that $|\beta'|$ is odd. Then~$\mu$ is a
power of the field automorphism~$\varphi$. Choose an involution
$\hat{z} \in \SL^\varepsilon_d(2) \leq \GL_d^{\varepsilon}(q)$ which commutes
with~$\hat{t}$ and whose number of Jordan blocks of size~$2$ is as large as
possible. From the structure of $C_{\GL_d^{\varepsilon}(q)}(\hat{t})$
(see Lemma~\ref{FormalitiesOnPolynomialsIII}(c)), we conclude that  we can
choose~$\hat{z}$ to have $\lfloor d/2 \rfloor$ Jordan blocks of size~$2$.
Let $z \in G$ denote the image of $\hat{z}$, and put
$\alpha := \ad_z \circ \beta' = \ad_{zt} \circ \mu$. Now $\ad_z$ commutes
with~$\beta'$, as~$\varphi( z ) = z$ and~$z$ commutes with~$t$. In particular,
$|\alpha| = 2|\beta'|$ and $\alpha_{(p)} = \beta'_{(p)}$ for all odd primes~$p$
dividing~$|\alpha|$. Clearly, $\alpha_{(2)} = \ad_z$.

(a) Let~$p$ denote a prime dividing~$|\beta'|$. We claim that 
$|C_G( \beta'_{(p)} )| \leq M_G$. If $\beta'_{(p)}$ is not an
inner-diagonal automorphism of~$G$, the claim follows from
Lemma~\ref{CentralizersEstimates}. Suppose then that $\beta'_{(p)}$ is an
inner-diagonal automorphism. As~$\beta'_{(p)}$ is a power of~$\beta'$,
Equation~(\ref{PowerFormula}) yields $\beta'_{(p)} = \ad_{N_{\mu,l}(t)}$ for some
positive inter~$l$. In particular,~$p$ is odd and $|N_{\mu,l}(t)| = p$. Also,
$C_G( \beta'_{(p)} ) = C_G( N_{\mu,l}(t) )$. 
As $N_{\hat{\mu},l}(\hat{t})$ is a lift of $N_{\mu,l}(t) )$, we obtain 
$$|C_G( N_{\mu,l}(t) )| \leq |C_{\GL_d^{\varepsilon}(q)}( N_{\hat{\mu},l}(\hat{t}) )|$$ 
by Lemma~\ref{CentralizerOrderComparison}. This bound is 
sufficient for our purpose, except if $d = 3$. Recall that 
$N_{\hat{\mu},l}(\hat{t}) = \hat{t}^k$ for some integer~$k$. Since 
$N_{\hat{\mu},l}(\hat{t}) \neq 1$, and the eigenvalues of~$\hat{t}$ lie in
$\langle \zeta \rangle$, we have $\zeta^k \neq 1$. Also, 
$\zeta^k \neq \zeta^{-k}$ as~$q$ is even. However, we may have $\xi^k = \zeta^k$ 
or $\xi^k = 1$. This allows to estimate the dimensions of the eigenspaces 
of~$N_{\hat{\mu},l}(\hat{t})$. Notice that the eigenvalues~$\zeta'$ 
of~$N_{\hat{\mu},l}(\hat{t})$ satisfy ${\zeta'}^{q - \varepsilon} = 1$.
Lemma~\ref{FormalitiesOnPolynomialsIII}(c) implies that
$C_{\GL_d^{\varepsilon}(q)}( N_{\hat{\mu},l}(\hat{t}) )$ is a direct product of 
groups $\GL_{d_i}^{\varepsilon}(q)$, where~$d_i$ is the dimension of an eigenspace 
of~$N_{\hat{\mu},l}(\hat{t})$.
If $d = 3$, we have $N_{\hat{\mu},l}(\hat{t}) = \diag(\zeta',1,\zeta')$ for some 
$\zeta'$ of order~$p$ so that 
$C_{\GL_d^{\varepsilon}(q)}( N_{\hat{\mu},l}(\hat{t}) ) \cong 
\GL_1^{\varepsilon}(q) \times \GL_2^{\varepsilon}(q)$.
Clearly, every element of $\GL_3^{\varepsilon}(q)$, which commutes with
$\diag(\zeta',1,\zeta')$ up to a scalar, in fact commutes with 
$\diag(\zeta',1,\zeta')$. Thus $C_{\overline{G}^{\sigma}}( N_{\mu,l}(t) )$
is the image of $C_{\GL_3^{\varepsilon}(q)}( N_{\hat{\mu},l}(\hat{t}) )$ under the
canonical epimorphism. It follows that 
$|C_G( N_{\mu,l}(t) )| \leq |C_{\overline{G}^{\sigma}}( N_{\mu,l}(t) )| =
|\GL_2^{\varepsilon}(q)| = q(q-\varepsilon)(q^2-1) 
< M_G$, as claimed in~(a). Suppose now that $d > 3$. For $d = 5, 6$, we obtain 
$C_{\GL_d^{\varepsilon}(q)}( N_{\hat{\mu},l}(\hat{t}) ) \cong \GL_2^{\varepsilon}(q) 
\times \GL_3^{\varepsilon}(q)$, respectively $\GL_2^{\varepsilon}(q) \times 
\GL_4^{\varepsilon}(q)$. Using the exact values of the orders of theses groups,
we obtain our claim in~(a). For $d = 7$ we have $\varepsilon = 1$ and
$C_{\GL_d(q)}( N_{\hat{\mu},l}(\hat{t}) ) \cong \GL_3(q) \times \GL_4(q)$.
In this case, the claim follows from Lemma~\ref{GUdPrepPrep}(e).
Now suppose that $d \geq 9$. If $d - \bar{d} = 4d'$, then
$C_{\GL_d^{\varepsilon}(q)}( N_{\mu,l}(\hat{t}) )$ is isomorphic to one of
$\GL_{2d'}^{\varepsilon}(q) \times \GL_{2d'+\bar{d}}^{\varepsilon}(q)$ or
$\GL_{2d'}^{\varepsilon}(q) \times \GL_{2d'}^{\varepsilon}(q) \times
\GL_{\bar{d}}^{\varepsilon}(q)$. 
If $d - \bar{d} = 4d' + 2$, then
$C_{\GL_d^{\varepsilon}(q)}( N_{\hat{\mu},l}(\hat{t}) )$ is isomorphic to one of 
$\GL_2^{\varepsilon}(q) \times \GL_{2d'}^{\varepsilon}(q) \times 
\GL_{2d'+\bar{d}}^{\varepsilon}(q)$ or
$\GL_2^{\varepsilon}(q) \times \GL_{2d'}^{\varepsilon}(q) \times 
\GL_{2d'}^{\varepsilon}(q) \times \GL_{\bar{d}}^{\varepsilon}(q)$. 
In each case, the first named possibility for the centralizer has a larger order 
than the second. For $d > 10$, the estimate $|\GL_{l}^{\varepsilon}(q)| \leq 
q^{l^2 + 2}$ from Lemma~\ref{GUdPrepPrep}(e) yields our bounds. For $d = 9, 10$ 
and $q \geq 4$, we use $|\GL_{l}^{\varepsilon}(q)| \leq q^{l^2 + 1/2}$, and for 
$d = 9, 10$ and $q = 2$, we use the exact values for 
$|\GL_{l}^{\varepsilon}(q)|$ to obtain our claim. This proves (a) in 
case~$|\mu|$ is even.
 
Now suppose that~$|\mu|$ is odd so that $\alpha = \ad_z \circ \beta'$.
We have to show that $|C_G( \alpha_{(2)} )| = |C_G( z )| \leq M_G$.
If $d = 3$, then
$|C_{\GL^\varepsilon_3(q)}( \hat{z} )| = q^3 (q - \varepsilon)^2$ by
Lemma~\ref{InvolutionCentralizerInLinearGroup}. The same lemma implies that
$|C_{\GL^\varepsilon_d(q)}( \hat{z} )| \leq q^{(d^2 + 1)/2}$ if 
$\varepsilon = 1$, and 
$|C_{\GL^\varepsilon_d(q)}( \hat{z} )| \leq q^{(d^2 + 9)/2}/(q^2-q-1)^2$, 
if $\varepsilon = -1$. This implies that 
$|C_{\GL^\varepsilon_d(q)}( \hat{z} )| \leq q^{d(d+1)/2}$ in all cases, except 
for $\varepsilon = -1$, $q = 2$ and $d \leq 8$. If $q = 2$, we have $3 \mid d$
so that we are left with $d = 6$ by Lemma~\ref{GUdPrepPrep}(a) (recall that
$(d,q) = (3,2)$ is excluded by hypothesis, as $\PSU_3(2)$ is not a simple 
group). If $d = 6$, the involution $\hat{z}$ has~$3$ Jordan blocks of size~$2$,
and hence $|C_{\GL^\varepsilon_6(q)}( \hat{z} )| = q^9|\GU_3(q)|$ by
Lemma~\ref{InvolutionCentralizerInLinearGroup}. Inserting the exact values for 
$q = 2$, we obtain $|C_{\GL^\varepsilon_6(q)}( \hat{z} )| \leq q^{d(d+1)/2}$
for $q = 2$ as well. Clearly, $C_{\overline{G}^{\sigma}}( z )$ is the image of 
$C_{\GL^\varepsilon_d(q)}( \hat{z} )$ in~$\overline{G}^{\sigma}$. It follows
that $|C_G( \alpha_{(2)} )| = |C_G( z )| \leq M_G$, thus completing the proof 
of~(a).

(b) Suppose first that $|\mu|$ is even, so that $|\alpha| = |\beta'|$. Since 
$|\beta'|$ divides $|t||\mu|$ and $|t|$ and $|\mu|$ divide $q-\varepsilon$ and 
$\delta f$, respectively, we obtain~(b). Now suppose that $|\mu|$ is odd, so that 
$|\alpha| = 2|\beta'|$. Since~$f$ is even if $\varepsilon = 1$, this implies 
that~$|\mu|$ divides $\delta f/2$, so that $|\beta'|$ divides 
$\delta f(q-\varepsilon)/2$, which yields~(b). 

(c) By Lemma~\ref{OrderBound}(b), the order of $\beta'$ divides $3 \delta f$,
and $3 \delta f/2$ if $|\beta'|$ is odd. This yields the assertion 
for~$|\alpha|$.

(d) If $d = 3$, $\varepsilon = 1$ and $q = 4$, then $\Gamma_G \times \Phi_G$ is
elementary abelian of order~$4$. In addition, $|\alpha| \in \{ 2, 6 \}$ by~(c).
Suppose that $|\alpha| = 6$. If $|\beta'|$ is odd, then $\beta' = \ad_t$ and
$G^{\ex}  \cong \PGL_3(4)$. If $|\beta'|$ is even, then $\alpha = \ad_t \circ \mu$
with $|t| = 3$ and $\mu$ an involution fixing~$t$. As~$\iota$ and~$\varphi$
invert~$t$, we must have $\mu = \iota \circ \varphi$. This yields the second
option for~$G^{\ex} $.

Suppose that $G = \PSU_3(8)$. Then $|\alpha| \in \{ 2, 6, 18 \}$ by~(c).
We have $\beta' = \ad_t \circ \mu$ for some
$t \in \PGU_3(8)$ of $3$-power order and some $\mu \in \Phi_G$. Recall that
$t = 1$ if $t \in G$. Hence either $t = 1$ or $|t| = 9$. If
$|\beta'| = 3$ or $|\beta'| = 9$, then $|\mu| \in \{ 1, 3 \}$. In either case,
$\alpha \not\in \Inn(G)$ but $\alpha^3 \in \Inn(G)$. Thus $G^{\ex}  \cong G.3$.
If  $|\alpha| = 18$, then $|\beta'| = 9$ by Lemma~\ref{OrderBound}(a)(c).
We are left with the case that $|\alpha| = |\beta'| = 6$. Then $|\mu| = 6$
and $\alpha = \beta'$. Hence $G^{\ex}  \cong G.6$.

Suppose that $G = \PSU_3(32)$ and $|\alpha| = 30$. Then $|\beta'| = 15$ by
Lemma~\ref{OrderBound}(c). It follows that $\alpha^{10} = {\beta'}^5 = \ad_{t'}$
for an element $t' \in \PGU_3(32) \setminus G$. This proves our assertion.

Suppose that $G = \PSU_d(4)$ for some $d \geq 5$. Then 
$|\beta'| \in \{ 1, 2, 4, 5\}$ by Lemma~\ref{OrderBound}(c), which yields the
claim for~$|\alpha|$. The analogous proof works for $G = \PSU_d(2)$ for some 
$d \geq 5$ and $G = \PSU_6(8)$.
If $G = \PSU_6(2)$ and $|\alpha| = 6$, then $|\beta'| = 3$ by 
Lemma~\ref{OrderBound}(c), and thus $\beta' \in \PGU_6(2)$.
\end{proof}

\addtocounter{subsection}{1}
\subsection{The linear groups of degree~$3$}
We can now prove that the linear groups of degree~$3$ have the $E1$-property.

\addtocounter{num}{1}
\begin{prp}
\label{PSU3qEven}
Let $G = \PSL^{\varepsilon}_3(q)$ and let $(V,n,\nu)$ be as in 
\cite[Notation~$4.1.1$]{HL1}. Then $(G,V,n)$ has the $E1$-property.
\end{prp}
\begin{proof}
Let~$\chi$ denote the character of~$V$.
The character table of~$G$ is available in~\cite{SiFra}. The irreducible 
characters of~$G$ of degree $q^2 + \varepsilon q + 1$ are not real. 
Thus~$\chi$ is one of the remaining irreducible character of odd degree, and
hence $$\chi(1) > (q^3 - 2q^2)/3$$ by~\cite{SiFra}. 
For these characters we are going to apply \cite[Lemma~$4.3.3$]{HL1},
with~$\alpha$ constructed from $\beta := \nu$ according to 
Proposition~\ref{QEvenPrep}. Thus~$\alpha$ has even order and 
$|C_G( \alpha_{(p)} )| \leq M_G$ for all primes~$p$ 
dividing~$|\alpha|$, where~$M_G$ is as in Definition~\ref{DefineMG}(c). 
Moreover,~$|\alpha|$ divides $3 \delta f$ by Proposition~\ref{QEvenPrep}(c). We 
also put $G^{\ex}  := \langle \Inn(G), \alpha \rangle = \langle \Inn(G), \nu \rangle$, 
and let~$\chi^{\ex}$ denote the extension of~$\chi$ to~$G^{\ex} $; 
\cite[Remark~$4.2.5$]{HL1}.

Suppose first that $\varepsilon = 1$. We aim to show that
\begin{equation}
\label{L3q}
(q^3 - 2q^2)/3 > (|\alpha|-1)M_G^{1/2}.
\end{equation}
Squaring, we see that~(\ref{L3q}) will follow from
$$q^6 - 4q^5 + 4q^4 > 9(|\alpha|-1)^2M_G.$$
If $q > 64$, we have $|\alpha|-1 < 3f < q/4$ and $M_G = q^4$, and it is enough 
to show that
$$16(q^6 - 4q^5 + 4q^4) > 9q^6.$$
This is equivalent to
$$q^5(7q-64) + 64q^4 > 0,$$
which is obviously true for $q > 64$. It remains to consider the cases
$q \in \{ 4, 16, 64 \}$. For $q = 64$, we have $M_G = q^4$ and 
$|\alpha| \leq 18$. With this,~(\ref{L3q}) holds. If $q = 16$, 
we have $|\alpha| \leq 6$ and $M_G = 62\,401$. Using the exact values 
for~$\chi(1)$ from~\cite{SiFra}, we obtain $\chi(1) > 5 \cdot 62\,401^{1/2}$, 
which gives our result.
Suppose that $q = 4$. Then $\chi(1) \in \{ 35, 63 \}$ by the Atlas~\cite{Atlas},
and $|\alpha| \in \{ 2, 6 \}$.  If $|\alpha| = 2$, then~(\ref{L3q}) holds. 
Suppose then that $|\alpha| = 6$.  By Proposition~\ref{QEvenPrep}(d), either 
$G^{\ex} \cong \PGL_3(4)$ or $G^{\ex} \cong G.6$. According to the Atlas~\cite{Atlas}, no 
group of the latter shape has a character of degree~$35$ or~$63$. If 
$G^{\ex} \cong \PGL_3(4)$, then $\chi(1) = 63$ and 
$\Res^{G^{\ex} }_{\langle \alpha \rangle}( \chi^{\ex} )$ contains each of the two 
irreducible real characters of~$\langle \alpha \rangle$ with positive 
multiplicity; see the Atlas~\cite{Atlas}. Hence $(G,V,n)$ has the $E1$-property
by \cite[Lemma~$4.3.1$]{HL1}.

Suppose now that $\varepsilon = -1$. Then $M_G = q^4 + q^3$ and~$|\alpha|$ 
divides~$6f$.  Suppose first that $f \geq 21$. Then $(6f - 1)^3 < 2^f$ and so 
$(|\alpha|-1) < q^{1/3}$. By \cite[Lemma~$4.3.3$]{HL1}, we have to show that
\begin{equation}
\label{UltimateInequality}
(q^3 - 2q^2)/3 > q^{1/3}(q^4 + q^3)^{1/2},
\end{equation}
as then
$$\dim(V) \geq (q^3 - 2q^2)/3 > q^{1/3}(q^4 + q^3)^{1/2} > 
(|\alpha|-1)(q^4 + q^3)^{1/2}.$$
Squaring~(\ref{UltimateInequality}), it suffices to show that
$$
q^6 - 4q^5 + 4q^4 > 9q^{2/3}(q^4 + q^3).
$$
Noting that $4q^4 > 0$ and dividing by~$q^{2/3}q^3$, it suffices to 
show that
$$q^{4/3}(q - 4) > 9(q + 1).$$
This is the case, as $q \geq 2^{21}$. For $q = 2^f$ with $7 \leq f \leq 19$, a 
GAP computation shows that 
$$(q^3 - 2q^2)^2 > 9(6f - 1)^2(q^4 + q^3)$$
which proves our assertion, except for $q = 32$ or $q = 8$.

Suppose first that $q = 32$. Then $|\alpha| \leq 10$ or $|\alpha| = 30$. Now 
$(q^3 - 2q^2)^2 > 9(10-1)^2(q^4 + q^3)$, so that $(G,V,n)$ has the $E1$-property 
if $|\alpha| \leq 10$. Suppose that $|\alpha| = 30$. Then 
$\PGU_3(q) \cong \langle G, \alpha^{10} \rangle \leq G^{\ex} $ by 
Proposition~\ref{QEvenPrep}(d). In particular,~$\chi$ extends to~$\PGU_3(q)$, so 
that $\chi(1) > q^3 - 2q^2$; see the character tables in~\cite{SiFra}. As 
$(q^3 - 2q^2)^2 > 29^2(q^4 + q^3)$, this case also has the $E1$-property.

Suppose finally that $q = 8$. Then $|\alpha| \in \{ 2, 6, 18 \}$ by 
Proposition~\ref{QEvenPrep}(c).  Also, $\chi(1) \in \{ 133, 399, 513 \}$; see 
the Alas~\cite{Atlas}. We have $\chi(1)^2 > (|\alpha| - 1)^2(q^4+q^3)$, unless 
$|\alpha| = 18$ or $|\alpha| = 6$ and $\chi(1) = 133$. If $|\alpha| = 18$, then 
$G^{\ex}  \cong G.3$, by Proposition~\ref{QEvenPrep}(d). There are three isomorphism 
types of groups~$G.3$; see \cite{Atlas}. Only two of them have elements of 
order~$18$. We use Gap to compute $\Res_{\langle \alpha \rangle}^{G^{\ex} }( \chi^{\ex} )$, 
where~$\chi^{\ex}$ is the character of~$G^{\ex} $ on~$V$; see \cite[Remark~$4.2.5$]{HL1}.
It turns out that this restriction contains every irreducible character of 
$\langle \alpha \rangle$ with positive multiplicity. Thus $(G,V,n)$ has the 
$E1$-property by \cite[Lemma~$4.3.1$]{HL1}.
Now suppose that $|\alpha| = 6$ and $\chi(1) = 133$. By 
Proposition~\ref{QEvenPrep}(d), either $G^{\ex}  \cong G.3$ or $G^{\ex}  \cong G.6$. A GAP 
computation as in the case of $|\alpha| = 18$ shows that $(G,V,n)$ has the
$E1$-property.
\end{proof}

\addtocounter{subsection}{1}
\subsection{The characters of large degree}
From now on we can assume that $d \geq 5$. In the next two lemmas we identify 
the irreducible characters of~$G$ whose degree is large enough to satisfy the 
condition of \cite[Lemma~$4.3.3$]{HL1}.
\addtocounter{num}{1}
\begin{lem}
\label{GLepsPrepPrep}
Let $\hat{G} = \GL_d^{\varepsilon}( q )$ for some $d \geq 5$. Let 
$1 \neq \hat{s} \in \hat{G}$ be semisimple and real. Let~$\Xi$ denote the 
characteristic polynomial of~$\hat{s}$. Let~$\Delta_1$ denote the monic 
polynomial of degree~$1$ with root~$1$, and let~$d_1$ be the multiplicity 
of~$\Delta_1$ in~$\Xi$. Let $\Delta_2, \ldots, \Delta_l$ denote the distinct 
irreducible monic factors of~$\Xi$ different from~$\Delta_1$. Let $d_j$ and 
$k_j$ denote, respectively, the multiplicity of~$\Delta_j$ in~$\Xi$ and the 
degree of~$\Delta_j$, $j = 2, \ldots , l$. Then the following hold.

{\rm (a)} We have
\begin{equation}
\label{DExpansion}
d = \sum_{j = 1}^l d_j k_j,
\end{equation}
with $k_1 = 1$.

{\rm (b)}
Let $d' := \max\{ d_j \mid 1 \leq j \leq l \}$.
Put $\hat{C} := C_{\hat{G}}( \hat{s} )$ and 
$$D := [\hat{G}\colon\!\hat{C}]_{2'} \cdot q^{-d(d+1)/4}.$$ 
Then
\begin{equation}
\label{DEstimate}
D \geq (1 - q^{-1} - q^{-2})^{l'+1}q^{d(d - 2d' - 1)/4},
\end{equation}
with $l' = 0$ if $\varepsilon = 1$ and $l' = l$ if $\varepsilon = -1$. 

{\rm (c)} If $d_1 < d/3$ and $d' \geq \lfloor d/2 \rfloor$, then
$d' = \lfloor d/2 \rfloor$ and 
$\Xi = \Delta_1^{d_1}\Delta^{\lfloor d/2 \rfloor}$,
where~$\Delta$ is a monic polynomial of degree~$2$ with $\Delta = \Delta^*$.
\end{lem}
\begin{proof}
(a) Equation~(\ref{DExpansion}) is clear.

(b) Suppose that $\varepsilon = 1$.
For $1 \leq j \leq l$, each factor$\Delta_j^{d_j}$ of~$\Xi$ contributes the 
direct factor $\GL_{d_j}( q^{k_j} )$ to~$\hat{C}$; see
Lemma~\ref{FormalitiesOnPolynomialsIII}(a). Suppose that 
$\varepsilon = - 1$. For $1 \leq j \leq l$, the factor $\Delta_j^{d_j}$ 
of~$\Xi$ with $\Delta_j = 
\Delta_j^{\dagger}$ contributes the direct factor $\GU_{d_j}( q^{k_j} )$
to~$\hat{C}$, and each factor $(\Delta_j\Delta_j^{\dagger})^{d_j}$ with
$\Delta_j \neq \Delta_j^{\dagger}$ the direct factor $\GL_{d_j}( q^{2k_j} )$;
see Lemma~\ref{FormalitiesOnPolynomialsIII}(c),(d).
Lemma~\ref{GUdPrepPrep}(e) yields, for either case of~$\varepsilon$,
\begin{eqnarray*}
(1-q^{-1}-q^{-2})^{l'}|\hat{C}|_{2'} & \leq & \prod_{j = 1}^l q^{k_jd_j(d_j+1)/2} \\
               & \leq & \prod_{j = 1}^l q^{k_jd_j(d'+1)/2} \\
               & =    & q^{(d'+1)/2 \sum_{j=1}^l k_jd_j} \\
               & =    & q^{d(d'+1)/2}.
\end{eqnarray*}
Also,
$$|\hat{G}|_{2'} \geq q^{d(d+1)/2}(1 - q^{-1} - q^{-2}),$$
by Lemma~\ref{OrderEstimates}, which gives~(\ref{DEstimate}).

(c) Let~$k'$ denote the maximum degree of the irreducible factors of~$\Xi$ 
occurring with multiplicity~$d'$. Let~$\Delta'$ be one of the
$\Delta_1, \ldots , \Delta_l$, for which these values are obtained. By 
hypothesis, $d' \geq \lfloor d/2 \rfloor \geq d/3 > d_1$, and thus 
$\Delta' \neq \Delta_1$.  If $d' > \lfloor d/2 \rfloor$,
Equation~(\ref{DExpansion}) implies that $k' = 1$; in turn,
$\Delta' \neq {\Delta'}^*$ by
Lemma~\ref{FormalitiesOnPolynomialsI}(a). As~${\Delta'}^*$ occurs with
multiplicity~$k'$ in~$\Xi$ by Lemma~\ref{FormalitiesOnPolynomialsII}(b), this
contradicts Equation~(\ref{DExpansion}).
Suppose now that $d' = \lfloor d/2 \rfloor$; then $k' \leq 2$,
and~$\Xi$ is as claimed.
\end{proof}

\begin{lem}
\label{GUdPrep}
Let $\hat{G} = \GL^{\varepsilon}_d(q)$ for some $d \geq 5$ such that 
$e = \gcd( d, q - \varepsilon ) > 1$. Let $1 \neq \hat{s} \in \hat{G}$ be 
semisimple and real with characteristic polynomial~$\Xi$. Let~$\Delta_1$ denote 
the monic polynomial of degree~$1$ with root~$1$, and let~$d_1$ be the 
multiplicity of~$\Delta_1$ in~$\Xi$. Put $\hat{C} := C_{\hat{G}}( \hat{s} )$. 
Then one of the following holds.

{\rm (a)} We have $d_1 \geq d/3$.

{\rm (b)} We have
$\Xi = \Delta_1^{d_1}\Delta^{\lfloor d/2 \rfloor}$, where~$\Delta$ is 
a monic polynomial of degree~$2$ with $\Delta = \Delta^*$. 

{\rm (c)} We have
$[\hat{G}\colon\!\hat{C}]_{2'} \geq \delta ef(q - \varepsilon) \cdot q^{d(d+1)/4}$.

{\rm (d)} We have $\varepsilon = -1$, $q = 4$ and
$[\hat{G}\colon\!\hat{C}]_{2'} \geq 45 \cdot q^{d(d+1)/4}$.

{\rm (e)} We have $\varepsilon = -1$, $q = 2$ and
$[\hat{G}\colon\!\hat{C}]_{2'} \geq 15 \cdot q^{d(d+1)/4}$.

{\rm (f)} We have $\varepsilon = -1$, $(d,q) = (5,4)$, and 
$[\hat{G}\colon\!\hat{C}]_{2'} > 12 \cdot q^{d(d+1)/4}$.

{\rm (g)} We have $\varepsilon = -1$, $(d,q) = (6,8)$ and
$[\hat{G}\colon\!\hat{C}]_{2'} > 51 \cdot q^{d(d+1)/4}$.

{\rm (h)} We have $\varepsilon = -1$ and $(d,q) = (6,2)$.
\end{lem}
\begin{proof}
Adopt the notation of Lemma~\ref{GLepsPrepPrep}. Suppose that~(a) and~(b) do not 
hold. Then $d_1 < d/3$ and $d' \leq \lfloor d/2 \rfloor - 1$ by 
Lemma~\ref{GLepsPrepPrep}(c). 

First assume that $\varepsilon = 1$. Lemma~\ref{GLepsPrepPrep}(b) gives
\begin{equation}
\label{FirstDEstimateGLd}
D \geq (1 - q^{-1} - q^{-2})q^{d(d - 2d' - 1)/4}.
\end{equation}
As $d' \leq \lfloor d/2 \rfloor - 1$, we have $d - 2d' - 1 \geq 1$, if~$d$ is 
even, and $d - 2d' - 1 \geq 2$, if~$d$ is odd. Since $\gcd( d, q - 1 ) > 1$, we 
have $q \geq 4$. Hence $1 - q^{-1} - q^{-2} \geq 1 - 1/4 - 1/16 = 11/16$. 

From now on we distinguish the cases~$d$ even and~$d$ odd. Suppose first 
that~$d$ is even, so that $d - 2d' - 1 \geq 1$. If $d > 12$, then
$$
D \geq 11/16 \cdot q^{d/4} > 1/2 \cdot q^{7/2} > q^3 \geq ef(q-1),$$
where the first estimate arises from~(\ref{FirstDEstimateGLd}). We are thus
in case~(c). This leaves the cases $d = 6, 10, 12$.
(The cases $d = 4, 8$ do not occur thanks to Lemma~\ref{GUdPrepPrep}(a).)

For $d = 10$, we have $e = 5$ and $q \geq 16$, and we obtain
$$D \geq 11/16 \cdot q^{5/2} \geq 11/16 \cdot q^{3/2}( q - 1 ) \geq 
5 f (q-1),$$ so that we are in case~(c). For $d = 6$ or $d = 12$ we
have $e  = 3$. If $d = 12$, we obtain
$$D \geq 11/16 \cdot q^{3} \geq 3 f (q-1)$$ 
for all $q \geq 4$, so that we are in case~(c). Suppose finally that $d = 6$.
By checking all the possibilities for~$\Xi$, using $d' \leq 2$ and $d_1 = 0$,
we find that the largest possible value for~$|\hat{C}|_{2'}$ is assumed for
$$|\hat{C}|_{2'} = (q^2-1)(q^2-1)(q^4-1).$$
Inserting the exact values for~$|\hat{G}|_{2'}$ and~$|\hat{C}|_{2'}$, we obtain
$$D \geq q^{-21/2} (q-1)(q^3-1)(q^5-1)(q^4+q^2+1).$$
Now $(q-1)(q^3-1)(q^5-1)(q^4+q^2+1) \geq q^{13}$, and so
$$D \geq q^{5/2}.$$
It is easily checked that 
$$q^{5/2} \geq 3 f (q-1)$$
for $q \geq 4$, so that we are again in case~(c).

Assume now that~$d$ is odd, so that $d - 2d' - 1 \geq 2$. If $d \geq 7$, we 
obtain
$$
D \geq 11/16 q^{d/2} > 1/2 \cdot q^{7/2} > q^3 \geq ef(q-1),$$
where the first estimate arises from~(\ref{FirstDEstimateGLd}). We are thus
in case~(c). If $d = 5$, we have $e = 5$ and $q \geq 16$. Moreover, $d' = 1$. 
Then $|\hat{C}|_{2'} \leq q^5 - 1$, and hence 
$[\hat{G}\colon\!\hat{C}]_{2'} \geq (q-1)(q^2-1)(q^3-1)(q^4-1)$. Now 
$$(q-1)(q^2-1)(q^3-1)(q^4-1) > q^{15/2}\cdot 5 f (q-1)$$
for all $q \geq 16$, so that we are in case~(c). This completes the proof in 
case $\varepsilon = 1$.

Assume now that $\varepsilon = -1$. If $d' = d_1$, and each irreducible factor 
$\Delta \neq \Delta_1$ of~$\Xi$ occurs with multiplicity strictly smaller 
than~$d_1$, put $\Delta' := \Delta_1$. Otherwise, choose a monic irreducible 
factor $\Delta' \neq \Delta_1$ of~$\Xi$ occurring with multiplicity~$d'$, and 
such that $\deg( \Delta' )$ is maximal among the degrees of the irreducible 
factors of~$\Xi$ with multiplicity~$d'$. 
Lemma~\ref{GLepsPrepPrep}(b) gives
\begin{equation}
\label{FirstEstimateGUd}
D \geq (1 - q^{-1} - q^{-2})^{l+1}q^{d(d - 2d' - 1)/4}.
\end{equation}
We now complete the proof with several claims. The end of the proof of each 
claim is indicated by the symbol~$\Diamond$.

\textit{Claim~$1$}: If $d \leq 6$, we are in one of the Cases~(c)--(h).

If $d = 5$ we have $d' = 1$ from $d' \leq \lfloor d/2 \rfloor - 1$. Hence
$|\hat{C}|_{2'} \leq (q + 1)^5$ by Lemma~\ref{FormalitiesOnPolynomialsIV}(d), 
and so
$$D \geq q^{-15/2} (q-1)^2(q^2-q+1)(q^2+1)(q^4-q^3+q^2-q+1).$$ 
As $d = 5$ we have $e = 5$ and $q \geq 4$. It is easy to check that
$$q^{-15/2}(q-1)^2(q^2-q+1)(q^2+1)(q^4-q^3+q^2-q+1) \geq 10 \cdot f(q+1)$$ 
for all $q > 4$ with $5 \mid q + 1$, so that we are in Case~(c) for $q > 4$. 
For $q = 4$ we are in Case~(f), as can be checked directly.

Suppose now that $d = 6$. Then $e = 3$, and thus $3 \mid q + 1$. Moreover,
$d' \leq 2$, as $d' \leq \lfloor d/2 \rfloor - 1$. Also, $d_1 = 0$, as 
$d_1 < d/3$. We also assume that $q > 2$, as the case $q = 2$ is listed in~(h).
Suppose first that $d' = 2$. Then $k' = 1$ by 
Lemma~\ref{FormalitiesOnPolynomialsIV}(c). It follows that 
$\Xi = (\Delta_2\Delta_3)^2\Delta_4\Delta_5$, with polynomials 
$\Delta_2, \ldots , \Delta_5$ of degree~$1$, and $(\Delta_2\Delta_3)^* = 
\Delta_2\Delta_3 = (\Delta_2\Delta_3)^{\dagger}$. 
There are four such possibilities and Lemma~\ref{FormalitiesOnPolynomialsIII}
yields
$|\hat{C}|_{2'} \leq |\GU_2(q)|_{2'}^2(q+1)^2 = (q+1)^4(q^2-1)^2$.
Thus
$$D \geq q^{-21/2} (q-1)(q^2+q+1)(q^2-q+1)^2(q^2+1)(q^4-q^3+q^2-q+1).$$
It is easy to check that 
$(q-1)(q^2+q+1)(q^2-q+1)^2(q^2+1)(q^4-q^3+q^2-q+1) \geq 
q^{21/2}\cdot6f(q+1)$ for
all $q > 8$ with $3 \mid q + 1$, so that we are in Case~(c) for $q > 8$. 
The case $q = 8$ gives the bound asserted in~(g).
If $d' = 1$, we have $k' \leq 3$ by Lemma~\ref{FormalitiesOnPolynomialsIV}(c),
and $|\hat{C}|_{2'} \leq (q + 1)^6$ by
Lemma~\ref{FormalitiesOnPolynomialsIV}(e). In particular,
$$D \geq q^{-21/2} (q-1)^3(q^2+q+1)(q^2-q+1)^2(q^2+1)(q^4-q^3+q^2-q+1).$$
It is easy to check that $(q-1)^3(q^2+q+1)(q^2-q+1)^2(q^2+1)(q^4-q^3+q^2-q+1) 
\geq q^{21/2}\cdot6f(q+1)$ for
all $q > 2$ with $3 \mid q + 1$, so that we are in Case~(c).
\hfill{$\Diamond$}

Assume henceforth that $d > 6$. Then $d \geq 9$ by Lemma~\ref{GUdPrepPrep}(a).
Define the positive integer~$m$ by
$$d' = \lfloor d/2 \rfloor - m,\quad\quad m = 1, 2, \ldots$$

\textit{Claim~$2$}: Suppose that $d \geq 6m'$ and $d \neq 6m' + 1$ for
some positive integer~$m'$. If $d' = d_1$, then $m > m'$. In particular,
if $d' = d_1$, then $m > 1$, as $d \geq 9$.

This follows from $\lfloor d/2 \rfloor - m = d' = d_1 < d/3$, and 
$\lfloor d/2 \rfloor -d/3 \geq m'$. \hfill{$\Diamond$}.

\textit{Claim~$3$}: If $m = 1$, then $\Delta' \neq \Delta_1$ and $k' = 1$.

By Claim~$2$ we have $d' > d_1$ and thus $\Delta' \neq \Delta_1$.
Lemma~\ref{FormalitiesOnPolynomialsIV}(c) yields
$d \geq 2k'(\lfloor d/2 \rfloor - 1)$, which gives $k' = 1$ as $d \geq 9$.
\hfill{$\Diamond$}

\textit{Claim~$4$}: Suppose that $m = 1$. Then
\begin{equation}
\label{SecondEstimateGUd}
D \geq (1 - q^{-1} - q^{-2})^{5}q^{(3d - 12)/4}.
\end{equation}
if~$d$ is even, and
\begin{equation}
\label{ThirdEstimateGUd}
D \geq (1 - q^{-1} - q^{-2})^{6}q^{(5d - 27)/4}.
\end{equation}
if~$d$ is odd.

By Claim~$3$ we have $\Delta' \neq \Delta_1$ and $k' = 1$. Moreover,
$\Delta' \neq {\Delta'}^*$ by Lemma~\ref{FormalitiesOnPolynomialsI}(a). 
It follows that
$\Xi = \Delta_1^{d_1} (\Delta'{\Delta'}^*)^{\lfloor d/2 \rfloor - 1}$ or
$\Xi = \Delta_1^{d_1} (\Delta'{\Delta'}^*)^{\lfloor d/2 \rfloor - 1}\Delta$
with~$\Delta$ of degree~$2$. In the latter case,~$\Delta$ is reducible, as
otherwise $\Delta \neq \Delta^{\dagger}$ and~$\Delta^{\dagger}$ would be a 
divisor of~$\Xi$; see Lemmas~\ref{FormalitiesOnPolynomialsI}(b), 
and~\ref{FormalitiesOnPolynomialsIV}(b). From 
Lemma~\ref{FormalitiesOnPolynomialsIII} we can determine the structure 
of~$\hat{C}$ in each case. Lemma~\ref{OrderEstimates} then gives the 
asserted estimates. \hfill{$\Diamond$}

\textit{Claim~$5$}: If $\Delta' \neq \Delta_1$, then 
$l \leq d - 2d' + 3$. In any case, $l \leq d - d' + 1$.

Suppose first that $\Delta' \neq \Delta_1$. In view of~(\ref{DExpansion}), we
get $d \geq d_1 + 2d' + l - 3$ by Lemma~\ref{FormalitiesOnPolynomialsIV}(c) and 
its proof. The first assertion follows from this. If $\Delta' = \Delta_1$, the
claim follows directly from~(\ref{DExpansion}). \hfill{$\Diamond$}

\textit{Claim~$6$}: We have $d - 2d' - 1 \geq 2m - 1$.

This follows from $d' = \lfloor d/2 \rfloor -m \leq d/2 -m$. \hfill{$\Diamond$}.

\textit{Claim~$7$}: We have
$$l + 1 \leq (d+5)/2 + m,$$
and
$$l + 1 \leq 2m + 5,\quad\quad\text{if\ $\Delta' \neq \Delta_1$}.$$

This follows from Claim~$5$ and $d' \geq (d-1)/2 - m$. \hfill{$\Diamond$}.

\textit{Claim~$8$}: If $q \geq 8$, we are in Case~(c). 

Suppose that $q \geq 8$. By Lemma~\ref{GUdPrepPrep}(b) we have
$1 - q^{-1} - q^{-2} \geq q^{-1/4}$, and it suffices to show that $D > q^3$.
Thus
$$D \geq q^{-(d+5)/8 - m/4 + d(2m - 1)/4}$$
by Claims~$6$ and~$7$ and~(\ref{FirstEstimateGUd}). We have 
$-(d+5)/8 - m/4 + d(2m - 1)/4 \geq 3$,
if and only if $d(4m-3) - 2m \geq 29$. The latter holds if $m \geq 2$.

Suppose then that $m = 1$ and that~$d$ is even. Then
$$D \geq q^{-5/4 + (3d - 12)/4}$$
by~(\ref{SecondEstimateGUd}). Hence $D \geq q^3$ for all $d \geq 10$. 
Suppose now that $m = 1$ and~$d$ is odd. 
Then
$$D \geq q^{-6/4 + (5d - 27)/4}$$
by~(\ref{ThirdEstimateGUd}). Hence $D \geq q^3$ for all $d \geq 9$.
\hfill{$\Diamond$}

\textit{Claim~$9$}: If $q = 4$, we are in Case~(d).

Suppose that $q = 4$. Then $q + 1 = e = 5$, so that $5 \mid d$. Since 
$4^{11/4} > 45$, it suffices to show that $D \geq q^{11/4}$.  By 
Lemma~\ref{GUdPrepPrep}(c) we have $1 - q^{-1} - q^{-2} \geq q^{-1/2}$, and thus
$$D \geq q^{-(d+5)/4 - m/2 + d(2m - 1)/4}$$
by Claims~$6$ and~$7$ and~(\ref{FirstEstimateGUd}). We have 
$-(d+5)/4 - m/2 + d(2m - 1)/4 \geq 11/4$,
if and only if $d(m-2) + m(d-2) \geq 16$. If $m \geq 2$, this is the case for 
all $d \geq 10$.
Suppose now that $m = 1$ and that~$d$ is even. Then
$$D \geq q^{-5/2 + (3d - 12 )/4}$$
by~(\ref{SecondEstimateGUd}). Hence $D \geq q^{11/4}$ for all even $d > 10$ with 
$5 \mid d$. Suppose that $d = 10$. Then $d' = 4$ and we obtain $k' = 1$ from 
Lemma~\ref{FormalitiesOnPolynomialsIV}(c).
Lemma~\ref{FormalitiesOnPolynomialsIV}(d) yields $|\hat{C}|_{2'} \leq
|\GU_4(q)|_{2'}^2|\GU_2(q)|_{2'}$. Using the exact values for 
$|\GU_4(4)|_{2'}$ and $|\GU_{10}(4)|_{2'}$, one checks that the bound on
$[\hat{G}\colon\!\hat{C}|_{2'}$ in~(d) is satisfied.
If $m = 1$ and~$d$ is odd, we have
$$D \geq q^{-6/2 + (5d - 27)/4}$$
by~(\ref{ThirdEstimateGUd}). Hence $D \geq q^{11/4}$ for all odd $d > 5$ 
with $5 \mid d$. \hfill{$\Diamond$}

\textit{Claim~$10$}: If $q = 2$, we are in Case~(e). 

Suppose that $q = 2$. Then $q + 1 = e = 3$, so that $3 \mid d$. 
It suffices to show that $D \geq q^4$. By Lemma~\ref{GUdPrepPrep}(c) we have
$1 - q^{-1} - q^{-2} \geq q^{-2}$. 
Suppose first that $\Delta' = \Delta_1$. Then $d_1 = d' > 1$ by our choice 
of~$\Delta'$, and
$$D \geq q^{-d - 5 - 2m + d(2m - 1)/4}$$
by Claims~(6) and~(7) and~(\ref{FirstEstimateGUd}). We have 
$-d - 5 - 2m + d(2m - 1)/4 \geq 4$,
if and only if 
\begin{equation}
\label{Eq36}
d(m-5) + m(d-8) \geq 36. 
\end{equation}
By Claim~$2$ we have $m \geq 5$ if $d \geq 24$. In particular,~(\ref{Eq36}) 
holds for $d \geq 24$. This leaves the cases $9 \leq d \leq 21$. 
Using $1 < d' = d_1 < d/3$ and $d \equiv d_1\,\,(\mbox{\rm mod}\,\,2)$ to
enumerate the remaining possibilities for $d$ and $d' = \lfloor d/2 \rfloor - m$,
we find that~(\ref{Eq36}) is satisfied unless $(d,d') \in \{ (15,3), (12,2) \}$.
Applying Lemma~\ref{FormalitiesOnPolynomialsIV}(e) to the orthogonal complement
of the fixed space of~$\hat{s}$, we get 
$|\hat{C}|_{2'} \leq |\GU_3(2)|_{2'}|\GU_2(2)|_{2'}^6$ if $d = 15$ and
$|\hat{C}|_{2'} \leq |\GU_2(2)|_{2'}(2+1)^{10}$ if $d = 12$. A computation 
yields
\begin{equation}
\label{qEq2}
[\hat{G}\colon\!\hat{C}]_{2'} > 15 \cdot 2^{d(d+1)/4}
\end{equation}
in these two cases.

Suppose now that $\Delta' \neq \Delta_1$. Then
$$D \geq q^{- 4m - 10 + d(2m - 1)/4}$$
by Claims~(6),~(7) and~(\ref{FirstEstimateGUd}). We have 
$- 4m - 10 + d(2m - 1)/4 \geq 4$, if and only if
$d(m-1) + m(d - 16) \geq 56$. For $m \geq 2$, the latter holds for 
$d \geq 30$, and if $m \geq 4$, the condition holds for $d \geq 18$.

Suppose then that $m = 1$ and~$d$ is even. Then
$$D \geq q^{-10 + (3d - 12 )/4}$$
by~(\ref{SecondEstimateGUd}). Hence $D \geq q^4$ for all even $d \geq 18$ with
$3 \mid d$. If $m = 1$ and~$d$ is odd, we have
$$D \geq q^{-12 + (5d - 27)/4}$$
by~(\ref{ThirdEstimateGUd}). Hence $D \geq q^4$ for all odd $d \geq 21$
with $3 \mid d$. 

We finally investigate the cases which cannot be ruled out with the above
arguments. These are included in the cases $9 \leq d \leq 15$ and all 
possible~$m$ such that $d_1 \leq d'$, and the cases $18 \leq d \leq 27$ and 
$1 \leq m \leq 3$. 
If $d' \leq 2$, we use the bounds on~$|\hat{C}|_{2'}$ of
Lemma~\ref{FormalitiesOnPolynomialsIV}(e). If $d' > 2$, we use the estimate
$|\hat{C}|_{2'} \leq |\GU_{d'}(2)|^2_{2'}|\GU_{d-2d'}(2)|_{2'}$ of 
Lemma~\ref{FormalitiesOnPolynomialsIV}(d). Inserting these upper bounds and the 
correct value for $|\GU_d(2)|_{2'}$ into the definition of~$D$, we 
get~(\ref{qEq2}) in all cases, except if $d = 9$ and $d' = 3$. Suppose that 
$d = 9$ and $d' = 3$. Since $d_1 < d/3 = 1$, we get that 
$\Xi = \Delta_1 (\Delta'{\Delta'}^*)^3\Delta$, where~$\Delta$ is a reducible 
polynomial of degree~$2$. Hence $|\hat{C}|_{2'} \leq |\GU_3(2)|^2_{2'}(2+1)^3$, 
which is enough to get~(\ref{qEq2}). 
\hfill{$\Diamond$}

This completes the proof.
\end{proof}

\addtocounter{subsection}{2}
\subsection{The characters of small degree}
In the next two lemmas we show that the irreducible characters of~$G$ whose 
degrees are too small to apply \cite[Lemma~$4.3.3$]{HL1}, satisfy the hypotheses
of Lemma~\ref{ProofByHCInductionQEven}. Recall the notation 
$\Aut'( \GL_d( \mathbb{F} ) )$ introduced in Subsection~\ref{AutomorphismsIV}.

\addtocounter{num}{1}
\begin{lem}
\label{MultiplesOfRealElements}
Let $\hat{s} \in \GL_d( \mathbb{F} )$ be semisimple and real with characteristic
polynomial~$\Xi$. Let~$\Delta_1$ denote the monic polynomial of degree~$1$ with
root~$1$, and let~$d_1$ be the multiplicity of~$\Delta_1$ in~$\Xi$.

{\rm (a)} Suppose that $\Xi = \Delta_1^{d_1} \Delta^{d_2}$ with
$\Delta = \Delta_2 \Delta_2^*$, where $\Delta_2 \neq \Delta_1$ is monic of
degree~$1$. Suppose also that $d_1 \neq d_2$. If
$\kappa \hat{s}$ is real for some $\kappa \in \mathbb{F}^*$, then $\kappa = 1$.

{\rm (b)} Let $\hat{\alpha} \in \Aut'( \GL_d( \mathbb{F} ) )$ be such that
$\hat{\alpha}( \hat{s} )$ is conjugate to $\kappa \hat{s}$ for some
$\kappa \in \mathbb{F}^*$. If $d_1 \geq d/3$, then $\hat{\alpha}( \hat{s} )$ is
conjugate to~$\hat{s}$.
\end{lem}
\begin{proof}
(a) Let~$\zeta$ denote the root of~$\Delta_2$. Then $\kappa \hat{s}$ has the
eigenvalues $\kappa$, $\kappa \zeta$, $\kappa \zeta^{-1}$ with multiplicities
$d_1$, $d_2$, $d_2$. Suppose that $\kappa \hat{s}$ is real. Then
$\kappa^{-1} = \kappa$, as $d_1 \neq d_2$. This gives our claim.

(b) Notice that the elements of $\Aut'( \GL_d( \mathbb{F} ) )$ preserve
dimensions of eigenspaces for the eigenvalue~$1$. Thus~$\hat{\alpha}( \hat{s} )$
has the eigenvalues~$1$ and~$\kappa$, each with multiplicity~$d_1$.
As~$\hat{\alpha}( \hat{s} )$ is real, it also has the eigenvalue~$\kappa^{-1}$
with multiplicity~$d_1$. Suppose that $\kappa \neq 1$. As $d_1 \geq d/3$, this
implies $d_1 = d/3$, and $1$, $\kappa$, $\kappa^{-1}$ are the eigenvalues
of~$\hat{\alpha}(\hat{s})$, and hence of $\kappa \hat{s}$, each with
multiplicity~$d/3$. It follows that $\hat{s} = \kappa^{-1}( \kappa \hat{s} )$
has the eigenvalues $\kappa^{-1}$, $1$, $\kappa^{-2}$, each with
multiplicity~$d/3$. Since~$\hat{s}$ is real, we must have
$\kappa^{-2} = \kappa$. Hence~$\hat{\alpha}( \hat{s} )$ and~$\hat{s}$ have the
same multiset of eigenvalues, and so they are conjugate.
\end{proof}

\begin{lem}
\label{HopefullyTrueGUd}
Let $\hat{G} = \GL^{\varepsilon}_d(q)$ for some $d \geq 5$ such that 
$e = \gcd( d, q - \varepsilon ) > 1$. If $\varepsilon = -1$, assume that
$(d,q) \neq (6,2)$. Let $1 \neq \hat{s} \in \hat{G}$ be semisimple and real and 
let~$\Xi$ denote the characteristic polynomial of~$\hat{s}$. Let~$\Delta_1$ 
denote the monic polynomial of degree~$1$ with root~$1$, and let~$d_1$ be the 
multiplicity of~$\Delta_1$ in~$\Xi$. 

Suppose that~$\Xi$ is as in~{\rm (a)} or~{\rm (b)} of {\rm Lemma~\ref{GUdPrep}}.
Then the following statements hold.

{\rm (a)} There is a standard $\hat{\iota}$-stable Levi subgroup~$\hat{L}$ 
of~$\hat{G}$, and there is a $\hat{G}$-conjugate $\hat{s}' \in \hat{L}$ such 
$\hat{s}'$ is real in~$\hat{L}$. Moreover, if~$\hat{\alpha} \in \Aut( \hat{G} )$ 
stabilizes~$\hat{L}$ and the $\hat{G}$-conjugacy class of~$\hat{s}$, 
then~$\hat{\alpha}( \hat{s}' )$ is conjugate to~$\hat{s}'$ in~$\hat{L}$.

{\rm (b)} The image of~$\hat{L}$ in 
$\overline{G}^{\sigma} = \PGL^{\varepsilon}_d(q)$ is of the 
form~$\overline{L}^{\sigma}$ for a standard $\iota$-stable Levi 
subgroup~$\overline{L}$ of~$\overline{G}$, and the image~$s$ of $\hat{s}$ 
in~$\overline{G}^{\sigma}$ satisfies the hypotheses of 
{\rm Lemma~\ref{ProofByHCInductionQEven}} with respect 
to~$\overline{L}^{\sigma}$.
\end{lem}
\begin{proof}
(a) Since~$\hat{s}$ is real, there is $\hat{g} \in \hat{G}$ such that 
$\hat{s}^{-2} = [\hat{g},\hat{s}] \in \SL^{\varepsilon}_d(q)$. As~$q$ is even, 
this implies that $\det(\hat{s}) = 1$.

Suppose first that~$\Xi$ is as in~(a) of Lemma~\ref{GUdPrep}. Notice that 
$d - d_1 \geq 2$, as~$\hat{s}$ is non-trivial and real. Put 
$d_1' := \lfloor d_1/2 \rfloor$, and let $\hat{L} := \hat{L}_I$ of~$\hat{G}$, 
where~$I$ corresponds to the nodes $d_1' + 1, d_1' + 2, \ldots , d - d_1' - 1$ 
of the Dynkin diagram of~$\GL_d( \mathbb{F} )$. Thus 
$\hat{L} \cong \hat{S} \times \GL^{\varepsilon}_{d - 2d_1'}( q )$, 
where~$\hat{S}$ is the standard torus of $\GL_{2d_1'}(q)$ if $\varepsilon = 1$,
and of $\GL_{d_1'}(q^2)$ if $\varepsilon = -1$.
As $1 \leq d_1' \leq d - 2$ and~$I$ is $\hat{\iota}$-invariant,~$\hat{L}$ is a 
proper, non-trivial (i.e.\ different from~$\hat{T}$), standard, 
$\hat{\iota}$-invariant Levi subgroup of~$\hat{G}$. By hypothesis, there is 
$\hat{x} \in \GL^{\varepsilon}_{d - 2d_1'}( q )$, such that~$\hat{s}$ is 
conjugate in~$\hat{G}$ to
$$\hat{s}' = \diag( \Id_{d_1'}, \hat{x}, \Id_{d_1'}) \in \hat{L}.$$
The characteristic polynomial of~$\hat{x}$ equals~$\Xi/\Delta_1^{2d_1'}$.
As~$\hat{s}$ is real, we have $\Xi^* = \Xi$, and thus 
$(\Xi/\Delta_1^{2d_1'})^* = \Xi/\Delta_1^{2d_1'}$. It follows that~$\hat{x}$
is conjugate to~$\hat{x}^{-1}$ in~$\GL^{\varepsilon}_{d-2d_1'}(q)$, and 
so~$\hat{s}$ is real in~$\hat{L}$. Now suppose 
that~$\hat{\alpha} \in \Aut( \hat{G} )$ stabilizes~$\hat{L}$ and the 
$\hat{G}$-conjugacy class of~$\hat{s}$. In particular,~$\hat{\alpha}$ stabilizes
$$[\hat{L},\hat{L}] = \{ \diag( \Id_{d_1'}, \hat{z}, \Id_{d_1'}) 
\mid \hat{z} \in \SL^{\varepsilon}_{d-2d_1'}(q) \}.$$
Since $1 = \det( \hat{s} ) = \det( \hat{s}' ) = \det( \hat{x} )$,
we have $\hat{s}' \in [\hat{L},\hat{L}]$ and thus
$$\hat{\alpha}(\hat{s}') = \diag( \Id_{d_1'}, \hat{x}', \Id_{d_1'} )$$
for some $\hat{x}' \in \SL^{\varepsilon}_{d-d_1}(q)$. As~$\hat{\alpha}$ 
stabilizes the
$\hat{G}$-conjugacy class of~$\hat{s}'$, the characteristic polynomial
of~$\hat{\alpha}(\hat{s}')$ equals~$\Xi$. Hence the characteristic polynomial
of~$\hat{x}'$ equals~$\Xi/\Delta_1^{2d_1'}$, which is the characteristic 
polynomial of~$\hat{x}$. Thus~$\hat{x}$ is conjugate to~$\hat{x}'$ in 
$\GL^{\varepsilon}_{d-2d_1'}(q)$ and so~$\hat{s}'$ is conjugate 
to~$\hat{\alpha}(\hat{s}')$ in~$\hat{L}$.

Suppose now that~$\hat{s}$ satisfies condition~(b) of Lemma~\ref{GUdPrep}. Put 
$d' := d - 4$. Consider the standard Levi subgroup $\hat{L} := \hat{L}_I$, 
where~$I$ corresponds to all the nodes of the Dynkin diagram 
of~$\GL_d( \mathbb{F} )$ without the nodes~$2$ and~$d-2$. Thus
$\hat{L} \cong \GL_2(q) \times \GL_{d'}(q) \times \GL_2(q)$ if 
$\varepsilon = 1$, and $\hat{L} \cong \GL_2(q^2) \times \GU_{d'}(q)$ if
$\varepsilon = -1$. If~$\varepsilon = 1$, let~$\hat{y} \in \GL_2(q)$ denote an
element with characteristic polynomial~$\Delta$. Then $\det(\hat{y}) = 1$, as 
the two roots of~$\Delta$ are mutually inverse. If $\varepsilon = -1$, 
then~$\Delta$ is reducible, as otherwise $\Delta \neq \Delta^{\dagger}$ by 
Lemma~\ref{FormalitiesOnPolynomialsI}(b), contradicting
Lemma~\ref{FormalitiesOnPolynomialsIV}(b). Thus $\Delta = \Delta_2 \Delta_2^*$
for a monic polynomial $\Delta_2 \neq \Delta_1$ of degree~$1$. Let 
$\zeta \in \mathbb{F}_{q^2}$ denote the root of~$\Delta_2$. 
In this case, put
$$\hat{y} := \diag( \zeta, \zeta^{-1} ) \in \GL_2( q^2 ).$$
Then $\hat{y} = \hat{\iota}_2(\hat{y})$, where~$\hat{\iota}_2$ denotes the 
standard graph automorphism of~$\GL_{2}( q^2 )$; see 
Subsection~\ref{AutomorphismsIV}. In either case of~$\varepsilon$, there 
is~$\hat{x} \in \GL_{d'}^{\varepsilon}(q)$, such that~$\hat{s}$ is conjugate 
in~$\hat{G}$ to
$$\hat{s}' = \diag( \hat{y}, \hat{x}, \hat{y} ) \in \hat{L}.$$
Notice that if $\varepsilon = -1$, the displayed element indeed lies 
in~$\GU_d(q)$, since $\hat{y} = \hat{\iota}_2(\hat{y})$.
If $d= 6$ and~$\varepsilon = 1$, we take~$\hat{x} = \hat{y}$.
It follows that~$\hat{s}'$ is real in~$\hat{L}$, as~$\hat{y}$ is real in
$\GL_2(q^{\delta})$ and~$\hat{x}$ is real in~$\GL^{\varepsilon}_{d'}(q)$, since
its characteristic polynomial equals 
$\Delta_1^{d_1} \Delta^{\lfloor d/2 \rfloor - 2}$.

As $\det(\hat{y}) = \det(\hat{x}) = 1$, we have 
$\hat{s}' \in [\hat{L},\hat{L}]$, where $[\hat{L},\hat{L}] \cong 
\SL_2( q ) \times \SL_{d'}( q ) \times \SL_2( q )$,
respectively $[\hat{L},\hat{L}] \cong \SL_{2}( q^2 ) \times \SU_{d'}(q)$, with
the obvious embeddings of the direct factors into~$\hat{L}$.
If $d = 5$, then $\SL_{d'}( q )$ is the trivial group. If $d = 6$, 
then $\SL^{\varepsilon}_{d'}( q ) \cong \SL_2( q )$ is nonabelian simple,
as $q > 2$ in this case (recall that $(d,q) \neq (6,2)$ if 
$\varepsilon = -1$). The case $d = 7$ and $\varepsilon = -1$ 
does not occur thanks to Lemma~\ref{GUdPrepPrep}(a). Hence for $d > 6$, the
group $\SL^{\varepsilon}_{d'}( q )$ is non-abelian simple and not isomorphic
to the group $\SL_2(q^{\delta})$.

Now let $\hat{\alpha} \in \Aut(\hat{G})$ normalize~$\hat{L}$, and thus
$[\hat{L},\hat{L}]$. By the remarks in the previous 
paragraph,~$\hat{\alpha}$ permutes the direct factors $\SL_2(q)$ 
of~$[\hat{L},\hat{L}]$ if $d = 6$ and $\varepsilon = 1$. In all other 
cases,~$\hat{\alpha}$, normalizes the direct 
factor~$\SL^{\varepsilon}_{d'}( q )$ of~$[\hat{L},\hat{L}]$, as
well as the factor $\SL_2(q^2)$ if $\varepsilon = -1$, whereas it 
permutes the two factors $\SL_2(q)$ if $\varepsilon = 1$.
Hence~$\hat{\alpha}(\hat{s}')$ is of the form
$$\hat{\alpha}(\hat{s}') = 
\diag( \hat{\beta}(\hat{y}), \hat{\gamma}(\hat{x}), \hat{\beta'}(\hat{y}) ),$$
with automorphisms $\hat{\beta}, \hat{\beta'}$ of $\SL_2(q^{\delta})$ and 
$\hat{\gamma}$ of $\SL_{d'}^{\varepsilon}( q )$. If $\varepsilon = -1$, we have
$\hat{\beta'} = \hat{\iota}_2 \circ \hat{\beta}$, and~$\hat{\beta}$ 
and~$\hat{\gamma}$ are the restrictions of~$\hat{\alpha}$ to the respective 
subgroups of~$[\hat{L},\hat{L}]$.
Since $\det(\hat{\beta}(\hat{y})) = 1 = \det(\hat{\beta'}(\hat{y}))$, the 
two eigenvalues of~$\hat{\beta}(\hat{y})$, respectively~$\hat{\beta'}(\hat{y})$,
are mutually inverse. As~$\hat{\alpha}$ stabilizes the $\hat{G}$-conjugacy class 
of~$\hat{s}'$, the characteristic polynomial of~$\hat{\alpha}(\hat{s}')$ 
equals~$\Xi$. This implies that $\hat{\beta}(\hat{y})$, $\hat{\gamma}(\hat{x})$ 
and~$\hat{\beta'}(\hat{y})$ have characteristic polynomial~$\Delta$, 
$\Xi/\Delta^2$, respectively~$\Delta$. Hence $\hat{\beta}(\hat{y})$ and
$\hat{\beta'}(\hat{y})$ are conjugate to~$\hat{y}$ in $\GL_2( q^{\delta} )$, and 
$\hat{\gamma}(\hat{x})$ is conjugate 
to~$\hat{x}$ in $\GL^{\varepsilon}_{d'}(q)$, which implies the result.

(b) Let $\overline{\hat{L}}$ denote the standard Levi subgroup of
$\overline{\hat{G}} := \GL_d( \mathbb{F} )$ corresponding to the subset~$I$
specified above in the two cases, and let~$\overline{L}$ denote the image 
of~$\overline{\hat{L}}$ in~$\overline{G}$ under the canonical epimorphism
\begin{equation}
\label{CanEpi}
\overline{\hat{G}} \rightarrow \overline{G}.
\end{equation}
Then~$\overline{L}$ is $\iota$-stable, and the image of~$\hat{L}$ 
in~$\overline{G}^{\sigma}$ equals~$\overline{L}^{\sigma}$. We may assume that 
$\hat{s} = \hat{s}' \in \hat{L}$. Then, in both cases, Condition~(i) of 
Lemma~\ref{ProofByHCInductionQEven} is satisfied.

To establish the second condition, suppose first that~$\hat{s}$ is
as in Case~(a) of Lemma~\ref{GUdPrep}. The natural embedding 
$\SL_{d-2d_1'}( \mathbb{F} ) \rightarrow [\overline{\hat{L}},\overline{\hat{L}}]$
yields an isomorphism 
$\SL_{d-2d_1'}( \mathbb{F} ) \rightarrow [\overline{L},\overline{L}]$ of 
algebraic groups. In particular, $[\overline{L},\overline{L}]$ is simply 
connected, and thus $C_{\overline{L}}( s )$ is connected by a theorem of 
Steinberg; see \cite[Theorem~$3.5.6$]{C2}. Suppose now that~$\hat{s}$ is as in 
Case~(b) of Lemma~\ref{GUdPrep}. 
The inverse image of $C_{\overline{L}}( s )$ in $\overline{\hat{L}}$ under the
map~(\ref{CanEpi}) equals
\begin{equation}
\label{InverseCentralizer}
\{ \hat{g} \in \overline{\hat{L}} \mid \hat{g}\hat{s}\hat{g}^{-1} = 
\kappa \hat{s} \text{\ for some\ } \kappa \in \mathbb{F}^* \}.
\end{equation}
Lemma~\ref{MultiplesOfRealElements}(a) implies that the
group~(\ref{InverseCentralizer}) equals $C_{\overline{\hat{L}}}( \hat{s} )$.
As centralizers of semisimple elements in~$\overline{\hat{L}}$ are connected by
the theorem of Steinberg cited above, the
image~$C_{\overline{L}}( s )$ of~(\ref{InverseCentralizer}) is connected as
well. Thus Condition~(ii) of Lemma~\ref{ProofByHCInductionQEven} is satisfied
in both cases.

Now let~$\alpha \in \Aut( \overline{G}^{\sigma} )$
stabilize~$\overline{L}^{\sigma}$ and the $\overline{G}^{\sigma}$-conjugacy
class of~$s$. By the discussion in Subsection~\ref{AutomorphismsIV}, there is an
an automorphism $\hat{\alpha} \in \Aut'( \overline{\hat{G}} )$, which 
stabilizes~$\hat{L}$ and such that $\hat{\alpha}( \hat{s} )$ is 
$\hat{G}$-conjugate to~$\kappa \hat{s}$ for some $\kappa \in \mathbb{F}_q$. 
Lemma~\ref{MultiplesOfRealElements}(b) implies that~$\hat{\alpha}( \hat{s} )$ is 
conjugate to~$\hat{s}$ in~$\hat{G}$, and so Condition~(iii) of 
Lemma~\ref{ProofByHCInductionQEven} follows from~(a).
\end{proof}

\begin{rem}
{\rm An analogous result as in Lemma~\ref{HopefullyTrueGUd}(b) does not hold for 
$d = 3$. For example, suppose that $G = \PSL_3(4)$. Then 
$\overline{G}^{\sigma} = \PGL_3(4)$ and ${\overline{G}^*}^{\sigma} = \SL_3(4)$. 
There are two conjugacy classes in $\overline{G}^{\sigma}$ of elements of
order~$5$, whose centralizer is the cyclic maximal torus of order~$15$. Let~$s$
be an element in one of these classes. Then 
$\mathcal{E}( {\overline{G}^*}^{\sigma}, s )$ contains a unique element, which 
has degree~$63$, is real and has $Z( {\overline{G}^*}^{\sigma} )$ in its kernel.
Thus these elements are relevant for our investigation.

Let~$\hat{s}$ be a real lift 
of~$s$ of order~$5$. Then the characteristic polynomial of~$\hat{s}$ has the 
form $\Delta_1 \Delta$, where $\Delta_1$ is as in Lemma~\ref{HopefullyTrueGUd}, 
and where~$\Delta$ is irreducible of degree~$2$ with roots of order~$5$. 
Thus~$\hat{s}$ is as in Lemma~\ref{GUdPrep}(a)(b).

The only $\iota$-stable standard Levi subgroups of~$\overline{G}^{\sigma}$ 
are~$\overline{T}^{\sigma}$ and~$\overline{G}^{\sigma}$ itself. No conjugate 
of~$s$ lies in~$\overline{T}^{\sigma}$.

Of course, there is a $\iota$-stable conjugate of a proper standard Levi 
subgroup of~$\overline{G}^{\sigma}$ meeting the conjugacy class of~$s$, but this
does not lie in any $\iota$-stable parabolic subgroups, so the argument
arising from Proposition~\ref{E1LieEvenEasy} cannot be applied.
}\hfill{$\Box$}
\end{rem}

\addtocounter{subsection}{3}
\subsection{The main result}

We are now ready to establish the main result of this section.
\addtocounter{num}{1}
\begin{lem}
\label{CliffordEstimate}
Let $\chi \in \Irr(\SL_d^\varepsilon(q))$, and let~$\hat{\chi} \in 
\Irr( \GL_d^\varepsilon(q) )$ lying above~$\chi$. Then there is a positive
integer~$e'$ with $e' \mid e$ such that $\chi(1) = \hat{\chi}(1)/e'$.
\end{lem}
\begin{proof}
Clearly, $Z( \GL_d^\varepsilon(q) )\SL_d^\varepsilon(q)$ stabilizes~$\chi$, and
$$[\GL_d^\varepsilon(q)\colon\!Z( \GL_d^\varepsilon(q) )\SL_d^\varepsilon(q)] = e.$$
As $\GL_d^\varepsilon(q)/\SL_d^\varepsilon(q)$ is cyclic, the restriction
of~$\hat{\chi}$ to~$\SL_d^\varepsilon(q)$ is a sum of~$e'$ distinct conjugates
of~$\chi$, where $e' \mid e$. Thus $\chi(1) = \hat{\chi}(1)/e'$.
\end{proof}

\begin{prp}
\label{PSUSmallDegrees}
Let $G = \PSL^{\varepsilon}_d(q)$ with~$q$ even, $d \geq 5$ and 
$e = \gcd( d, q - \varepsilon ) > 1$. Then~$G$ is not a minimal counterexample 
to {\rm \cite[Theorem~$1.1.5$]{HL1}}.
\end{prp}
\begin{proof}
Suppose that~$G$ is a minimal counterexample to \cite[Theorem~$1.1.5$]{HL1}. Let 
$(V,n,\nu)$ be as in \cite[Notation~$4.1.1$]{HL1}. Let~$\chi$ be the character
of~$V$, viewed as a character of 
$\SL^{\varepsilon}_d(q) = {\overline{G}^*}^{\sigma}$. Let
$s \in \overline{G}^{\sigma}$ be semisimple such that 
$\chi \in \mathcal{E}({\overline{G}^*}^{\sigma}, s )$.
As~$\chi$ is real,~$s$ is real by Lemma~\ref{SemisimpleCharactersLemma}(c)(iii). 
By Lemma~\ref{SemisimpleElementsInLeviSubgroups}(a), there exists a real lift
$\hat{s} \in \GL^{\varepsilon}_d(q)$ of~$s$. 

Suppose that~$\hat{s}$ is as in Cases~(a) or~(b) of Lemma~\ref{GUdPrep}. 
Lemma \ref{HopefullyTrueGUd} shows that~$s$ satisfies the hypotheses of 
Lemma~\ref{ProofByHCInductionQEven}. Hence $(G,V,n)$ has the $E1$-property, as 
we have assumed that~$G$ is a minimal counterexample.

Assume now that~$\hat{s}$ is as in one of the remaining cases of 
Lemma~\ref{GUdPrep}. Let~$\alpha$ be constructed from $\beta := \nu$ as in 
Proposition~\ref{QEvenPrep}. Then $|C_G( \alpha_{(p)} )| < M_G$ for every 
prime~$p$ dividing $|\alpha|$ and every element $\alpha_{(p)}$ of 
$\langle \alpha \rangle$ of order~$p$. Here, $M_G = q^{d(d+1)/2}$; see 
Definition~\ref{DefineMG}. By \cite[Lemma~$4.3.3$]{HL1},
it suffices to show that 
\begin{equation}
\label{Dgeq5Eq}
\chi(1) \geq (|\alpha| - 1)M_G^{1/2} = (|\alpha| - 1)q^{d(d+1)/4}. 
\end{equation}
Bounds for~$|\alpha|$ are given in Proposition~\ref{QEvenPrep}(b)(d). We will 
use these bounds below without any further reference.

If~$\hat{s}$ is as in~(c) of Lemma~\ref{GUdPrep}, then~(\ref{Dgeq5Eq}) holds.
Suppose that we are in Case~(d) of Lemma~\ref{GUdPrep}. Then 
$\chi(1) \geq [\hat{G}\colon\!\hat{C}]_{2'}/5$ by 
Lemma~\ref{CliffordEstimate}, and $|\alpha| \leq 10$. It follows 
that~(\ref{Dgeq5Eq}) is satisfied. Analogous proofs work in Case~(e) and~(g) of
Lemma~\ref{GUdPrep}.

Suppose now that we are in Case~(f) of Lemma~\ref{GUdPrep}. Then $d = e = 5$ 
and $q = 4$. Also, $\chi(1) \geq 12\cdot 2^{15}/5 > 78\,634$. On the other 
hand,~(\ref{Dgeq5Eq}) holds for $\chi(1) > 9 \cdot 2^{15} = 294\,912$. 
The character table of $G = \PSU_5(4)$ is available in GAP~\cite{GAP04}. The two
inequalities above imply that $\chi(1) = 81\,549$. Then, however,~$\chi$ is not 
invariant in~$\PGU_5(q)$. As~$\chi$ is $\alpha$-invariant and $\Phi_G$ has 
order~$4$, this implies that $\alpha \in \Phi_G$. If $|\alpha| = 2$, then~$G$ is 
not a minimal counterexample by \cite[Corollary~$4.3.2$]{HL1}.
Hence $|\alpha| = 4$. As~$d$ is odd, the graph automorphism~$\varphi^2$ of~$G$ 
has fixed subgroup isomorphic to~$\Sp_4(4)$; 
see~\cite[Proposition~$9.4.2$(b)(2)]{GLS}. We have $|\Sp_4(4)| = 979\,200$ and 
hence $|\Sp_4(4)|^{1/2} < 1\,000$. Then \cite[Lemma~$4.3.3$]{HL1} shows that 
$(G,V,n)$ has the~$E1$-property, as $\dim(V) = 81\,549 > 3 \cdot 1\,000$. 

Suppose finally that we are in Case~(g) of Lemma~\ref{GUdPrep}. Then $d = 6$ 
and $q = 2$. Moreover, $|\alpha| = 2$ or $|\alpha| = 6$. In the former case,
$(G,V,n)$ has the $E1$-property by \cite[Corollary~$4.3.2$]{HL1}. Assume then 
that $|\alpha| = 6$. Then 
$G^{\ex} := \langle \Inn(G), \alpha \rangle \cong \PGU_6( q )$ by 
Proposition~\ref{QEvenPrep}(d). Also,~$\chi$ extends to a character~$\chi^{\ex}$ 
of~$G^{\ex} $; see \cite[Remark~$4.2.5$]{HL1}. The character table of~$\PGU_6( q )$ is 
available in the Atlas as well as in GAP. If $\chi( 1 ) > 5\cdot 2^{21/2}$, 
then~(\ref{Dgeq5Eq}) is satisfied. Using this, we are left with the two cases 
$\chi(1) \in \{ 231, 385 \}$. With the notation of \cite[Lemma~$4.3.1$]{HL1}, it 
suffices to show that $\Res^{G^{\ex} }_{\langle \alpha \rangle}( \chi^{\ex} )$ contains 
each of the two irreducible, real characters of~$\langle \alpha \rangle$ with 
positive multiplicity. This is easily checked to be true.
\end{proof}

\section{The groups of type~$E_6$}

In this section we complete the proof that no finite simple group of Lie 
type~$E_6$ is a minimal counterexample to \cite[Theorem~$1.1.5$]{HL1}.

\subsection{The Weyl group and the graph automorphism}

Let $G = E_6^{\varepsilon}( q )$. The strategy to deal with these groups is the 
same as the one employed for the linear and unitary groups in 
Subsection~\ref{RemainingLinearUnitary}. Suppose that 
$\overline{G} = E_6( \mathbb{F} )$. For the explicit computations in the Weyl 
group $W( \overline{G} )$ and the root system $\Sigma = \Sigma(\overline{G})$ 
reported below, we use the GAP3 package Chevie~\cite{chevie} with its extensions 
by Jean Michel~\cite{Michel}. In particular, we follow the numbering of the 
roots used in~\cite{chevie}. Thus~$r_j$ denotes the root of 
position~$j$, $1 \leq j \leq 36$, in the list of positive roots of~$\Sigma$ 
provided by~\cite{chevie}. The simple roots are $r_1, \ldots , r_6$, where the 
root~$r_j$ is attached to the node with number~$j$, $j = 1, \ldots , 6$, in the 
Dynkin diagram of~$\Sigma$ displayed in \cite[Figure~$1$]{HL1}.  Finally, 
$s_j \in W(\overline{G})$ denotes the reflection on the hyperplane perpendicular 
to~$r_j$. The elements $n_j \in N_{\overline{G}}( \overline{T} )$ introduced in
\cite[Lemma~$6.4.2$]{C1} are lifts of the reflections~$s_j$;
see~\cite[Theorem~$7.2.2$]{C1}. As~$\mathbb{F}$ has characteristic~$2$, the 
standard torus~$\overline{T}$ has a complement in 
$N_{\overline{G}}( \overline{T} )$, namely the subgroup generated by the
elements $n_j$, $j = 1, \ldots , 6$. This follows from the Steinberg relations 
as exhibited in \cite[p.~$192$--$193$]{C1}. Although the Steinberg relations 
define the simply connected group~$\overline{G}^*$ of type~$E_6$, the claimed 
result for the adjoint group~$\overline{G}$ follows from this, 
as~$\overline{G}$ is a quotient of~$\overline{G}^*$ with a kernel contained in 
its standard torus. In the following, we will thus identify $W( \overline{G} )$ 
with a subgroup of~$N_{\overline{G}}(\overline{T})$.

The graph automorphism $\iota \in \Aut( \overline{G} )$ arises from the 
non-trivial symmetry, also denoted by~$\iota$, of the Dynkin diagram. The 
symmetry~$\iota$ of the Dynkin diagram induces an automorphism of 
$W(\overline{G})$, which is, in fact, the restriction of 
$\iota \in \Aut(\overline{G})$ to~$W(\overline{G})$. Once again, this 
automorphism of~$W(\overline{G})$ will also be denoted by~$\iota$. If
$\varepsilon = -1$, then~$\iota$ equals the automorphism of~$W(\overline{G})$
induced by~$\sigma$; see \cite[Subsection~$5.3$]{HL1} and 
Subsection~\ref{TheRemainingGroups}. As~$\varphi$ acts trivially 
on~$W(\overline{G})$, an element 
$\mu \in \Gamma_{\overline{G}} \times \Phi_{\overline{G}}$ fixes
$w \in W(\overline{G})$, if and only if~$\iota$ fixes~$w$.

\subsection{Some Levi subgroups}
For $I \subseteq \Pi$, let $\overline{L}_{I}$ denote the corresponding standard 
Levi subgroup of~$\overline{G}$; see 
\cite[Subsection~$5.3$]{HL1}.
Notice 
that~$\overline{L}_{I}$ is $\varphi$-invariant, and~$\overline{L}_{I}$ is
$\iota$-invariant if and only if~$I$ is invariant under symmetry~$\iota$ of the 
Dynkin diagram. More generally, if $J \subseteq \Sigma$, the subsystem subgroup 
of~$\overline{G}$ corresponding to the closed subsystem of~$\Sigma$ generated 
by~$J$, is denoted by~$\overline{L}_J$; see \cite[Section~$13.1$]{MaTe} for the 
definition and construction of subsystem subgroups.

For $I \subseteq \Pi$, the normalizer of the standard Levi subgroup 
$\overline{L}_{I}$ can be computed by an algorithm of Howlett~\cite{How}. In 
fact, $N_{\overline{G}}( \overline{L}_{I} )/\overline{L}_{I} \cong
\Stab_{W(\overline{G})}( I )$, and $N_{\overline{G}}( \overline{L}_{I} )$ is 
generated by $\overline{L}_{I}$ together with elements in
$N_{\overline{G}}( \overline{T} )$, whose images in $W( \overline{G} )$ generate 
$\Stab_{W(\overline{G})}( I )$. With our identification of $W( \overline{G} )$ 
with a subgroup of~$N_{\overline{G}}(\overline{T})$, we have
$N_{\overline{G}}( \overline{L}_{I} ) = 
\overline{L}_{I}\Stab_{W(\overline{G})}( I )$. 

We record a result on twisting of Levi subgroups; for this concept
see~\cite[$3.3.1$]{GeMa}.

\begin{lem}
\label{Twisting}
Let $I \subseteq \Pi$ be $\sigma$-stable, so that~$\overline{L}_I$ 
and $\Stab_{W(\overline{G})}( I )$ are $\sigma$-stable. Let
$y \in \Stab_{W(\overline{G})^{\sigma}}( I )$ and choose
$g \in \overline{G}$ with $g^{-1}\sigma(g) = y$.
Put $\overline{M} := {^g\!\overline{L}}_I$, and 
$S := {^g\Stab}_{W(\overline{G})}( I )$. 

{\rm (a)} The groups $\overline{M}$ and~$S$ are 
$\sigma$-stable and $N_{\overline{G}}( \overline{M} )$ is a semidirect product
$N_{\overline{G}}( \overline{M} ) = \overline{M}S$. Moreover,~$g^{-1}$
conjugates~$S^{\sigma}$ to the set 
$$C_{I,\sigma}( y ) := 
\{ u \in \Stab_{W(\overline{G})}(I) \mid uy\sigma(u)^{-1} = y \},$$
the $\sigma$-centralizer of~$y$.

{\rm (b)} Let $t \in Z( \overline{M}^{\sigma} )$ with 
$C_{\overline{G}}( t ) = \overline{M}$. Then~$S^{\sigma}$ acts regularly on the 
set of $\overline{G}^{\sigma}$-conjugates of~$t$ 
in~$Z( \overline{M}^{\sigma} )$. In particular,~$|S^{\sigma}|$ is even, if~$t$
is non-trivial and real in~$\overline{G}^{\sigma}$.
\end{lem}
\begin{proof}
(a) This is well known. It can be proved by a direct calculation.

(b) By~(a), we have
$N_{\overline{G}}( \overline{M} )^{\sigma} = \overline{M}^{\sigma}S^{\sigma}$,
a semidirect product.  Notice that $Z( \overline{M}^{\sigma} ) = 
Z( \overline{M} )^{\sigma}$ by \cite[Proposition~$3.6.8$]{C2}. 
It $t' \in Z( \overline{M}^{\sigma} )$ is conjugate to~$t$ by an element
$g \in \overline{G}^{\sigma}$, then 
$g \in N_{\overline{G}}( \overline{M} )^{\sigma}$ by Lemma~\ref{Conjugation0}.
It follows that~$t$ and~$t'$ are conjugate by an element of~$S^{\sigma}$. Since 
the stabilizer of~$t$ in~$S^{\sigma}$ is trivial, we get our assertion.
If~$t$ is real,~$t$ is conjugate to~$t^{-1}$ in~$S^{\sigma}$. Since~$t$ is
non-trivial and~$q$ is even, $t \neq t^{-1}$, and so a conjugating element must 
have even order.
\end{proof}

The following two cases are of particular relevance for our proof.
For simplicity, the elements of $I \subseteq \Sigma$ are denoted by their 
Chevie numbers; see~\cite{chevie}. The root~$r_{36}$ occurring below is the 
unique positive root of maximal height. The corresponding reflection~$s_{36}$ 
is~$\iota$- and hence $\sigma$-stable. 

\begin{lem}
\label{ClassTypes2Split}
{\rm (a)} Let $I = \{1,3,4,5,6\}$. Then
$\Stab_{W(\overline{G})}( I ) = \langle s_{36} \rangle$. Moreover,~$s_{36}$ is
conjugate to~$s_4$ by a $\iota$-stable element~$x$ such that 
$W( \overline{L}_I) \cap 
{^{x^{-1}}W}( \overline{L}_I ) = W( \overline{L}_{\{3,4,5\}} )$.

{\rm (b)} Let $I = \{ 2, 3, 4, 5 \}$. Then $\Stab_{W(\overline{G})}( I ) =
\langle v, w \rangle$ is a non-abelian group of order~$6$. 
Here,~$v$ and~$w$ are distinct involutions, and~$v$ is the unique 
$\iota$-stable involution in $\Stab_{W(\overline{G})}( I )$.
Also,~$v$ acts like~$\iota$ on the subdiagram of the Dynkin diagram of~$E_6$
induced by the nodes $2, 3, 4, 5$, and~$vw$ permutes its leaves cyclically.

Moreover, there is a unique $x \in W(\overline{G})$ such that the
following hold.

\begin{itemize}
\item[{\rm (i)}] The element~$x$ is a $\iota$-stable involution.
\item[{\rm (ii)}] We have ${^xv} = s_1s_6$ and ${^xw} = s_1s_3s_1s_5$.
\item[{\rm (iii)}] We have $W( \overline{L}_I) \cap 
{^{x^{-1}}W}( \overline{L}_{\{1,3,4,5,6\}} ) = W( \overline{L}_{\{2,3,5\}} )$.
\end{itemize}
\end{lem}
\begin{proof}
This can be proved with a Chevie~\cite{chevie} computation.
\end{proof}

We will need some properties of certain semisimple elements
of~$\overline{G}^\sigma$. The semisimple elements of~$\overline{G}^\sigma$ are
organized in class types. By definition, two semisimple elements
$s , s' \in \overline{G}^\sigma$ belong to the same class type, if and only
if~$C_{\overline{G}}( s )$ and~$C_{\overline{G}}( s' )$ are conjugate
in~$\overline{G}^\sigma$. A list of these class types and the corresponding 
centralizers is provided by Frank L{\"u}beck in~\cite{LL}. In fact, there is 
one list for each value of~$\varepsilon$. In Lemma~\ref{ClassTypeLemma} below, 
we will follow~\cite{LL} in the labelling of the class types. Before this, we 
have to explore the structure of a particular Levi subgroup of~$\overline{G}$.

\begin{lem}
\label{A5Levi}
Let $\overline{L} := \overline{L}_{\{1,3,4,5,6\}}$ denote the standard Levi 
subgroup of~$\overline{G}$ of type~$A_5$; see \cite[Figure~$1$]{HL1}. Then
$\overline{L} = [\overline{L}, \overline{L}] \times Z( \overline{L} )$, where 
$Z( \overline{L} )$ is a split torus of rank~$1$.
Moreover, 
$N_{\overline{G}}( \overline{L} ) = \langle \overline{L}, s_{36} \rangle$,
where~$s_{36}$ centralizes $[\overline{L}, \overline{L}]$ and inverts the 
elements of~$Z( \overline{L} )$.

Also,~$\overline{L}$ is $\sigma$-stable and $\overline{L}^{\sigma} = 
[\overline{L}, \overline{L}]^{\sigma} \times Z( \overline{L} )^{\sigma}$, where 
$[\overline{L}, \overline{L}]^{\sigma} \cong \PGL_6^{\varepsilon}(q)$ and
$Z( \overline{L} )^{\sigma} = Z( \overline{L}^{\sigma} )$ is cyclic of 
order~$q - 1$. Moreover, $N_{\overline{G}^{\sigma}}( \overline{L} ) = 
\langle \overline{L}^\sigma, s_{36} \rangle$ and
$N_{\overline{G}^{\sigma}}( \overline{L} )
= N_{\overline{G}^{\sigma}}( \overline{L}^{\sigma} )$, unless $q = 2$, in which
case $N_{\overline{G}^{\sigma}}( \overline{L}^{\sigma} ) = 
[\overline{L},\overline{L}]^{\sigma} \times H_2$, with $H_2 \cong \SL_2(2)$.
\end{lem}
\begin{proof}
As is easily checked with Chevie~\cite{chevie}, the root system~$\Sigma$ 
of~$\overline{G}$ contains a closed subsystem of type $A_5 + A_1$, generated by 
the positive roots $r_1, r_3, r_4, r_5, r_6, r_{36}$.
Let $\overline{H} := \overline{L}_{\{1,3,4,5,6,36\}}$ denote the corresponding 
subsystem subgroup of~$\overline{G}$; see \cite[Section~$13.1$]{MaTe}.

Put $\overline{H}_1 := [\overline{L}, \overline{L}]$, and 
$\overline{H}_2 := [\overline{L}_{\{36\}}, \overline{L}_{\{36\}}]$. 
Then~$\overline{H}_i$ is semisimple for $i = 1, 2$. It follows 
that~$\overline{H}_1$ is generated by the root subgroups corresponding to
the roots $\pm r_1, \pm r_3, \pm r_4, \pm r_5, \pm r_6$, and
$\overline{H}_2$ is generated by the root subgroups corresponding 
to~$\pm r_{36}$; see, e.g.\ \cite[Theorem~$8.21$(a)]{MaTe}. Since its root 
system has rank~$6$, the group~$\overline{H}$ is semisimple as well, and we 
obtain $\overline{H} = \langle \overline{H}_1, \overline{H}_2 \rangle$.
As $r_j \pm r_{36}$ is not a root for all $j \in \{ 1, 3, 4, 5, 6 \}$, the
commutator relations show that~$\overline{H}_1$ and~$\overline{H}_2$ centralize 
each other. Since~$q$ is even, the center of~$\overline{H}_2$ is trivial. Hence 
$\overline{H} = \overline{H}_1 \times \overline{H}_2$.
Since $\overline{L} \leq \overline{H}$ and $\overline{H}_1 \leq \overline{L}$,
we have $\overline{L} = [\overline{L}, \overline{L}] \times 
(\overline{L} \cap \overline{H}_2)$ with $\overline{L} \cap \overline{H}_2 =
\overline{T} \cap \overline{H}_2 = Z( \overline{L} )$. 
As $s_{36} \in \overline{H}_2$ it follows that~$s_{36}$ centralizes
$[\overline{L}, \overline{L}]$. Clearly,~$s_{36}$ inverts the elements of the 
standard torus~$\overline{T} \cap \overline{H}_2$ of $\overline{H}_2$. By 
Lemma~\ref{ClassTypes2Split}(a), we have
$\Stab_{W( \overline{G} )}( \{ 1, 3, 4, 5, 6 \} ) = \langle s_{36} \rangle$.
Hence $N_{\overline{G}}( \overline{L} )
= \langle \overline{L}, s_{36} \rangle$.

As~$\overline{L}$ and~$\overline{H}_2$ are $\iota$-stable, they are 
$\sigma$-stable as well. We obtain 
$\overline{H}^{\sigma} = [\overline{L}, \overline{L}]^{\sigma}
\times \overline{H}_2^{\sigma}$ with $\overline{H}_2^{\sigma} \cong \SL_2(q)$.
This yields $\overline{L}^{\sigma} = 
[\overline{L}, \overline{L}]^{\sigma} \times Z( \overline{L} )^{\sigma} =
[\overline{L}, \overline{L}]^{\sigma} \times Z( \overline{L}^{\sigma} )$. 
Clearly, $Z( \overline{L}^{\sigma} )$ is cyclic of order $q - 1$, 
as~$\overline{L}$ is a standard Levi subgroup of~$\overline{G}$ of 
semisimple rank~$5$. If $q = 2$, then $\varepsilon = - 1$, and 
$\overline{L}^{\sigma} = [\overline{L}, \overline{L}]^{\sigma} \cong \PGU_6(2)$; 
see the Atlas~\cite[p.~$191$ and p.~$115$]{Atlas}. Suppose then that $q > 2$. 
Then, by the tables in~\cite{LL}, the group~$\overline{L}^{\sigma}$ is the 
centralizer of a semisimple element of class type~$[8,1,1]$. These tables show 
that the center of $[\overline{L}, \overline{L}]^{\sigma}$ is trivial. Hence
$[\overline{L}^{\sigma}, \overline{L}^{\sigma}]$ is the simple group
$\PSL_6^{\varepsilon}(q)$, which has index~$3$ in 
$[\overline{L}, \overline{L}]^{\sigma}$. Moreover, 
$[\overline{L}, \overline{L}]^{\sigma}$ is isomorphic to a subgroup of 
$\Aut( \PSL_6^{\varepsilon}(q) )$. It follows from general principles about 
isogenies, that the standard torus of $[\overline{L}, \overline{L}]^{\sigma}$ 
has the same order as the standard torus of $\PGL_6^{\varepsilon}(q)$. This 
implies that $[\overline{L}, \overline{L}]^{\sigma} = \langle 
[\overline{L}^{\sigma}, \overline{L}^{\sigma}], t \rangle$, where~$t$ is an
element of the standard torus of $[\overline{L}, \overline{L}]^{\sigma}$.
In particular,~$t$ induces an inner diagonal automorphism on
$[\overline{L}^{\sigma}, \overline{L}^{\sigma}]$, and so
$[\overline{L}, \overline{L}]^{\sigma}
\cong \PGL_6^{\varepsilon}(q)$ as claimed.

As~$s_{36}$ is $\sigma$-stable, we get 
$N_{\overline{G}}( \overline{L} )^\sigma = 
\langle \overline{L}^\sigma, s_{36} \rangle$. 
To prove the final statement, suppose first that $q \neq 2$.
Observe that $C_{\overline{G}}( Z( \overline{L}^{\sigma} ) ) = \overline{L}$.
Indeed, $\overline{L} \leq 
C_{\overline{G}}( Z( \overline{L}^{\sigma} ) )$, and by the tables in \cite{LL}
we must have equality. Hence 
$$N_{\overline{G}^{\sigma}}( \overline{L}^{\sigma} ) \leq
N_{\overline{G}^{\sigma}}( Z( \overline{L}^{\sigma} ) ) \leq
N_{\overline{G}^{\sigma}}( C_{\overline{G}}( Z( \overline{L}^{\sigma} ) ) )
= N_{\overline{G}^{\sigma}}( \overline{L} ) \leq 
N_{\overline{G}^{\sigma}}( \overline{L}^{\sigma} ).$$
Now suppose that $q = 2$. Then, $\overline{L}^{\sigma} = 
[\overline{L},\overline{L}]^{\sigma}$, and thus
$[\overline{L},\overline{L}]^{\sigma} \times H_2 \leq 
N_{\overline{G}^{\sigma}}( \overline{L}^{\sigma} )$. The list of maximal 
subgroups of~$\overline{G}^{\sigma}$ given in the Atlas (see 
\cite[p.~$191$]{Atlas} and \cite[Appendix~$2$]{ModAtl}) now proves our claim.
\end{proof}

\begin{lem}
\label{A5Levi2Split}
Let the notation be as in {\rm Lemma~\ref{A5Levi}} and its proof. Let 
$g \in \overline{H}_2$ be such that $g^{-1}\sigma(g) = s_{36}$, and put
$\overline{L}' := {^g\!\overline{L}}$. Then $\overline{L}'$ is $\sigma$-stable
and $[\overline{L}', \overline{L}'] = [\overline{L}, \overline{L}]$. Moreover, 
${\overline{L}'}^{\sigma} = 
[\overline{L}', \overline{L}']^{\sigma} \times Z( \overline{L}' )^{\sigma}$, 
$Z( \overline{L'} )^{\sigma}$ is cyclic of order~$q + 1$.
\end{lem}
\begin{proof}
This follows from Lemma~\ref{Twisting} and the fact that~$g$ commutes with
$[\overline{L}, \overline{L}]$.
\end{proof}

\addtocounter{subsection}{2}
\subsection{Some class types and their centralizers}
For the concept of $e$-split Levi subgroups of~$\overline{G}$ used below,
see~\cite[$3.5.1$]{GeMa}.

\addtocounter{num}{1}
\begin{rem}
\label{A5LeviRem}
{\rm
The Levi subgroups $\overline{L}$ and ${\overline{L}'}$ described in 
Lemmas~\ref{A5Levi} and~\ref{A5Levi2Split} are $1$-split and $2$-split Levi 
subgroups of~$\overline{G}$ of type~$A_5$, respectively. They are 
representatives, with respect to $\overline{G}^{\sigma}$-conjugation, for the 
centralizers of semisimple elements of class types $[8,1,1]$ and $[8,1,2]$ (in 
the notation of~\cite{LL}), respectively.
}\hfill{$\Box$}
\end{rem}

In the course of the proof, we will consider various other class types.
The following lemma gives a construction of the corresponding centralizers. 

\begin{lem}
\label{CentralizerConstruction}
Let $s \in \overline{G}^\sigma$ be a semisimple element in one of the
class types $[8,1,k]$ or $[14,1,k]$ with $k = 1, 2$. 

Then~$C_{\overline{G}}( s )$ is connected. In fact,~$C_{\overline{G}}( s )$ is 
$\overline{G}$-conjugate to $\overline{L}_I$, where~$I$ is one of 
$\{1,3,4,5,6\}$ or $\{ 2, 3, 4, 5 \}$, according as~$s$ is of class type 
$[8,1,k]$ or $[14,1,k]$, respectively, $k = 1, 2$. 

An element~$s$ of class type $[14,2,1]$ has order~$3$ and lies in the center of 
the centralizer of an element of class type~$[14,1,1]$. In this 
case,~$C^{\circ}_{\overline{G}}( s )$
is $\overline{G}$-conjugate to $\overline{L}_{\{ 2, 3, 4, 5 \}}$,
and $|C_{\overline{G}}( s )/C^{\circ}_{\overline{G}}( s )| = 3$.
\end{lem}
\begin{proof}
This is implicit in~\cite{LL}. 
\end{proof}

The above lemma gives representatives, up to 
$\overline{G}^{\sigma}$-conjugation, of the centralizers of the class types
in the split case, i.e.\ where $C^{\circ}_{\overline{G}}( s )$ is 
$\overline{G}^{\sigma}$-conjugate to a standard Levi subgroup of~$\overline{G}$.
Representatives for the $2$-split cases will be constructed in the following 
lemma. Recall the significance of the symbol~$\iota$ discussed
in the introductory paragraphs of this subsection.

\begin{table}
\caption{\label{TwistingsForClassTypes} Twisting elements for some class types
(explanations in the proof of Lemma~\ref{ClassTypeLemma})}
$$\begin{array}{ccccccc}\\ \hline
\text{\rm class type} & I & \Stab_{W(\overline{G})}( I ) & \varepsilon & y & z 
                      & |C_{I,\sigma}( y )| \\ \hline
\ [8,1,1] & \{1,3,4,5,6\} & \langle s_{36} \rangle & \pm 1 & 1 & 1 & 2 \\ 
\ [8,1,2] &               &        &       & s_{36} & s_4 & 2 \\
\ [14,1,1] & \{ 2, 3, 4, 5 \} & \langle v, w \rangle & 1 & 1 & 1 & 6 \\
\ [14,1,2] &                  &          &   & v & s_1s_6 & 2 \\
\ [14,1,2] &                  &          & -1 & 1 & 1 & 2 \\
\ [14,1,1] &                  &          &    & v & s_1s_6 & 6 \\ \hline
\end{array}
$$
\end{table}

\begin{lem}
\label{ClassTypeLemma}
Let $s \in \overline{G}^\sigma$ be a semisimple element in one of the 
class types $[8,1,k]$, $[14,1,k]$ with $k = 1, 2$, or $[14,2,1]$. 
In the latter case assume that $q \neq 2$.
Suppose that~$s$ is real in~$\overline{G}^{\sigma}$. Then there is a
$\sigma$-stable standard Levi subgroup~$\overline{L}$ of~$\overline{G}$, and
there is $s' \in \overline{L}^{\sigma}$ such that~$s$ and~$s'$ are 
$\overline{G}^\sigma$-conjugate, and the following conditions hold.
\begin{itemize}
\item[{\rm (i)}] The element~$s'$ is real in~$\overline{L}^\sigma$.
\item[{\rm (ii)}] The centralizer~$C_{\overline{L}}( s' )$ is connected.
\item[{\rm (iii)}] If $\alpha \in \Aut(G)$ stabilizes~$\overline{L}^{\sigma}$
and the $\overline{G}^{\sigma}$-conjugacy class of~$s$,
then~$\alpha(s')$ and~$s'$ are conjugate in~$\overline{L}^{\sigma}$.
\end{itemize}
\end{lem}
\begin{proof}
By Lemma~\ref{CentralizerConstruction}, the group $C^{\circ}_{\overline{G}}(s)$ 
is $\overline{G}$-conjugate to $\overline{L}_{I}$, where~$I$ is one of 
$\{ 1, 3, 4, 5, 6 \}$ or $\{ 2, 3, 4, 5 \}$; notice that~$I$ is $\iota$-stable 
and thus also $\sigma$-stable. Put 
$\overline{K} := \overline{L}_{ \{ 1, 3, 4, 5, 6 \} }$. 
The groups $\Stab_{W(\overline{G})}( I )$ have been determined in 
Lemma~\ref{ClassTypes2Split}, whose notation is used in Column~$3$ of 
Table~\ref{TwistingsForClassTypes}. Let~$x$ be a $\iota$-stable element with the 
properties exhibited in Lemma~\ref{ClassTypes2Split}.
The elements $y \in \Stab_{W(\overline{G})}( I )$ given in the fifth column of 
Table~\ref{TwistingsForClassTypes} are $\sigma$-stable.
Let us write $z := {^x\!y}$, so that $z \in W( \overline{K} )^{\sigma} \leq 
\overline{K}^\sigma$, where the first inclusion follows from 
Lemma~\ref{ClassTypes2Split}. The elements~$z$ are given in the sixth column of
Table~\ref{TwistingsForClassTypes}.

Fix~$I$,~$\varepsilon$ and $k \in \{ 1, 2 \}$. We claim that 
twisting~$\overline{L}_I$ with the element~$y$ associated to these data in 
Table~\ref{TwistingsForClassTypes} (see \cite[$3.3.1$]{GeMa}), yields 
representatives for the centralizers of the corresponding class types
$[8,1,k]$ respectively $[14,1,k]$.
To see this, first observe that the two $y$-elements associated to distinct~$k$s
lie in distinct $\sigma$-conjugacy classes of $\Stab_{W(\overline{G})}(I)$. This 
is clear for $I = \{ 1, 3, 4, 5, 6 \}$, as then $\Stab_{W(\overline{G})}( I )$ 
has order~$2$. For $I = \{ 2, 3, 4, 5 \}$, the corresponding $\sigma$-centralizers
$C_{I,\sigma}(y)$ of~$y$ (see Lemma~\ref{Twisting}(b)) have different orders.
Next, choose $h \in \overline{K}$ with $h^{-1}\sigma(h) = z$; if $z = 1$, choose 
$h = 1$. Putting $g := hx$, we get $g^{-1}\sigma(g) = y$. Now let
$\overline{M} := {^g\overline{L}}_I$. By Lemma~\ref{Twisting}(b), the number of 
$\sigma$-stable elements of~$Z( \overline{M} )$ with centralizer~$\overline{M}$ 
equals $|({^g\Stab}_{W(\overline{G})}(I))^{\sigma}|$. By Lemma~\ref{Twisting}(a), 
the latter number equals~$|C_{I,\sigma}(y)|$. On the other hand, the number in 
question can be read off the tables in~\cite{LL} as the denominator of the 
leading coefficient in the polynomial giving the number of conjugacy classes in 
a given class type. This proves our claim.

We now prove our assertions if~$s$ has class type $[8,1,k]$ or $[14,1,k]$, 
$k = 1, 2$, with $\overline{L} := \overline{K} = 
\overline{L}_{ \{ 1, 3, 4, 5, 6 \} }$. By the claim in the previous paragraph, 
$\overline{M} = {^h({^x\overline{L}}_I)}$ is the centralizer 
in~$\overline{G}$ of an element of class type~$s$, i.e.\ $C_{\overline{G}}( s )$ 
is $\overline{G}^{\sigma}$-conjugate to $\overline{M}$. In particular,~$s$ is 
$\overline{G}^{\sigma}$-conjugate to an element $s' \in Z( \overline{M} )$ with 
$C_{\overline{G}}( s' ) = \overline{M}$. 

(i) As $Z( ^x\overline{L}_I ) \leq {^x\overline{T}} = \overline{T} \leq \overline{K}$ 
and $h \in \overline{K}$, we have 
$Z(  \overline{M} ) = Z(^g\overline{L}_I) \leq \overline{K}$, and 
thus $s' \in Z( \overline{M} )^{\sigma} \leq \overline{K}^{\sigma}$. Now~$s'$ 
is real in~$\overline{G}^{\sigma}$ by hypothesis. Lemma~\ref{Twisting}(b) implies 
that~$s'$ is 
inverted by a $\sigma$-stable element of ${^{g}\Stab_{W(\overline{G})}( I )} = 
{^{h}({^x\Stab_{W(\overline{G})}( I )})}$. Hence~$s'$ is real
in~$\overline{K}^{\sigma}$, as
${^x\Stab_{W(\overline{G})}( I )} \leq \overline{K}$ and $h \in \overline{K}$.


(ii) By construction, the centralizer~$C_{\overline{K}}( s' )$ is 
$\overline{K}$-conjugate to $^x\overline{L}_I \cap \overline{K}$. The latter is
the centralizer in~$\overline{G}$ of the subtorus 
$^xZ(\overline{L}_I )Z( \overline{K} )$ of~$\overline{T}$;
see~\cite[Propositions~$12.6$ and~$3.9$]{MaTe}. Hence
$^x\overline{L}_I \cap \overline{K}$ is a Levi subgroup of~$\overline{G}$ and
thus connected. To be specific, Lemma~\ref{ClassTypes2Split} implies that 
$^x\overline{L}_I \cap \overline{K}$ is $\overline{G}$-conjugate to the
standard Levi subgroup $\overline{L}_{ \{ 3, 4, 5 \} }$, respectively 
$\overline{L}_{ \{ 2, 3, 5 \} }$.
Hence~$C_{\overline{K}}( s' )$ is connected.

(iii) We may assume that $s = s'$. We begin with a special case. Suppose first 
that $\alpha = \mu \in \Gamma_G \times \Phi_G$. Then $\alpha = \mu$ certainly
stabilizes~$\overline{K}^{\sigma}$, as~$\overline{K}$ is $\sigma$-stable by 
Lemma~\ref{A5Levi}, and~$\mu$ commutes with~$\sigma$. By the same 
argument,~$\overline{L}_I$ and~$\overline{L}_I^{\sigma}$ are $\mu$-stable.

Suppose that~$s'$ is conjugate in $\overline{G}^{\sigma}$ to~$\mu(s')$. Then 
$g^{-1}s'g$ is conjugate in~$\overline{G}$ to $\mu(g)^{-1}\mu(s')\mu(g)$. 
Now~$g^{-1}s'g$ and~$\mu(g)^{-1}\mu(s')\mu(g)$ are contained 
in~$Z( \overline{L}_I )$, as~$\overline{L}_I$, and hence~$Z( \overline{L}_I )$ 
are $\mu$-stable. By Lemma~\ref{Conjugation0}, there is 
$u \in \Stab_{W(\overline{G})}( I )$ such that
$ug^{-1}s'gu^{-1} = \mu(g)^{-1}\mu(s')\mu(g)$. As $g = hx$ and $\mu(x) = x$,
we obtain 
$$\mu(s') = [\mu(h) {xux^{-1}} h^{-1}] s' [\mu(h) {xux^{-1}} h^{-1}]^{-1}.$$
Now $h \in \overline{K}$ and $\overline{K}$ is $\mu$-stable. Moreover,
$xux^{-1} \in \overline{K}$ by Lemma~\ref{ClassTypes2Split}. Thus
$\mu(h) {xux^{-1}} h^{-1} \in \overline{K}$ and so~$s'$ and~$\mu(s')$ are
conjugate in~$\overline{K}$. As $C_{\overline{K}}( s' )$ is connected,~$s'$ 
and~$\mu(s')$ are conjugate in~$\overline{K}^{\sigma}$; see
\cite[Example~$1.4.10$]{GeMa}.

Now let $\alpha \in \Aut(G)$ stabilize~$\overline{K}^{\sigma}$ 
and the $\overline{G}^{\sigma}$-conjugacy class of~$s'$. Then $\alpha = 
\ad_c \circ \mu$ with $c \in \overline{G}^{\sigma}$ and 
$\mu \in \Gamma_G \times \Phi_G$. As~$\mu$ stabilizes~$\overline{K}^{\sigma}$,
we must have $c \in N_{\overline{G}^{\sigma}}( \overline{K}^{\sigma} )$.
Now $\alpha( s' ) = c \mu(s') c^{-1}$ is conjugate in~$\overline{G}^{\sigma}$ 
to~$s'$. Hence~$\mu(s')$ is conjugate to~$s'$ in~$\overline{G}^{\sigma}$.
By what we have already proved, there is $c' \in \overline{K}^{\sigma}$ such 
that $\mu(s') = c' s' {c'}^{-1}$. Thus $\alpha(s') = c\mu(s')c^{-1} =
(cc') s' (cc')^{-1}$ and so~$\alpha(s')$ and~$s'$ are conjugate in
$N_{\overline{G}^{\sigma}}( \overline{K}^{\sigma} )$. 

As~$s'$ is real in~$\overline{K}^{\sigma}$ and $\overline{K}^{\sigma} =
[\overline{K},\overline{K}]^{\sigma} \times Z( \overline{K} )^{\sigma}$ by 
Lemma~\ref{A5Levi}, we must have $s' \in [\overline{K},\overline{K}]^{\sigma}$. 
By the structure of $N_{\overline{G}^{\sigma}}( \overline{K}^{\sigma} )$ given 
in Lemma~\ref{A5Levi}, the elements~$\alpha(s')$ and~$s'$ are conjugate 
in~$\overline{K}^{\sigma}$. This completes our proof if~$s$ is of class type 
$[8,1,k]$ or $[14,1,k]$, $k = 1, 2$.

We finally prove our assertions for elements of class type $[14,2,1]$.
Suppose we are in the situation where~$\overline{M}$ is the centralizer of an  
element of class type~$[14,1,1]$, i.e.\ $I = \{ 2, 3, 4, 5\}$. Thus there 
$t \in Z( \overline{M}^{\sigma} )$  
of order~$3$ such that $C^{\circ}_{\overline{G}}( t ) = \overline{M}$ and
$|(C_{\overline{G}}( t )/C^{\circ}_{\overline{G}}( t ))^{\sigma}| = 3$; 
see Lemma~\ref{CentralizerConstruction}. Now $C_{\overline{G}}( t ) \leq 
N_{\overline{G}}( C^{\circ}_{\overline{G}}( t ) ) = 
N_{\overline{G}}( \overline{M} ) = 
\overline{M} {^g\Stab_{W(\overline{G})}( I )}$. Thus~$t$ is
centralized by the elements of order~$3$ in
${^g\Stab_{W(\overline{G})}( I )}$. As~$t$ is real, it is inverted by the 
involutions in ${^g\Stab_{W(\overline{G})}( I )}$. Now every element of
${^g\Stab_{W(\overline{G})}( I )}$ is $\sigma$-stable, and so~$t$ is real
in~$\overline{K}^{\sigma}$. Also, $C_{\overline{L}}( t )$ is not connected, as 
the Sylow $3$-subgroup of ${^g\Stab_{W(\overline{G})}( I )}$ centralizes~$t$ and 
normalizes $\overline{M} \cap \overline{K} = C^{\circ}_{\overline{K}}( t )$.
By the proof of (ii) above, the group $\overline{M} \cap \overline{K}$ is 
$\overline{G}$-conjugate to $\overline{L}_{ \{ 2, 3, 5 \} }$. A Chevie
computation shows that the latter is $\overline{G}^{\sigma}$-conjugate to 
$\overline{L}_{ \{ 1, 4, 6 \} } \leq \overline{K}$.

Now $\overline{K}^{\sigma} = 
[\overline{K}, \overline{K}]^{\sigma} \times Z( \overline{K} )^{\sigma}$, 
where $Z( \overline{K} )^{\sigma}$ is a cyclic group of order $q - 1$ and
$[\overline{K}, \overline{K}]^{\sigma} \cong \PGL^{\varepsilon}_6(q)$;
see Lemma~\ref{A5Levi}. 
As~$t$ is real in~$\overline{K}^{\sigma}$, we must
have $t \in [\overline{K}, \overline{K}]^{\sigma}$. Let $\hat{G} =
\GL_6^{\varepsilon}(q)$ and consider the natural map 
$\hat{G} \rightarrow [\overline{L}, \overline{L}]^{\sigma}$.
Let~$\hat{t}$ be a real lift of~$t$ in~$\hat{G}$. Then~$\hat{t}$ has order~$3$,
since~$\hat{t}^3$ is a real central element of~$\hat{G}$.
The structure of $C_{\overline{K}}( t )$ exhibited above, together with
Lemmas~\ref{FormalitiesOnPolynomialsII} and~\ref{FormalitiesOnPolynomialsIII},
imply that~$\hat{t}$ has eigenvalues $1, \zeta, \zeta^{-1}$, each with
multiplicity~$2$, where $\zeta \in \mathbb{F}_{q^\delta}$ has order~$3$.
In particular, the minimal polynomial of~$\hat{t}$ is as in~(a) of
Lemma~\ref{GUdPrep}.

Let $\overline{L} := \overline{L}_{\{3,4,5\}} \leq \overline{K} = 
\overline{L}_{\{1,3,4,5,6\}}$. Then~$\overline{L}$ is a $\sigma$-stable standard 
Levi subgroup of~$\overline{G}$. Lemma~\ref{HopefullyTrueGUd}, applied 
to~$\hat{t}$ and $[\overline{K},\overline{K}]^{\sigma}$, shows that there is an 
element $t' \in {\overline{L}}^{\sigma}$ conjugate to~$t$ in 
$[\overline{K},\overline{K}]^{\sigma}$, such that~$t$ satisfies~(i) and~(iii). 
Moreover, $C_{\overline{L}}( t' )$ is connected, since the standard Levi 
subgroup $\overline{L}^* \leq \overline{G}^*$ dual to~$\overline{L}$
has connected center; see~\cite{LL}. This completes our proof for the case 
that~$s$ has class type $[14,2,1]$.
\end{proof}

\begin{lem}
\label{NonReal}
The elements of the class types $[3,2,3]$, $[4,1,1]$, $[5,1,1]$, $[9,1,2]$,
$[9,1,1]$, $[14,2,2]$ and $[14,1,3]$ are not real.
\end{lem}
\begin{proof}
If~$s$ is of class type $[14,2,2]$ or $[3,2,3]$, then the center of
$C_{\overline{G}^{\sigma}}( s )$ has order~$3$, and there are two conjugacy
classes of this class type; thus~$s$ and~$s^{-1}$ lie in different
$\overline{G}^{\sigma}$-conjugacy classes. For the class types $[4,1,1]$,
$[5,1,1]$ and $[14,1,3]$, the claim follows from the last statement of
Lemma~\ref{Twisting}(b).

Suppose that the semisimple element $s \in \overline{G}^{\sigma}$ is of class
type~$[9,1,1]$ or~$[9,1,2]$ and that~$s$ is
real in~$\overline{G}^{\sigma}$. Then~$s$ is real in~$\overline{G}$. Let
$J := \{ 1, 3, 4, 5 \}$ and put $\overline{M} := \overline{L}_J$.
As~$C_{\overline{G}}( s )$ is $\overline{G}$-conjugate to~$\overline{M}$ by
\cite{LL}, the center of~$\overline{M}$ contains a real element~$t$ with
centralizer~$\overline{M}$. By Lemma~\ref{Conjugation0}, there is an element in
$\Stab_{W(\overline{G})}(J)$ inverting~$t$.

Now $\overline{M} = \overline{L}_J \leq \overline{L}_I$ with
$I = \{ 1, 3, 4, 5, 6 \}$. By Table~\ref{TwistingsForClassTypes},
we have $\Stab_{W(\overline{G})}(J) = \Stab_{W(\overline{G})}(I) = 
\langle s_{36} \rangle$. By Lemma~\ref{A5Levi}, we may write $t = t_1 t_2$ with
$t_1 \in [\overline{L}_I,\overline{L}_I]$ and $t_2 \in Z( \overline{L}_I )$.
Moreover, $t_1$ is centralized by~$s_{36}$. Now $t_1 \neq t_1^{-1}$ as
$C_{\overline{G}}( t ) = \overline{L}_J$ and~$q$ is even. Hence~$t$ is not
inverted by $s_{36}$, a contradiction.
\end{proof}

\addtocounter{subsection}{4}
\subsection{Modification of~$\nu$}
The following is the analogue of Proposition~\ref{QEvenPrep}.

\addtocounter{num}{1}
\begin{prp}
\label{QEvenPrepE6}
Let~$G$ be as in
{\rm Hypothesis~\ref{GroupsOfEvenCharacteristic}(b)(iii),(iv)}, and let 
$\beta \in \Aut(G)$. Write $\beta = \ad_h \circ \mu$ for some 
$h \in \overline{G}^{\sigma}$ and some $\mu \in \Gamma_G \times \Phi_G$.
Then there is $g \in G$ such that $\alpha := \ad_g \circ \beta$ has even order 
and the following statements hold.

{\rm (a)} We have $|C_G( \alpha_{(p)} )| < q^{48}$ for every prime~$p$ 
dividing~$|\alpha|$ and every element $\alpha_{(p)}$ of $\langle \alpha \rangle$ 
of order~$p$, unless $p = 2$ and $\alpha_{(2)}$ is a graph automorphism of~$G$.
In the latter case, $|C_G( \alpha_{(p)} )| < q^{52}$.

{\rm (b)} The order of~$\alpha$ divides~$3 \delta f$.

{\rm (c)} If $h \in G$, then $|\alpha|$ divides $\delta f$.

{\rm (d)} If $\varepsilon = -1$ and $|\mu|$ is even, then $|\alpha|$ divides $2f$.
\end{prp}
\begin{proof}
Let $\overline{L} := \overline{L}_{\{ 2, 3, 4, 5 \}}$.
Then $[\overline{L},\overline{L}]$ is a simple group of type~$D_4$, and thus
$[\overline{L},\overline{L}]^{\sigma} \cong \P\Omega^\varepsilon_8(q)$, as~$q$ is 
even. This is a non-abelian simple group, so that  
$[\overline{L},\overline{L}]^{\sigma} = 
[\overline{L}^{\sigma},\overline{L}^{\sigma}]$. It follows that
$\overline{L}^{\sigma} = [\overline{L}^{\sigma},\overline{L}^{\sigma}] \times Z$
with~$Z$ isomorphic to $[q-1] \times [q-1]$, respectively to $[q^2 - 1]$; 
see~\cite{LL}. Write~$Z_3$ for the Sylow $3$-subgroup of~$Z$. Then~$Z_3$ is
invariant under $\Gamma_G \times \Phi_G$. 
Recall from Subsection~\ref{TheRemainingGroups} that 
$\overline{G}^{\sigma} = G \overline{T}^{\sigma}$, and thus
$\overline{G}^{\sigma} = G \overline{L}^{\sigma}$.
Now $[\overline{L}^{\sigma},\overline{L}^{\sigma}] \leq 
[\overline{G}^{\sigma},\overline{G}^{\sigma}] = G$.
Hence $Z \not\leq G$ and so $\overline{G}^{\sigma} = GZ_3$. 

Thus there is $y \in G$ such that $\beta' := \ad_y \circ \beta$ is of the form
\begin{equation}
\label{FormOfAlpha}
\beta' = \ad_t \circ \mu,
\end{equation}
with $t \in Z_3$. We choose $t = 1$ if $h \in G$. Equation~(\ref{PowerFormula}) 
and the fact that~$Z_3$ is invariant in $\Gamma_G \times \Phi_G$, imply that 
\begin{equation}
\label{FormOfAlphaI}
{\beta'}^l = 
\ad_{t'} \circ \mu^l
\end{equation}
with $t' \in Z_3$ for every integer~$l$.

(a) Let~$p$ be a prime dividing~$|\beta'|$. The claim that $|C_G( \beta'_{(p)} )| 
< q^{48}$, unless $p = 2$ and $\beta'_{(2)}$ is a graph automorphism of~$G$, in
which case $|C_G( \beta'_{(p)} )| < q^{52}$. This claim follows from
Lemma~\ref{CentralizersEstimates}, except if $\beta'_{(p)}$ is an inner-diagonal 
automorphism. In the latter case, $\beta'_{(p)} = \ad_{t'}$ for some $t' \in Z_3$ 
by Equation~(\ref{FormOfAlphaI}). Since~$t'$ is an element of order~$3$ in the
center of $\overline{L}^{\sigma}$, the tables in~\cite{LL} imply that 
$|C_{\overline{G}^{\sigma}}( \beta'_{(3)} )| = |C_{\overline{G}^{\sigma}}( t' )|$ 
is bounded above by $(q - \varepsilon)|\P\Omega^{\varepsilon}_{10}( q )|$. A rough
upper bound for the latter order is~$q^{48}$, proving our claim.

Once more with Equation~(\ref{FormOfAlphaI}), we see that~$|\beta'|$ is even, if
and only if~$|\mu|$ is even. In this case we put $\alpha := \beta'$.
Then~$\alpha$ satisfies the assertions. Suppose now that~$|\mu|$ is odd. Then 
$\mu \in \Phi_G$. In particular,~$\mu$ is a power of~$\varphi$. Let 
$g \in [\overline{L}, \overline{L}]$ be the product of a $\iota$-stable set of 
three pairwise commuting $\varphi$-stable root elements. Then~$g$ is a 
$\iota$-stable involution in~$G$.  Moreover,~$g$ belongs to the class~$3A1$ of 
unipotent elements in the notation of~\cite{Lawther}. Put 
$\alpha := \ad_g \circ \beta'$. As~$\mu$ is a power of~$\varphi$ and thus
fixes~$g$, and as~$g$ commutes with~$t$, we have $\alpha = \beta' \circ \ad_g$.
Also, $C_{G}( \alpha_{(p)} ) = C_{G}( \beta'_{(p)} )$ if~$p$ is odd, and
$|C_{\overline{G}^{\sigma}}( \alpha_{(2)} )| = 
|C_{\overline{G}^{\sigma}}( g )| = 
q^{31} (q - \varepsilon)^3 (q + \varepsilon)^2 (q^2 + \varepsilon q + 1) \leq 
q^{48}$. In case of $\varepsilon = 1$, the order
of~$C_{\overline{G}^{\sigma}}( g )$ has been computed by Mizuno
(see \cite[Lemma~$4.2$]{Mizuno}), in case $\varepsilon = -1$
by Gunter Malle (see \cite[Proposition~$3$]{MalleGreen}). This completes the 
proof of~(a).

The statements (b)--(d) follow from Lemma~\ref{OrderBound}.
\end{proof}

\addtocounter{subsection}{1}
\subsection{The main result}
We are now ready to finish the proof for the groups~$E_6^{\varepsilon}(q)$.

\addtocounter{num}{1}
\begin{prp}
\label{E6SmallDegrees}
Suppose that $G = E_6^\varepsilon(q)$ with~$q$ even and 
$3 \mid q - \varepsilon$. Then~$G$ is not a minimal counterexample to 
{\rm \cite[Theorem~$1.1.5$]{HL1}}.
\end{prp}
\begin{proof}
Suppose that~$G$ is a minimal counterexample to \cite[Theorem~$1.1.5$]{HL1},
let~($V,n)$ be a pair without the 
$E1$-property and let~$\chi$ denote the character of~$V$. Let~$\alpha$ be
constructed as in Proposition~\ref{QEvenPrepE6} from $\beta := \nu$.
By \cite[Lemma~$4.3.3$]{HL1}
and Proposition~\ref{QEvenPrepE6} we have 
\begin{equation}
\label{CriticalBound}
\chi(1) \leq (|\alpha| - 1)q^{26}.
\end{equation}
By Proposition~\ref{QEvenPrepE6}(b), we have $|\alpha| \leq 3 \delta f$. Hence 
$|\alpha| \leq q$ if $f > 3$. For $f = 1, 2, 3$ we have $3 \delta f = 6, 6, 18$.

Now consider the lists of character degrees and their multiplicities 
of~${\overline{G}^*}^{\sigma}$ 
provided by Frank L{\"u}beck in~\cite{LL2}. In fact there are two such lists, 
one for each $\varepsilon \in \{ -1, 1 \}$; all the considerations below are 
true for each of these lists. It turns out that exactly~$116$ of the character 
degrees of~${\overline{G}^*}^{\sigma}$ are odd. Let 
$(d_i,m_i)$, $i = 1, \ldots , 116$ denote the entries of L{\"u}beck's list 
corresponding to the odd degree characters, where~$d_i$ denotes the degree of a 
character and $m_i$ its multiplicity, $i = 1, \ldots , 116$. We choose the 
notation that that $i < j$ if and only if $(d_i,m_i)$ comes before $(d_j,m_j)$ 
in the full list~\cite{LL2}. The quantities $d_i$, $m_i$, $i = 1, \ldots , 116$ 
are polynomials in~$q$ with rational coefficients. Using GAP~\cite{GAP04} and 
some obvious estimates, one checks that $d_i > q^{27}$ for $q > 8$ and $i > 14$. 
We also have $d_i > 17 \cdot 8^{26}$ if $q = 8$ and $i > 14$, as well as 
$d_i > 5 \cdot 4^{26}$ if $q = 4$ and $i > 14$. Finally, if $q = 2$, we have 
$d_i > 5 \cdot 2^{26}$ for all $i > 20$. However, $m_i = 0$ for $q = 2$ and 
$15 \leq i \leq 19$.  By~(\ref{CriticalBound}) we may thus assume that 
$\chi( 1 ) = d_i$ for some $1 \leq i \leq 14$.

The semisimple elements of~$\overline{G}^{\sigma}$ and their centralizers are 
enumerated and described in the tables of~\cite{LL}. Again, there are two such 
tables, but the statements below hold for each of these tables. As already 
mentioned in the paragraph preceding Lemma~\ref{ClassTypeLemma}, the semisimple 
conjugacy classes are combined to class types. Using 
Lemma~\ref{SemisimpleCharactersLemma}(a), we find that the characters of degrees
$d_i$, $i = 1, \ldots , 14$ correspond, via Lusztig's generalized Jordan 
decomposition of characters, to the following class types:
$[1,1,1]$, $[4,1,1,]$, $[8,1,2]$, $[8,1,1]$, $[14,2,2]$, $[14,2,1]$, $[14,1,3]$, 
$[14,1,2]$, $[14,1,1]$, $[5,1,1]$, $[9,1,2]$, $[9,1,1]$, $[3,2,3]$, $[3,2,1]$,
where the order of the class types in this list corresponds to the order 
$d_1, d_2, \ldots , d_{14}$ if $\varepsilon = 1$. If $\varepsilon = -1$, this 
order has to be changed according to the permutation $(5,6)(7,9)$. The trivial 
character of degree $d_1 = 1$ corresponds to the trivial element, i.e.\ to the 
class type $[1,1,1]$. By Lemma~\ref{NonReal}, the elements of the class types 
$[4,1,1]$, $[14,2,2]$, $[14,1,3]$, $[5,1,1]$, $[9,1,2]$, $[9,1,1]$ and $[3,2,3]$ 
are not real.  In all these cases,~$\chi$ is not real by 
Lemma~\ref{SemisimpleCharactersLemma}(b)(ii).

In the following, we may assume that~$s$ belongs to one of the remaining~$6$
class types. We begin by ruling out the three characters corresponding to the
class type~$[3,2,1]$ (this class type consists of a unique conjugacy class of 
elements of order~$3$) as candidates for~$\chi$. Suppose that~$\chi$ corresponds 
to~$s$ in class~$[3,2,1]$. Using the exact value of~$\chi(1)$, we find
$\chi(1) > q^{27}/3$. Clearly, $3 \delta f < q/3$ for $f \geq 6$. By 
Proposition~\ref{QEvenPrepE6}(b) and~(\ref{CriticalBound}), we are left with
the cases $f \leq 5$. Observe that~$\chi$ is not invariant 
in~$\overline{G}^{\sigma}$. On the other hand,~$\chi$ is invariant under~$\nu$,
and thus also under~$\beta'$, with $\beta' = \ad_y \circ \nu = \ad_t \circ \mu$
as in~(\ref{FormOfAlpha}). This implies that $t \in G$ if $3 \nmid f$.
In the latter case, $|\alpha| \leq \delta f$ by 
Proposition~\ref{QEvenPrepE6}(c). Since $\delta f - 1 < q/3$ for $f = 2, 4, 5$,
we are left with the cases $f = 1$ and $f = 3$. If $f = 1$, then $\alpha$ is an 
involution, and we are done with \cite[Corollary~$4.3.2$]{HL1}. Suppose finally 
that $f = 3$. Assume first that~$|\mu|$ is odd. Then~$|\alpha|$ divides~$18$ by 
Proposition~\ref{QEvenPrepE6}(b), and $|C_G( \alpha_{(p)} )| < q^{48}$ for 
$p = 2, 3$ by Proposition~\ref{QEvenPrepE6}(a). Since $q^{27}/3 > 17 q^{24}$, 
this case is ruled out by \cite[Lemma~$4.3.3$]{HL1}. Suppose finally 
that $|\mu|$ is even. Then $|\alpha| \in \{ 2, 6 \}$ by 
Proposition~\ref{QEvenPrepE6}(d). Once more by \cite[Corollary~$4.3.2$]{HL1},
we may assume that $|\alpha| = 6$. Put $G^{\ex} = \langle \Inn(G), \nu \rangle$ as in 
\cite[Definition~$4.2.3$]{HL1} and let~$\chi^{\ex}$ denote the extension of~$\chi$ 
to~$G^{\ex} $ as in \cite[Remark~$4.2.5$]{HL1}. The elements of order~$6$ in 
$\langle \alpha \rangle$ square to elements of order~$3$, and $\alpha_{(2)}$ is 
the unique element of order~$2$ in $\langle \alpha \rangle$. Thus 
$|C_G( \alpha' )| < q^{48}$ for every element
$\alpha' \in \langle \alpha \rangle$ of order~$3$ or~$6$, and
$|C_G( \alpha_{(2)} )| < q^{52}$; see Proposition~\ref{QEvenPrepE6}(a). Now 
$q^{27}/3 > q^{26} + 4 q^{24}$. If follows from \cite[Lemma~$2.5.5$]{HL1} that 
the restriction of~$\chi^{\ex}$ to $\langle \alpha \rangle$ contains the two real 
irreducible characters of $\langle \alpha \rangle$ with positive multiplicity. 
Then \cite[Lemma~$4.3.1$]{HL1} implies that this case does not yield a 
counterexample. 

We are thus left with the cases that~$s$ belongs to one of the class 
types~$[8,1,k]$ or~$[14,1,k]$, $k = 1, 2$, or $[14,2,1]$. As we have assumed 
that~$G$ is a minimal counterexample, these cases are ruled out by
Lemmas~\ref{ClassTypeLemma}(b) and~\ref{ProofByHCInductionQEven}.
\end{proof}

\section*{Acknowledgements}
The authors thank, in alphabetical order, Thomas Breuer, Xenia Flamm, Meinolf 
Geck, Jonas Hetz, Frank Himstedt, Frank L{\"u}beck, Jean Michel, 
Britta Sp{\"a}th, Andrzej Szczepa{\'n}ski and Jay Taylor for numerous helpful 
hints and enlightening conversations.

\end{document}